\documentclass{amsart}
\usepackage{fullpage}
\usepackage{blkarray}
\usepackage{array}
\usepackage{mathrsfs}  
\usepackage{amsmath}
\usepackage{amssymb}
\usepackage[mathscr]{euscript}
\usepackage{mathtools}
\usepackage{tensor}
\usepackage[all]{xy}

\newcommand{\xyR}[1]{\xydef@\xymatrixrowsep@{#1}}
\newcommand{\xyC}[1]{\xydef@\xymatrixcolsep@{#1}}

	\usepackage{tikz}
	\usepackage{tikz-cd}

	\usetikzlibrary{matrix,arrows,decorations.pathmorphing,positioning,decorations.pathreplacing}
	\tikzset{commutative diagrams/.cd, 
		mysymbol/.style = {start anchor=center, end anchor = center, draw = none}}

\newtheorem{theorem}{Theorem}[section]
\newtheorem{lemma}[theorem]{Lemma}
\newtheorem{corollary}[theorem]{Corollary}
\newtheorem{proposition}[theorem]{Proposition}

\theoremstyle{definition}
\newtheorem{definition}[theorem]{Definition}
\newtheorem{example}[theorem]{Example}
\newtheorem{notation}[theorem]{Notation}
\newtheorem{assumption}[theorem]{Assumption}

\theoremstyle{remark}
\newtheorem{remark}[theorem]{Remark}

\numberwithin{equation}{section}
	\newenvironment{acknowledgements}{
	\renewcommand\abstractname{Acknowledgements}	\begin{abstract}} {\end{abstract}}

\begin{document}
\title[Classification of modules for complete gentle algebras]{Classification of modules for complete gentle algebras}
    \author{Raphael Bennett-Tennenhaus}
\address{  R. Bennett-Tennenhaus\\
Faculty of Mathematics\\
Bielefeld University\\
Universit{\"{a}}t sstra{\ss}e 25\\
33615 Bielefeld\\
 Germany}
\email{raphaelbennetttennenhaus@gmail.com}


\begin{abstract} We classify finitely generated modules over a class of algebras introduced in the authors' Ph.D thesis, called complete gentle algebras. These rings generalise the finite-dimensional gentle algebras introduced by Assem and Skowro\'{n}ski, in such a way so that the ground field is replaced by any complete local noetherian ring. Our classification is written in terms of string and band modules. For the proof we apply the main result from the authors' thesis, which classifies complexes of projective modules with finitely generated homogeneous components up to homotopy. In doing so we construct the resolution of a string or band module, generalising some calculations due to \c{C}anak\c{c}{\i}, Pauksztello and Schroll.
\end{abstract}

\maketitle

\section{Introduction.}
Classifying the representations of a given finite-dimensional algebra has been a central goal in representation theory for the last 50 years. With this perspective Drozd's tame-wild dichotomy is fundamental. This says that if the representation-type of a finite-dimensional algebra is infinite, then it is either tame or wild, and not both. Since Drozd's dichotomy, a key pursuit in the representation theory of finite-dimensional algebras has been finding algebras with finite or tame representation type, and classifying their representations. It is well known that this pursuit reduces to considering finite-dimensional modules over quotients of (basic, connected) path algebras by admissible ideals. 

Skowro\'{n}ski and Waschb\"{u}sch \cite{SkoWas1983} showed that if an algebra is representation-finite and biserial in the sense of Fuller \cite{Ful1978}, then it is special biserial; meaning that the quiver and relations satisfy certain combinatorial constraints. Wald and Waschb\"{u}sch \cite{WalWas1985} then showed that: any special biserial algebra has finite or tame representation-type; any indecomposable finite-dimensional module over such an algebra is either a non-uniserial projective-injective, or (what we call) a \emph{string module} or a \emph{band module}; and the Auslander-Reiten (AR) sequences may be described explicitly. Butler and Ringel \cite{ButRin1987} considered string algebras, which are monomial special biserial algebras, and hence any projective-injective indecomposable module over a string algebra is uniserial. Assem and Skowro\'{n}ski \cite{AssSko1986} also introduced \emph{gentle} algebras, defined as special biserial algebras satisfying certain combinatorial constraints on their AR quiver. In particular, gentle algebras are quadratic string algebras (but not every quadratic string algebra is gentle).

The use of triangulated categories in the representation theory of finite-dimensional algebras has been beautifully documented in the well-known book of Happel \cite{Happ1988}. Here a triangulated fully-faithful functor is defined from the bounded derived category of a finite-dimensional algebra to the stable module category of its repetitive algebra -- which is dense provided the algebra has finite global dimension. Bekkert and Drozd \cite{BekDro2003} since gave a tame-wild dichotomy for the derived representation-type of a finite dimensional algebra. Studying homological properties of gentle algebras has recently been a topic of popular interest. 

Ringel \cite{Rin1995} showed that the repetitive algebra of a gentle algebra is special biserial. Bobi\'{n}ski \cite{Bob2011} then combined this with the aforementioned description of AR sequences in \cite{WalWas1985} to study perfect complexes and AR triangles in the derived category of a gentle algebra (by means of Happel's functor). Bekkert and Merklen \cite{BekMer2003} then described all indecomposable objects in the bounded derived category in terms of (what we call) \emph{string complexes} and \emph{band complexes}; and Arnesen, Laking and Pauksztello \cite{ArnLakPau2016} gave a basis for the space of morphisms between these complexes. \c{C}anak\c{c}{\i}, Pauksztello and Schroll described a basis for the extension group between string and band modules over a gentle algebra by calculating the homology of any string or band complex, and hence combinatorially realising string and band complexes as projective resolutions of these modules. Considering string and band modules by means of projective presentations or resolutions was an idea used by Huisgen-Zimmerman and Smal{\o} \cite{HuiSma2005}, who used this perspective to study the global dimensional of a string algebra.

In the author's Ph.D thesis \cite{Ben2018} (see also \cite{Ben2016}) we introduced a class of rings, which we call \emph{complete gentle algebras}, and which generalise the gentle algebras introduced in \cite{AssSko1986}. The definition of these rings involves substituting the ground field of a (finite-dimensional) gentle algebra with a complete, local noetherian ring $R$. The main classification achieved in \cite{Ben2018} gave a description of the objects in the homotopy category of complexes of finitely generated projective modules over any complete gentle $R$-algebra; see Theorems \ref{theorem:from-thesis-I}, \ref{theorem:from-thesis-II} and \ref{theorem:from-thesis-III}. Our main results in this article are Theorems \ref{theorem:main-fin-gen-modules-are-sums-of-strings-and-bands}, \ref{theorem:main-isomorphisms-between-string-and-band-modules}, \ref{theorem:main-krull-remak-schmidt-property} and \ref{theorem:resolutions}. The first three of these give a classification of the finitely generated modules for any complete gentle $R$-algebra. The proof uses a direct application of the above classification from \cite{Ben2018}. Before stating our main theorems we loosely explain and cross reference the notation and terminology we use.
\begin{itemize}
    \item Fix a commutative, noetherian and local ring $(R,\mathfrak{m},k)$ where $R$ is $\mathfrak{m}$-adically complete.
    \item Let $R[T,T^{-1}]$ be the ring of Laurent polynomials with coefficients in the ring $R$.
    \item Let $R[T,T^{-1}]\text{-}\boldsymbol{\mathrm{Mod}}_{R\text{-}\boldsymbol{\mathrm{Proj}}}$ be the category of $R[T,T^{-1}]$-modules which are free over $R$.
    \item Let $\mathrm{res}$ be the involution of $R[T,T^{-1}]\text{-}\boldsymbol{\mathrm{Mod}}_{R\text{-}\boldsymbol{\mathrm{Proj}}}$ which swaps the action of $T^{-1}$ and $T$. 
    \item Complete gentle $R$-algebras $\Lambda$ are defined in Definition \ref{definition:complete-gentle-algebra} by a quiver $Q$ (with relations).
    \item Words $C$, together with their inverses $C^{-1}$ and their shifts $C[n]$ ($n\in\mathbb{Z}$) are written in an alphabet given by the arrows in $Q$ (avoiding the above mentioned relations), and are defined in Definition \ref{definition:words-for-modules}.
    \item String words are words of a certain form, defined in Definition \ref{definition:string-words}. These words index the string modules $M(C)$ defined in Definition \ref{definition:string-modules-by-string-words}.
    \item Band words are words of a different particular form, defined in Definition \ref{definition:band-words}. These words index the band modules $M(C,V)$ defined in Definition \ref{definition:band-module} for any object $V$ of $R[T,T^{-1}]\text{-}\boldsymbol{\mathrm{Mod}}_{R\text{-}\boldsymbol{\mathrm{Proj}}}$.
\end{itemize}
To explain this terminology we provide an example of a string module and an example of a band module. Throughout the article we refer back to these examples in order to explain our notation.
\begin{example}\label{example:intro-string-module} Let $\Lambda=k[[x,y]]/(xy,yx)$, where $k[[x,y]]$ is the power series ring in two variables. This is a particular instance of a large class of complete gentle algebras over $R=k[[t]]$ discussed in Example \ref{example:path-complete-gentles}. Note that $t$ acts on $\Lambda$ by the sum $x+y$. Here $Q$ is the quiver consisting of two loops $x$ and $y$ at a single vertex $v$, and  $\mathscr{I}$ is the ideal in $RQ$ generated by $xy$ and $yx$. As in Example \ref{example:path-complete-gentles} we have that $\Lambda\simeq\overline{kQ}/\overline{\langle xy,yx\rangle}$ where $\overline{kQ}$ is the completion of the path algebra $kQ$ with respect to the arrow ideal.

A word $C$ here is a sequence of letters $C_{i}$ in $\{x,y,x^{-1},y^{-1}\}$ such that $C$ does not contain $xy$, $yx$, $x^{-1}y^{-1}$ or $y^{-1}x^{-1}$ as a sub-sequence of consecutive letters. For example,
\[
\begin{array}{c}
     C=x^{3}y^{-3}x^{2}y^{-1}x^{2}y^{-3}x^{4}y^{-1}y^{-1}y^{-1}y^{-1}\dots
\end{array}
\]
is an example of an $\mathbb{N}$-word where $C_{i}=x$ for $i=1,2,3,7,8,10,11,15,16,17$, and $C_{i}=y^{-1}$ otherwise. In fact, $C$ is a string word, since we have $C=B^{-1}(AD)$ where $B=x^{-3}$ and $D=y^{-1}y^{-1}\dots$ are inverse and the word $A=y^{-3}x^{2}y^{-1}x^{2}y^{-3}x^{4}$ is alternating in the sense of Definition \ref{definition:finitely-generated-words-II}. Using this decomposition $(B,A,D)$ of $C$, the string module $M(C)$ is defined as the quotient of a free (more generally, projective) module $N(C)=\oplus_{i=0}^{3}\Lambda g_{C,i}$ subject to the relations $x^{4}g_{C,0}=0$, $y^{3}g_{C,0}=x^{2}g_{C,1}$, $yg_{C,1}=x^{2}g_{C,2}$ and $y^{3}g_{C,2}=x^{4}g_{C,3}$. Hence we may depict the generators and relations of $M(C)$ as follows.
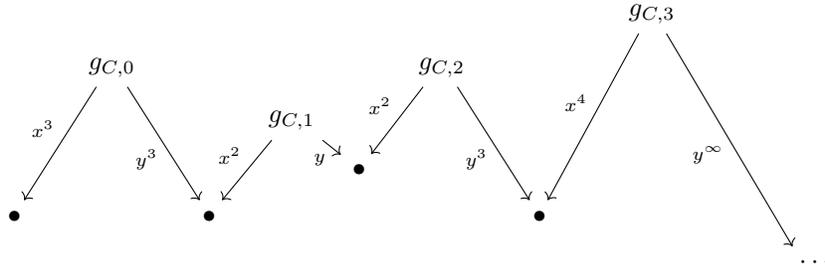
\[
\begin{tikzcd}[column sep=0.2cm, row sep=0.2cm]
 & & & & & & & & & &  & & & & & & & & & g_{C,3}\arrow[swap]{ddddllll}{x^{4}}\arrow[swap]{dddddrrrrr}{y^{\infty}} & & & & & &\\&
 & & & g_{C,0}\arrow[swap]{dddrrr}{y^{3}}\arrow[swap]{dddlll}{x^{3}} & & & & & & & & g_{C,2}\arrow[swap]{ddll}{x^{2}}\arrow[swap]{dddrrr}{y^{3}} & & & & & & & & & & & &\\&
 & & & &  & &  & & g_{C,1}\arrow[swap]{dr}{y}\arrow[swap]{ddll}{x^{2}} & & & &  & & & & & & & & & & & &\\&
  & & & &  & & & & & \bullet & & &  & & & & & & & & & & & &\\&\bullet
 & & & &  & & \bullet & & & & & &  & & \bullet &  & & & & & & & & &\\&
 & & & &  & & & & & & & &  & & & & & & & & & & & \quad\cdots  &
\end{tikzcd}
\]
Note that one indexes the generators $g_{C,i}$ using the $C$-peaks in $\mathbb{N}$, in the sense of Definition \ref{definition:finitely-generated-words-II}. The inverse $E=C^{-1}$ of $C$ is the $-\mathbb{N}=\{\dots,-2,-1,0\}$-word $E=\dots y y yx^{-4}y^{3}x^{-2}yx^{-2}y^{3}x^{-3}$. Repeating the procedure above to define $M(E)$ gives a diagram symmetric to the diagram above. Consequently there is a canonical isomorphism $M(C)\to M(E)$ sending the coset of $g_{C,i}$ to that of $g_{E,-i}$.
\end{example}

\begin{example}\label{example:intro-band-module}
Let $p$ be prime and $R=\widehat{\mathbb{Z}}_{p}$, the ring of $p$-adic integers. Consider the $R$-algebra
\[
\Lambda = \left\{\begin{pmatrix}r_{11} & r_{12}\\
r_{21} & r_{22}
\end{pmatrix}\in\mathbb{M}_{2}(\widehat{\mathbb{Z}}_{p})\,\colon\,r_{11}-r_{22},\,r_{12}\in p\widehat{\mathbb{Z}}_{p}\,\right\}.
\]
As in Example \ref{example:comp-gentle-2-by-2-matrices}, $\Lambda$ is a complete gentle $R$-algebra where the quiver $Q$ consists of two loops $ a$ and $ b$ at a vertex $v$, and where one considers the ideal $\mathscr{I}=\langle a^{2}, b^{2}\rangle$ of $RQ$. In fact, we shall see that $\Lambda \simeq RQ/\langle a^{2}, b^{2}, a b+ b a-p\rangle$ as $R$-algebras. In this case, a word $C$ is written in the alphabet $\{a,b,a^{-1},b^{-1}\}$ subject to the constraint that neither $aa$, $bb$, $a^{-1}a^{-1}$ nor $b^{-1}b^{-1}$ arises as a sub-sequence of consecutive letters. For an example of a band word, take the $\mathbb{Z}$-word $C={}^{\infty\hspace{-0.5ex}}E^{\infty}=\dots E\mid E E \dots$ where
\[
E= a^{-1}b^{-1} a^{-1} b a b^{-1} a^{-1} b^{-1} a b^{-1} a b a.
\]
Here $E$ is an alternating $\{0,\dots,12\}$-word that is cyclic in the sense of Definition \ref{definition:cyclic-words}. As in Definition \ref{definition:periodicstring-modules}, and similarly to Example \ref{example:intro-string-module}, one can construct a module $M(C)$ defined using generators $g_{C,i}$ and relations determined by $C$. In contrast to Example \ref{example:intro-string-module}, there are infinitely many $C$-peaks (in $\mathbb{Z}$) here, since $C$ is periodic and non-primitive, in the sense of Definition \ref{definition:inverses-shifts-words}. Since $C$ is $12$-periodic, by Lemma \ref{lemma:periodic-string-module-has-T-automorphism} we have that $M(C)$ is a right $R[T,T^{-1}]$-module where the action of $T$ is depicted as follows.
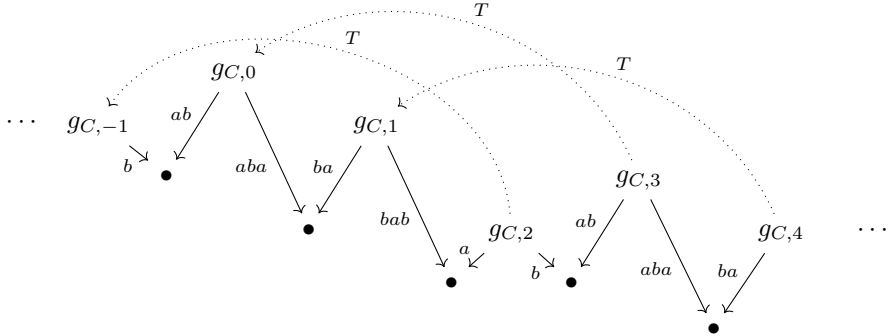
\[
\begin{tikzcd}[column sep=0.1
cm, row sep=0.2cm]
 & & & & & & & & & &  & & & & & & & & & & & & & & &\\&
 & & & g_{C,0}\arrow[swap]{dddrrr}{ a b a}\arrow[swap]{ddll}{ a b} & & & & & & & &  & & & & & & & & & & & &\\\cdots  & g_{C,-1}\arrow[swap]{dr}{ b}
 & & & &  & &  & & g_{C,1}\arrow[swap]{dddrrr}{ b a b}\arrow[swap]{ddll}{ b a} & & & &  & & & & & & & & & & & &\\&
 &\bullet & & &  & &  & & & & & &  & & & g_{C,3}\arrow[dotted,bend right=50,swap]{uullllllllllll}{T}\arrow[swap]{ddll}{ a b}\arrow[swap]{dddrrr}{ a b a} &  &  & & & & & & &\\&
 & & & &  & & \bullet & & & & & & g_{C,2}\arrow[dotted,bend right=70,swap]{uullllllllllll}{T}\arrow[swap]{dl}{ a}\arrow[swap]{dr}{ b}  & &  &  & & & & & g_{C,4}\arrow[dotted,bend right=50,swap]{uullllllllllll}{T}\arrow[swap]{ddll}{ b a} & & & &\cdots\\&
  & & & &  & & & & &  & & \bullet &  & \bullet & & & & & & &  & & & & \\&
  & & & &  & & & & &  & &  &  &  & & & & & \bullet &  & & & & &
\end{tikzcd}
\]
The band module associated to the word $C$ and some object $V$ of $R[T,T^{-1}]\text{-}\boldsymbol{\mathrm{Mod}}_{R\text{-}\boldsymbol{\mathrm{Proj}}}$ is defined and denoted $M(C,V)=M(C)\otimes _{R[T,T^{-1}]}V$. The shifts $C[n]$ of $C$ are defined for each $n\in\mathbb{Z}$ so that, for example,
\[
C[1]=\dots EE a^{-1}\mid b^{-1} a^{-1} b a b^{-1} a^{-1} b^{-1} a b^{-1} a b a EE\dots 
\]
Hence $C[5]$ and $C[9]$ are band words, and as in Example \ref{example:intro-string-module}, there are canonical isomorphisms $M(C,V)\simeq M(C[5],V)\simeq M(C[9],V)$. Here $k$ is the finite field $\mathbb{F}_{p}$ with $p$ elements. We shall also see, by taking projective resolutions, that if $k\otimes _{R}V\simeq k\otimes_{R}W$ as $k[T,T^{-1}]$-modules then $M(C,V)\simeq M(C,W)$. Here, inverting the word has the effect of swapping $T$ and $T^{-1}$ via an isomorphism $M(C,V)\simeq M(C^{-1},\mathrm{res}\, V)$; see Lemma \ref{lemma:equivalence-relation-on-words-bands}.
\end{example}
We now state the main theorems in this article.
\begin{theorem}
\label{theorem:main-fin-gen-modules-are-sums-of-strings-and-bands}  Let $\Lambda$ be a complete gentle algebra.
\begin{enumerate}
\item Every finitely generated $\Lambda$-module is isomorphic to a direct sum of string and band modules.
\item Every finitely generated string or band module is indecomposable.
\end{enumerate}
\end{theorem}
\begin{theorem}
\label{theorem:main-isomorphisms-between-string-and-band-modules}Let $\Lambda$ be a complete gentle algebra, $C,E$ be words and $V,W$ lie in $R[T,T^{-1}]\text{-}\boldsymbol{\mathrm{Mod}}_{R\text{-}\boldsymbol{\mathrm{Proj}}}$.
\begin{enumerate}
\item If $C,E$ are string words then $M(C)\simeq M(E)$ if and only if $E=C[n]$ or $E=C^{-1}[n]$ for some $n$.
\item If $C,E$ are band words then $M(C,V)\simeq M(E,W)$ if and only if \emph{(}$E=C[n]$, $k\otimes_{R}V\simeq k\otimes_{R}W$\emph{)} or \emph{(}$E=C^{-1}[n]$, $k\otimes_{R}V\simeq k\otimes_{R}\mathrm{res} \,W$\emph{)} for some $n$.
\item If $C$ is a string word and $E$ is a band word then $M(C)\not\simeq M(E,V)$.
\end{enumerate}
\end{theorem}
\begin{theorem}
\label{theorem:main-krull-remak-schmidt-property} Let $\Lambda$ be a complete gentle algebra. If a finitely generated $\Lambda$-module may be written as a direct sum of finitely generated string and band modules in two different ways, then there is an isoclass
preserving bijection between the summands.
\end{theorem}
As a consequence of the methods we use, we also obtain the following result. 
\begin{theorem}\label{theorem:resolutions}\emph{(see \cite[\S2, Corollary 3.13]{CanPauSch2021})} Let $\Lambda$ be a complete gentle algebra. Let $C$ be a string \emph{(}respectively, band\emph{)} word. Then there is algorithm for calculating a projective resolution of the string module $M(C)$ \emph{(}respectively, band module $M(C,V)$\emph{)}. In particular, any band module has projective dimension $1$.
\end{theorem}
The article is organised as follows. In \S\ref{section:complete-gentle-algebras} we recall the definition of a complete gentle $R$-algebra from \cite{Ben2016}, and give some examples and properties. In \S\ref{section:string-and-band-modules} we define string and band modules over these algebras. In \S\ref{section:string-and-band-complexes} we recall the definition of \emph{string complexes} and \emph{band complexes} from \cite{Ben2016}. In \S\ref{section:Resolutions-and-string-and-band-complexes} we characterise the string and band complexes which arise as projective resolutions. In \S\ref{section:words-and-corresponding-generalised-words} we give the algorithm referred to in the statement of Theorem \ref{theorem:resolutions}. In \S\ref{section:proofs-of-main} we provide proofs for Theorems \ref{theorem:main-fin-gen-modules-are-sums-of-strings-and-bands}, \ref{theorem:main-isomorphisms-between-string-and-band-modules}, \ref{theorem:main-krull-remak-schmidt-property} and \ref{theorem:resolutions}
\section{Complete gentle algebras}\label{section:complete-gentle-algebras}
\begin{assumption}\label{assumption:first-assumptions}In all that remains we assume we have fixed: 
\begin{enumerate}
\item a commutative noetherian complete local ring $(R,\mathfrak{m},k)$;
\item a finite quiver $Q$ with path algebra $RQ$ and an ideal $\mathscr{I}\triangleleft RQ$ generated by a set of length $2$ paths; 
\item and a surjective $R$-algebra homomorphism $\vartheta\colon RQ\rightarrow\Lambda$ satisfying $\mathscr{I}\subseteq\ker(\vartheta)$ and $\vartheta(\gamma)\neq0$ for any path $\gamma\notin\mathscr{I}$. Notation is abused by writing $\gamma$ for $\vartheta(\gamma)$.
\end{enumerate}
\end{assumption}
Conditions (2) and (3) from Assumption \ref{assumption:first-assumptions} imply that $\vartheta(e_{v})\neq0\neq\vartheta(x)$ for any vertex $v$ and arrow $x$.
\begin{notation}\label{notation:paths-P-arrows-A} Let $\mathbf{P}$ be the set of non-trivial paths $\gamma\notin \mathscr{I}$ with head $h(\gamma)$ and tail $t(\gamma)$. For each $l>0$ and each vertex $v$, let $\mathbf{P}(l,v\rightarrow)$
(respectively $\mathbf{P}(l,\rightarrow v)$) be the set of paths $\gamma\in\mathbf{P}$ of length $l$ with $t(\gamma)=v$ (respectively $h(\gamma)=v$). Let $\mathbf{A}$ be the set of arrows in $Q$, $\mathbf{A}(v\rightarrow)=\mathbf{P}(1,v\rightarrow)$ and  $\mathbf{A}(\rightarrow v)=\mathbf{P}(1,\rightarrow v)$. The composition of $y\in\mathbf{A}(\rightarrow v)$ and $x\in\mathbf{A}(u\rightarrow)$ is $xy$ if $u=v$, and $0$ otherwise. 
\end{notation}
For the most part, elements of $\mathbf{A}$ are denoted by lower case letters of the English alphabet, and elements of $\mathbf{P}$ will be denoted by lower case letters of the Greek alphabet.
\begin{definition}\label{definition:complete-gentle-algebra}
\cite[Definitions 2.3, 2.5, 2.12]{Ben2016} We say  $\Lambda$ is a \emph{complete gentle $R$-algebra} if the following holds.
\begin{enumerate}
    \item For any vertex $v$ (of $Q$) we have $\vert\mathbf{A}(v\rightarrow)\vert\leq2\geq\vert\mathbf{A}(\rightarrow v)\vert$.
\item For any arrow $y$ we have 
\[
\vert\{ x\in \mathbf{A}(h(y)\rightarrow)\colon xy\in\mathbf{P}\}\vert\leq 1\geq\vert\{ z\in \mathbf{A}(\rightarrow t(y))\colon yz\in\mathbf{P}\}\vert.
\]
\item For any arrow $y$ we have 
\[
\vert\{ x\in \mathbf{A}(h(y)\rightarrow)\colon xy\notin\mathbf{P}\}\vert\leq 1\geq \vert\{ z\in \mathbf{A}(\rightarrow t(y))\colon yz\notin\mathbf{P}\}\vert.
\]
\item For each vertex $v$ the $\Lambda$-module $\Lambda e_{v}$ (respectively  $e_{v}\Lambda$) has a unique maximal submodule given by
\[\begin{array}{cc}
\mathrm{rad}(\Lambda e_{v})=\sum_{x\in\mathbf{A}(v\rightarrow)}\Lambda x & (\text{respectively, }\mathrm{rad}(e_{v}\Lambda )=\sum_{x\in\mathbf{A}(\rightarrow v)}x\Lambda)
\end{array}
\]
\item For any $x\in\mathbf{A}$ the $R$-modules $\Lambda x$ and $x\Lambda$ are finitely generated.
\item For any distinct $x,y\in\mathbf{A}$ we have $\Lambda x\cap \Lambda y=0$ and $x\Lambda \cap y\Lambda=0$.
\end{enumerate}
\end{definition}
By \cite[Corollary 1.2.11]{Ben2018} a complete gentle algebra over a field is the same thing as a finite-dimensional gentle algebra as introduced by Assem and Skowro\'{n}ski \cite{AssSko1986}. For Remark \ref{remark:props-of-complete-gen-algs} we recall some notation.
\begin{definition}\label{definition:firstlastarrows}\cite[Notation 2.10]{Ben2016} Let $\Lambda$ be a complete gentle algebra. By condition (6) of Definition \ref{definition:complete-gentle-algebra} any $\gamma\in\mathbf{P}$ has a unique \emph{first arrow} $\mathrm{f}(\gamma)\in\mathbf{A}$ satisfying $\eta\mathrm{f}(\gamma)=\gamma$ for some (possibly trivial) path $\eta\notin\mathscr{I}$. Likewise $\gamma\in\mathbf{P}$ has a unique \emph{last arrow} $\mathrm{l}(\gamma)\in\mathbf{A}$ satisfying $\mathrm{l}(\gamma)\mu=\gamma$
for a path $\mu\notin\mathscr{I}$.
\end{definition}
\begin{remark}\label{remark:props-of-complete-gen-algs} Here we detail consequences of each condition from Definition \ref{definition:complete-gentle-algebra}. Item (1) below is the statement of \cite[Lemma 1.1.14]{Ben2018}. Item (2) is essentially the statement and proof of \cite[Corollary 1.1.25]{Ben2018}.
\begin{enumerate}
\item Suppose conditions (1) and (2) of Definition \ref{definition:complete-gentle-algebra} hold. For any  $\gamma,\mu\in\mathbf{P}$ where $\gamma$ is not shorter than $\mu$, if $\mathrm{f}(\gamma)=\mathrm{f}(\mu)$ (respectively  $\mathrm{l}(\gamma)=\mathrm{l}(\mu)$) then $\gamma=\eta\mu$ (respectively $\gamma=\mu\eta$) for some path $\eta\notin\mathscr{I}$.
\item Suppose condition (4) of Definition \ref{definition:complete-gentle-algebra} holds. Note that this means the set of $e_{v}$ (as $v$ runs through the vertices in $Q$) is a complete set of local idempotents. By the surjectivity of $\vartheta$ in condition (3) of Assumption \ref{assumption:first-assumptions}, any element of $\Lambda e_{v}/\mathrm{rad}(\Lambda e_{v})$ is a coset of the form $re_{v}+\sum_{x\in\mathbf{A}(v\rightarrow)}\Lambda x$ for some $r\in R$. So, in this case there is a surjective $R$-module homomorphism $R\to\Lambda e_{v}/\mathrm{rad}(\Lambda e_{v})$ given by sending $r$ to this coset. This shows that the quotient $\Lambda/\mathrm{rad}(\Lambda)$ is module finite over $R$.
\item Suppose conditions (4) and (5) from Definition \ref{definition:complete-gentle-algebra} hold. By combining these conditions with (2) above, the first and last terms of the exact sequence $0\to\mathrm{rad}(\Lambda)\to\Lambda\to \Lambda/\mathrm{rad}(\Lambda)\to 0$ are finitely generated $R$-modules. Hence $\Lambda$ is module finite over $R$. In particular, we have the following.
\begin{enumerate}
    \item By \cite[Proposition 20.6]{Lam1991}, the ring $\Lambda$ is (both noetherian, and) semilocal; and $(\mathrm{rad}(\Lambda))^{n}\subseteq\Lambda \mathfrak{m}\subseteq \mathrm{rad}(\Lambda)$ for some integer $n\geq1$. Hence the $R$-module epimorphism from (2) above give isomorphisms of $R$-modules
    \[
    \Lambda e_{v}/\mathrm{rad}(\Lambda)e_{v}\simeq k\simeq e_{v}\Lambda/e_{v}\mathrm{rad}(\Lambda).
    \]
    \item Consequently, and by Krull's intersection theorem \cite[Ex. 4.23]{Lam1991},  
\[
\begin{array}{c}
\bigcap_{n>0}(\mathrm{rad}(\Lambda))^{n}M\subseteq \bigcap_{n>0}\mathfrak{m}^{n}M=0
\end{array}
\]
for any $\Lambda$-module $M$ which is module-finite over $R$. In particular, the above holds for any $M$ which is module finite over $\Lambda$.
\end{enumerate}
\end{enumerate}
\end{remark}
\begin{remark}\label{remark:propsofqbspecialbiserials2}
We now compare Definition \ref{definition:complete-gentle-algebra} with \cite[Definition 2.12]{Ben2016}.
\begin{enumerate}
    \item Assumption \ref{assumption:first-assumptions} is exactly \cite[Assumption 2.1]{Ben2016} together with the additional assumptions that $R$ is $\mathfrak{m}$-adically complete and that $\mathscr{I}$ is quadratic (that is, generated by paths of length equal to $2$). Assuming \cite[Assumption 2.1]{Ben2016} holds, that $\mathscr{I}$ is quadratic is equivalent to condition GI of \cite[Definition 2.3]{Ben2016}. Conditions SPI, SPII and GII of \cite[Definition 2.3]{Ben2016} are exactly conditions (1), (2) and (3) (of Definition \ref{definition:complete-gentle-algebra}). Likewise conditions (iii), (iv) and (v) from \cite[Definition 2.5]{Ben2016} are exactly conditions (4), (5) and (6).
    \item Note that the condition (ii) from \cite[Definition 2.5]{Ben2016} is the statement that $(\mathrm{rad}(\Lambda))^{n}\subseteq\Lambda \mathfrak{m}\subseteq \mathrm{rad}(\Lambda)$ for some $n\geq 1$. Thus, by Remark \ref{remark:props-of-complete-gen-algs}(3a), condition (ii) from \cite[Definition 2.5]{Ben2016} holds if both conditions (3) and (4) from Definition \ref{definition:complete-gentle-algebra} hold.
\item By combining the above we can show that Definition \ref{definition:complete-gentle-algebra} and \cite[Definition 2.12]{Ben2016} are equivalent.
\end{enumerate}
In Definition \ref{definition:with-no-or-full-relations} we recall terminology used by Bessenrodt and Holm \cite{BesHol2008}.
\begin{definition}\cite[\S 1.1]{BesHol2008}\label{definition:with-no-or-full-relations} Let $\Lambda$ be a complete gentle algebra. Let $\gamma=x_{n}\dots x_{1}$ be an (oriented) \emph{cycle in} $Q$, meaning that $h(x_{n})=t(x_{1})$ and $h(x_{i})=t(x_{i+1})$ when $i<n$. We call $\gamma$ a \emph{cycle with no relations} if $x_{1}x_{n}\notin\mathscr{I}$ and $x_{i+1}x_{i}\notin\mathscr{I}$ when $i<n$. We call a cycle $\gamma$ with no relations \emph{primitive} if $\gamma$ is not the power of another cycle with no relations.
\end{definition}
By Assumption \ref{assumption:first-assumptions} the orientated cycles in $\mathbf{P}$ are the cycles with no relations. 
\begin{example}\label{example:path-complete-gentles}
Let $k[[t]]$ be the ring of power series in one variable $t$ over the field $k$. Here we recall \cite[Remark 2.6]{Ben2016}, in which examples of complete gentle $k[[t]]$-algebras are given, which arise as examples of \emph{completed string algebras} in the sense of the Ph.D thesis of Ricke \cite{Ric2017}. Let $R=k[[t]]$ and let $Q$ be a quiver and $\mathscr{I}\triangleleft RQ$ be an ideal generated by length $2$ paths where: any vertex of $Q$ is the head of at most two arrows, and the tail of at most two arrows; any for any arrow $y$ there is at most one arrow $x$ with $t(x)=h(y)$ and $xy\in\mathscr{I}$ (respectively $z$ with $h(z)=t(y)$ and $yz\in\mathscr{I}$) and at most one arrow $x'$ with $t(x')=h(y)$ and $x'y\notin\mathscr{I}$ (respectively $z'$ with $h(z')=t(y)$ and $yz'\notin\mathscr{I}$). Let $\Lambda$ be the quotient of the completion $\overline{kQ}$ of the path algebra $kQ$ modulo the completion $\overline{\mathscr{I}}$ of $\mathscr{I}$. Then by construction $\Lambda$ is a completed string algebra in the sense of \cite{Ric2017}. By letting $t$ act by the sum of the primitive cycles in $Q$, as in \cite[Remark 2.6]{Ben2016}, $\Lambda$ is a complete gentle $k[[t]]$-algebra; see \cite[Corollary 1.2.29]{Ben2018}. In Remark \ref{remark:comparing-terminology-used-by-Ricke} we explain how our main results are consistent with \cite[Theorem 3.3.48]{Ric2017}, since the modules in each of these classifications are indexed by the same words.
\end{example}
\end{remark}
Lemma \ref{lemma:qbsb-uniserial-cyclics-given-by-arrows} recalls a useful consequence of Definition \ref{definition:complete-gentle-algebra}.  We sketch the proof, as details are given in \cite{Ben2018}. 
\begin{lemma}\label{lemma:qbsb-uniserial-cyclics-given-by-arrows} \emph{(see \cite[Remark 2.9]{Ben2016}, \cite[Corollary 1.1.17]{Ben2018}).} Let $\Lambda$ be a complete gentle $R$-algebra and let $\gamma\in\mathbf{P}$.
\begin{enumerate}
\item Any non-trivial submodule of $\Lambda \gamma$ \emph{(}respectively, $\gamma\Lambda$\emph{)}
has the form $\Lambda \mu\gamma$ \emph{(}respectively, $\gamma\mu\Lambda$\emph{)} with $\mu\in\mathbf{P}$. 
\item The $\Lambda$-module $\Lambda \gamma$ \emph{(}respectively, $\gamma\Lambda$\emph{)} is local with radical $\mathrm{rad}(\Lambda)\gamma$ \emph{(}respectively, $\gamma\mathrm{rad}(\Lambda)$\emph{)} which is either $0$ or has the form $\Lambda x\gamma$ with $x\gamma\in\mathbf{P}$ \emph{(}respectively, $\gamma x\Lambda $ with $\gamma x\in\mathbf{P}$\emph{)} for some $x\in\mathbf{A}$.
\item For all $\mu\in\mathbf{P}$ with $\Lambda \gamma=\Lambda \mu$ and $\mathrm{f}(\gamma)=\mathrm{f}(\mu)$ \emph{(}respectively, $\gamma\Lambda=\mu\Lambda$ and $\mathrm{l}(\gamma)=\mathrm{l}(\mu)$\emph{)}
we have $\gamma=\mu$. 
\end{enumerate}
\end{lemma}
\begin{proof}By Remark \ref{remark:props-of-complete-gen-algs}(3a), for some integer $n\geq 1$ we have $(\mathrm{rad}(\Lambda))^{n}\subseteq\Lambda \mathfrak{m}\subseteq \mathrm{rad}(\Lambda)$. By Remark \ref{remark:props-of-complete-gen-algs}(3b) we have that $\bigcap_{n>0}(\mathrm{rad}(\Lambda))^{n}x=\bigcap_{n>0}x(\mathrm{rad}(\Lambda))^{n}=0$ for any arrow $x$. The proofs of the respective statements about right modules are similar, and omitted. 

(1) Suppose $M\neq 0$ is a submodule of $\Lambda \gamma$. Since the intersections over $n$ above are trivial the set of all paths $\sigma$ such that ($\sigma\gamma\in\mathbf{P}$
and $M\subseteq\Lambda \sigma\gamma$) is finite, as $M\neq0$. Let $\mu$ be the
longest such $\sigma$. For a contradiction assume that $M\neq\Lambda \mu\gamma$. Consider the submodule $L$ of $\Lambda e_{h(\mu)}$ consisting of
all $\lambda\in\Lambda e_{h(\mu)}$ with $\lambda \mu\gamma\in M$. Since $M\neq\Lambda \mu\gamma$
we have $L\neq\Lambda e_{h(\mu)}$. By condition (3) of Definition \ref{definition:complete-gentle-algebra} we have $L\subseteq\sum\Lambda x$ where the sum runs over $x\in\mathbf{A}(h(\mu)\rightarrow)$. 

Let $m\in M\subseteq\Lambda \mu\gamma$, and suppose $m\neq0$. Write $m=\lambda \mu\gamma$ for some $\lambda\in\Lambda$. By construction $0\neq\lambda\in L$, and by conditions (1) and (2) from Definition \ref{definition:complete-gentle-algebra}, without loss of generality we can choose $x,y\in\mathbf{A}(h(\mu)\rightarrow)$ such that $\lambda=\lambda' x+\lambda'' y$ and $y\mu\gamma=0$ in $\Lambda$. Altogether this shows any $m\in M$ lies in $\Lambda x\mu\gamma$, and so $M\subseteq \Lambda x\mu\gamma$. Since $M\neq 0$ this contradicts the maximality of $\mu$. 

(2) It suffices to note that $\Lambda$ is semilocal by Remark \ref{remark:props-of-complete-gen-algs}(3a) (see for example \cite[Proposition 24.4]{Lam1991}), and combine this with conditions (1), (2) and (3) from Definition \ref{definition:complete-gentle-algebra}.

(3) Suppose $\Lambda \gamma=\Lambda \mu\neq0$ and $\mathrm{f}(\gamma)=\mathrm{f}(\mu)$. By Remark \ref{remark:props-of-complete-gen-algs}(1) we have $\mu=\sigma \gamma$ for some path
$\sigma$. For a contradiction we assume $\sigma$ is non-trivial, and so $\sigma\in\mathbf{P}$. Consider the map $\tau\colon\Lambda e_{h(\gamma)}\rightarrow\Lambda \gamma$ sending
$\lambda$ to $\lambda \gamma$. Since $\Lambda \gamma=\Lambda \sigma\gamma$ we have $\gamma\in \Lambda \sigma\gamma$, and so $e_{h(\gamma)}-\lambda \sigma\in\mathrm{ker}(\tau)$ for some $\lambda\in\Lambda e_{h(\gamma)}$. Since $\Lambda \gamma\neq0$, by condition (3) of Definition \ref{definition:complete-gentle-algebra} we have that $\mathrm{ker}(\tau)\subseteq\sum\Lambda x$ where $x$ runs through $\mathbf{A}(h(\gamma)\rightarrow)$. Since $\lambda \sigma$ lies in this sum, so must $e_{h(\gamma)}$, a contradiction.
\end{proof}
The following is a direct application of Lemma \ref{lemma:qbsb-uniserial-cyclics-given-by-arrows}.
\begin{corollary}\label{corollary:span-power-of-radical}
Let $\Lambda$ be a complete gentle $R$-algebra. For any integer $n>0$ and any vertex $v$ we have 
\[
\begin{array}{cc}
  (\mathrm{rad}(\Lambda))^{n}e_{v}=\sum_{\gamma\in\mathbf{P}(n,v\rightarrow)}\Lambda \gamma, &   e_{v}(\mathrm{rad}(\Lambda))^{n}=\sum_{\gamma\in\mathbf{P}(n,\rightarrow v)}\gamma\Lambda.
\end{array}
\]
Furthermore, for any $\gamma\in\mathbf{P}$ of length $n>1$ with $\mathrm{f}(p)=x$ and $\mathrm{l}(p)=y$ we have
\[
\begin{array}{cc}
 (\mathrm{rad}(\Lambda))^{n-1}x=\Lambda \gamma, &  y(\mathrm{rad}(\Lambda))^{n-1}=\gamma\Lambda.
\end{array}
\]
\end{corollary}
Recall Example \ref{example:intro-string-module}. By Corollary \ref{corollary:span-power-of-radical}, or by a direct calculation, one can show that for any integer $n>0$ the ideal $(\mathrm{rad}(k[[x,y]]/(xy)))^{n}$ is generated by the paths $x^{n},y^{n}$. Likewise the corresponding power of the radical for the algebra from Example \ref{example:intro-band-module} is generated by the alternating sequences in $a$ and $b$ of length $n$.

Parts (1) and (2) of Lemma \ref{lemma:technical-comp-gen-props} are precisely \cite[Corollary 2.11(i,ii)]{Ben2016} and \cite[Corollary 2.11(iv,v)]{Ben2016} respectively. We omit the proofs, as they are given in \cite[Corollary 1.2.14(iii)]{Ben2018} and \cite[Corollary 1.2.18]{Ben2018} respectively, and use similar arguments to those used in the proof of Lemma \ref{lemma:qbsb-uniserial-cyclics-given-by-arrows}.
\begin{lemma}\label{lemma:technical-comp-gen-props}\emph{\cite[Corollaries 1.2.14 and 1.2.18]{Ben2018}} Let $\Lambda$ be a complete gentle $R$-algebra, and $\gamma\in\mathbf{P}$, $\lambda\in\Lambda e_{v}$ \emph{(}respectively, $\lambda\in e_{v}\Lambda$\emph{)} where $v=h(\gamma)$. 
\begin{enumerate}
    \item If $\lambda \gamma=0$ \emph{(}respectively, $\gamma\lambda=0$\emph{)} then \emph{$\lambda e_{v}\in\mathrm{rad}(\Lambda e_{v})$}.
    \item Provided $\lambda=\lambda'x$ where $x\in \mathbf{A}(v\rightarrow)$, $x\gamma\in\mathbf{P}$ \emph{(}respectively, $\lambda=x\lambda'$ where $x\in\mathbf{A}(\rightarrow v)$, $\gamma x\in\mathbf{P}$\emph{)} and $\lambda'\in\Lambda$\emph{;} if $\lambda \gamma=0$ \emph{(}respectively, $\gamma\lambda=0$\emph{)} then $\lambda=0$. \end{enumerate}
\end{lemma}

We now introduce a class of complete gentle $R$-algebras for any discrete valuation ring $R$. Hence this furnishes one with examples where $R$ does not contain a canonical ground field.
\begin{example}\label{example:comp-gentle-2-by-2-matrices} Let $R$ be a discrete valuation ring where $\mathfrak{m}=\langle m \rangle$ and $m\in R$ is not nilpotent. Let $Q$ consist of two loops $a$ and $ b$ at a single vertex.  Consider the $R$-subalgebra $\Lambda$ of $M_{2}(R)$ given by the matrices $(r_{ij})$ (where $1\leq i,j\leq 2$) such that $m\mid r_{11}-r_{22},r_{12}$. Consider the $R$-algebra map $\vartheta\colon RQ\to \Lambda$ satisfying
\[
\begin{array}{cc}
\vartheta(a)=\begin{pmatrix}0 & 0\\
1 & 0
\end{pmatrix}, & \vartheta( b)=\begin{pmatrix}0 & m\\
0 & 0
\end{pmatrix}.
\end{array}
\]
Any element $(r_{ij})$ of $\Lambda$ satisfies $r_{11}-r_{22}=sm$ and $r_{12}=tm$ for some $s,t\in R$; and hence has the form
\[
\begin{array}{c}
\begin{pmatrix}r_{11} & r_{12}\\
r_{21} & r_{22}
\end{pmatrix}
=\vartheta(s ba+t b+r_{12}a+r_{22}e_{v}).
\end{array}
\]
This shows $\vartheta$ is onto. Furthermore, if $s_{11},s_{12},s_{21},s_{22}\in R$ and $w, x, y, z\geq 0$ are intgers with $w,z>0$, then
\begin{equation}\label{equation:2loopexample}
\vartheta(s_{11}(a b)^{w}+s_{12} b(a b)^{x}+s_{21}a( ba)^{y}+s_{22}( ba)^{z})=\begin{pmatrix}s_{11}m^{w} & s_{12}m^{x+1}\\
s_{21}m^{y} & s_{22}m^{z}
\end{pmatrix}.
\end{equation}
We now let $\mathscr{I}=\langle a^{2}, b^{2}\rangle$. By our calculations above, in order to check Assumption \ref{assumption:first-assumptions} holds, it suffices to check the claim that $\vartheta(p)\neq 0$ for any path $p\notin\mathscr{I}$. Here $\mathbf{P}$ consists of the alternating sequences in $a$ and $ b$.

By 
Equation \ref{equation:2loopexample}
we have that Assumption \ref{assumption:first-assumptions} holds. Note also that conditions (1), (2) and (3) of Definition \ref{definition:complete-gentle-algebra} hold: (1) holds by looking at the quiver $Q$; and (2) and (3) both hold because $ a^{2}$ and $ b^{2}$ are the generators of $\mathscr{I}$. Moreover, we have that
\[
\begin{array}{cccc}
\Lambda a=\begin{pmatrix}\mathfrak{m} & 0\\ R & 0
\end{pmatrix}, & \Lambda b=\begin{pmatrix}0 & \mathfrak{m}\\
0 & \mathfrak{m}
\end{pmatrix},&
 a\Lambda=\begin{pmatrix}0 & 0\\
R & \mathfrak{m}
\end{pmatrix}, & b\Lambda=\begin{pmatrix}\mathfrak{m} &\mathfrak{m}\\
0 & 0
\end{pmatrix}.
\end{array}
\]
These calculations verify conditions (5) and (6) of Definition \ref{definition:complete-gentle-algebra} hold. Hence to show that $\Lambda$ is a complete gentle  $R$-algebra, it suffices to check condition (4) holds.
Using Equation \ref{equation:2loopexample} we can show that the map
\[
R\to \Lambda e_{v}/ (\Lambda a + \Lambda b), r\mapsto \begin{pmatrix}r &0\\
0 & r
\end{pmatrix}+(\Lambda a + \Lambda b).
\]
is onto. Using our calculations above the kernel of this map is $\mathfrak{m}$. Thus we have $k\simeq \Lambda e_{v}/ (\Lambda a + \Lambda b)$ as $R$-modules, and so $\Lambda a + \Lambda b$ must be a maximal $R$-submodule of $\Lambda e_{v}$. Hence $\Lambda a + \Lambda b$ is a maximal $\Lambda$-submodule of $\Lambda e_{v}$, and so $\mathrm{rad}(\Lambda)e_{v}\subseteq \Lambda a + \Lambda b$. It is straightforward to check that
$\Lambda a$ is a superfluous left ideal
of $\Lambda$. Similarly $\Lambda b$ is superfluous, and so $\mathrm{rad}(\Lambda e_{v})\supseteq\Lambda a\oplus\Lambda b$.
Together with a similar yet dual argument, this verifies condition (3) of Definition \ref{definition:complete-gentle-algebra} holds. Altogether this shows $\Lambda$ is a complete gentle algebra over $R$. It is straightforward to check that $\vartheta$ gives an isomorphism
\[\Lambda\simeq RQ/\langle a^{2},\, b^{2},\,-m+ a b+ b a\rangle
\]
of $R$-algebras. The primitive cycles for this algebra are $ a b$ and $ b a$.
\end{example}
In what follows, for any noetherian ring $A$ we will write: $A\text{-}\boldsymbol{\mathrm{Mod}}$ for the category of all $A$-modules;  $A\text{-}\boldsymbol{\mathrm{mod}}$ for the full subcategory of $A\text{-}\boldsymbol{\mathrm{Mod}}$ consisting of finitely generated $A$-modules;  $\mathcal{K}(A\text{-}\boldsymbol{\mathrm{Proj}})$ for the homotopy category of complexes of projective $A$-module; and $\mathcal{K}(A\text{-}\boldsymbol{\mathrm{proj}})$ for the full subcategory of $\mathcal{K}(A\text{-}\boldsymbol{\mathrm{Proj}})$ consisting of complexes $P$ where each homogeneous component $P^{n}$ is a (projective) object of $A\text{-}\boldsymbol{\mathrm{mod}}$.
\begin{remark}\label{remark:complete-gentle-semiperfect}
Let $\Lambda$ be a complete gentle algebra. By Remark \ref{remark:props-of-complete-gen-algs}(2) $\Lambda $ is module finite over $R$, and recall $R$ is noetherian, local and $\mathfrak{m}$-adically complete by Assumption \ref{assumption:first-assumptions}(1). Altogether this shows $\Lambda$ is noetherian and semiperfect (by, for example, \cite[Example 23.3]{Lam1991}); see \cite[Corollary 2.11(vi)]{Ben2016}. Hence every object $M$ of $\Lambda\text{-}\boldsymbol{\mathrm{mod}}$ has a (projective cover, and hence) a projective resolution $P_{M}$, such that if $P'_{M}$ is another projective resolution of $M$ then $P_{M}\simeq P'_{M}$ in $\mathcal{K}(\Lambda\text{-}\boldsymbol{\mathrm{proj}})$. 
\end{remark}
\begin{assumption}
For the remainder of the article we assume $\Lambda$ is a complete gentle algebra, fixing the notation from Assumption \ref{assumption:first-assumptions}, Notation \ref{notation:paths-P-arrows-A} and Definitions \ref{definition:complete-gentle-algebra} and \ref{definition:firstlastarrows}.
\end{assumption}
\section{String and band modules}\label{section:string-and-band-modules}
\subsection{Words.} Definitions \ref{definition:words-for-modules}, \ref{definition:sign-and-composing-words}, \ref{definition:cyclic-words}, \ref{definition:finitely-generated-words}, \ref{definition:equivalence-relation-on-words} and  \ref{definition:inverses-shifts-words} follow the deinitions and notation used by Crawley-Boevey in \cite[\S 2]{Cra2018}. The same terminology was used by Ricke \cite{Ric2017}.
\begin{definition}\label{definition:words-for-modules}
By a \emph{letter} we mean a symbol of the form $x$ (called \emph{direct}) or $x^{-1}$ (called \emph{inverse}) where $x\in\mathbf{A}$. The head and tail of a letter are defined by $h(x^{-1})=t(x)$ and $t(x^{-1})=h(x)$. 

Let $I$ be one of the
sets $\{0,\dots,m\}$ (for some $m\geq0$), $\mathbb{N}$, $-\mathbb{N}=\{-n\colon n\in\mathbb{N}\}$
or $\mathbb{Z}$. For $I=\{0\}$ there are \emph{trivial} words
$1_{v,1}$ and $1_{v,-1}$ for each vertex
$v$.  For $I\neq\{0\}$, an $I$-\emph{word} refers to a sequence of the form 
\[
 C =\begin{cases}
 C_{1} \dots  C_{m}  & (\mbox{if }I=\{0,\dots,m\})\\
 C_{1}  C_{2} \dots & (\mbox{if }I=\mathbb{N})\\
\dots C_{-2}  C_{-1}  C_{0}  & (\mbox{if }I=-\mathbb{N})\\
\dots\dots C_{-2}  C_{-1}  C_{0} \mid  C_{1}  C_{2} \dots & (\mbox{if }I=\mathbb{Z})
\end{cases}
\] 
such that the following conditions hold.
\begin{enumerate}
    \item Each $C_{i}$ is a letter.
    \item If $i-1,i,i+1\in I$ then $t(C_{i})=h(C_{i+1})$ and $C_{i}^{-1}\neq C_{i+1}$.
    \item If a subsequence of consecutive letters $C_{i}\dots C_{i+n}$ has the form $\gamma$ or $\gamma^{-1}$ for a path $\gamma$, then $\gamma\in\mathbf{P}$.
\end{enumerate}
By condition (2), for each $i\in I$ there is an \emph{associated vertex} $v_{C}(i)$
defined by $v_{C}(i)=t(C_{i})$ for $i\leq0$ and $v_{C}(i)=h(C_{i})$
for $i>0$ provided $I\neq\{0\}$, and $v_{ 1 _{v,\pm1}}(0)=v$ otherwise.

The \emph{inverse} $l^{-1}$ of a letter is given by $(x)^{-1}=x^{-1}$ and $(x^{-1})^{-1}=x$. The \emph{inverse} of an $I$-word $C$ is defined by $(1_{v,\delta})^{-1}=1_{v,-\delta}$ when $I=\{0\}$, and otherwise inverting the letters and reversing their order:
\[
 C^{-1} =\begin{cases}
 C_{m}^{-1}\dots C_{1}^{-1}  & (\mbox{if }I=\{0,\dots,m\})\\
 \dots C_{2}^{-1}  C_{1}^{-1}  & (\mbox{if }I=\mathbb{N})\\
C_{0}^{-1}  C_{-1}^{-1} C_{-2}^{-1}\dots   & (\mbox{if }I=-\mathbb{N})\\
\dots\dots C_{3}^{-1}  C_{2}^{-1}  C_{1}^{-1} \mid  C_{0}^{-1}  C_{-1}^{-1} \dots & (\mbox{if }I=\mathbb{Z})
\end{cases}
\]
Let $n\in\mathbb{Z}$. If $I=\mathbb{Z}$ we define the word $C[n]$ by $C[n]_{i}=C_{i+n}$ for all $i\in\mathbb{Z}$. If $I\neq \mathbb{Z}$ we let $C[n]=C$. 
\end{definition}
Remarks \ref{remark:quadraticness-of-gentles-on-words} and \ref{remark:sign-and-composability} below do not hold in general for the words studied by Crawley-Boevey \cite{Cra2018} or Ricke \cite{Ric2017}. This is because these authors consider string algebras, which need not be quadratic, such as $k[x]/(x^{3})$.
\begin{remark}\label{remark:quadraticness-of-gentles-on-words}
Let $C$ be an $I$-word. The third item in Definition \ref{definition:words-for-modules} says, in other words, that the walk in $Q$ given by $C$ does not pass over any zero relations in $\mathscr{I}$. Recall that, by Assumption \ref{assumption:first-assumptions}, the ideal $\mathscr{I}$ of $RQ$ is generated by a set of length $2$ paths, and the set $\mathbf{P}$ of paths $\gamma$ in $Q$ with $\gamma\notin\mathscr{I}$ is precisely the set of paths $\gamma$ with $\vartheta(\gamma)\neq 0\in\Lambda$. Consequently the third item in Definition \ref{definition:words-for-modules} is equivalent to requiring that when both $C_{i},C_{i+1}$ are direct (respectively inverse) the consecutive pair $C_{i}C_{i+1}$ (respectively $C^{-1}_{i+1}C^{-1}_{i}$) lies in $\mathbf{P}$.  
\end{remark}
\begin{definition}\label{definition:sign-and-composing-words}
For any letter $l$ we choose a \emph{sign} $s(l)\in\{1,-1\}$, such that if distinct letters $l$ and $l'$ satisfy  ($h(l)=h(l')$ and $s(l)=s(l')$) then $\{l,l'\}=\{x^{-1},y\}$ where $xy\in\mathscr{I}$. Let $C$ be an $I$-word with $I\subseteq\mathbb{N}$. We define the \emph{sign} of $C$ by $s(1_{v,\delta})=\delta$ when $I=\{0\}$ and $s(C)=s(C_{1})$ otherwise. 

If $B$ is a (finite or $-\mathbb{N}$)-word and $D$ is a (finite or $\mathbb{N}$)-word then the \emph{composition} $BD$ \emph{of} $B$ \emph{and} $D$ is given by concatenating the letters, and $B$ and $D$ are \emph{composable} if $h(B^{-1})=h(D)$ and $s(B^{-1})=-s(D)$. 
\end{definition}
\begin{remark}\label{remark:sign-and-composability}
Let $C$ be a (finite or $-\mathbb{N}$)-word and $D$ be a (finite or $\mathbb{N}$)-word, and in case they are non-trivial write $D=D_{1}\dots$ and $C=\dots C_{0}$. Assuming $h(C^{-1})=h(D)$ and $s(C^{-1})=-s(D)$ gives $s(C_{0}^{-1})=-s(D_{1})$, and so $C_{0}^{-1}\neq D_{1}$ and ($C_{0}D_{1}\notin\mathbf{P}$ if and only if $s(C_{0}^{-1})=s(D_{1})$) since  $t(C_{0})=h(D_{1})$. Thus $C$ and $D$ are composable if and only if $CD$ is a word.
\end{remark}
\begin{example}\label{example:signs-for-k[[x,y]]/(xy)}We continue with Example \ref{example:intro-string-module} where $\Lambda=k[[x,y]]/(xy)$. Here the letters are $x$, $y$, $x^{-1}$ and $y^{-1}$. A sequence of letters is a word if and only if it does not contain any of $xy$, $yx$, $x^{-1}y^{-1}$ or $y^{-1}x^{-1}$ as a consecutive pair of letters. Hence the sequences $B$, $A$ and $D$ from Example \ref{example:intro-string-module} are all words. To define a sign here, start by choosing $s(x)=1$. To satisfy Definition \ref{definition:sign-and-composing-words} we must choose $s(y^{-1})=1$ and $s(x^{-1})=-1=s(y)$. If $E=\dots E_{0}$ is a non-trivial (finite or $-\mathbb{N}$)-word and $F=F_{1}\dots$ is a non-trivial (finite or $\mathbb{N}$)-word, then $E$ and $F$ are composable if and only if 
\[
(E_{0},F_{1})\in\{(x,x),(x,y^{-1}),(x^{-1},x^{-1}),(x^{-1},y),(y,y),(y,x^{-1}),(y^{-1},y^{-1}),(y^{-1},x)\}.
\]
For example, the words $B^{-1}$ and $A$ from Example \ref{example:intro-string-module} are composable. Similarly, $A$ and $D$ are composable. 
\end{example}
\begin{definition}\label{definition:cyclic-words}
A non-trivial finite word $A=A_{1}\dots A_{n}$ is called \emph{cyclic} if $A^{2}$ is a word and $A$ is not a non-trivial power of another word. By Remark \ref{remark:sign-and-composability}, $A^{2}$ is a word if and only if $t(A_{n})=h(A_{1})$ and $s(A_{n})=-s(A_{1})$. For a cyclic word $A$ we define the $-\mathbb{N}$-word ${}^{\infty\hspace{-0.5ex}}A$, the $\mathbb{N}$-word $A^{\infty}$ and the $\mathbb{Z}$-word ${}^{\infty\hspace{-0.5ex}}A^{\infty}$ by
\[
{}^{\infty\hspace{-0.5ex}}A=\dots AA=\dots A_{1}\dots A_{n} A_{1} \dots A_{n},
\]
\[
A^{\infty}= AA\dots= A_{1}\dots A_{n} A_{1} \dots A_{n}\dots,
\]
and 
\[
{}^{\infty\hspace{-0.5ex}}A^{\infty}=\dots A\mid AA\dots=\dots A_{1}\dots A_{n}\mid A_{1} \dots A_{n}\dots
\]
\end{definition}
\begin{example}\label{2-by-2-matrices-cyclic-words}
We continue with Example \ref{example:intro-band-module} where $\Lambda \simeq \widehat{\mathbb{Z}}_{p}Q/\langle a^{2}, b^{2}, a b+ b a-p\rangle$. Here we let $s( a)=1$, which must mean $s( a^{-1})=1$ and $s( b)=-1=s( b^{-1})$. Recall that, in Example \ref{example:intro-band-module}, we considered the $\mathbb{Z}$-word $C={}^{\infty\hspace{-0.5ex}}E^{\infty}$ where $E$ is the word $a^{-1}b^{-1} a^{-1} b a b^{-1} a^{-1} b^{-1} a b^{-1} a b $, and since $ba^{-1}$ is a word, $E^{2}$ is a word. It is straightforward to check that $E$ cannot be a power of a shorter word; and hence $E$ is cyclic.
\end{example}
\begin{definition}\label{definition:finitely-generated-words} Let $C$ be an $I$-word. For each $i\in I$ define the words
\[\begin{array}{cc}
    C_{>i}=
    \begin{cases}
    1_{h(C^{-1}),-s(C^{-1})} & (\text{if }i+1\notin I)\\
    C_{i+1}C_{i+2}\dots & (\text{otherwise})
    \end{cases} &  
    C_{\leq i}=
    \begin{cases}
    1_{h(C),s(C)} & (\text{if }i-1\notin I)\\
   \dots C_{i-1}C_{i} & (\text{otherwise})
    \end{cases}
\end{array}
\]
We say $C$ is \emph{direct} (respectively \emph{inverse}) if $C_{i}$ is direct (respectively inverse) for all $i\in I$ with $i-1\in I$. We say $C$ is \emph{eventually left downward} 
 (respectively \emph{upward})
provided $-\mathbb{N}\subseteq I$ and there exists $i\in I$ such that $C_{\leq i}$ is direct (respectively inverse). Dually, we say $C$ is \emph{eventually right downward}
 (respectively \emph{upward})
provided $\mathbb{N}\subseteq I$ and there exists $i\in I$ such that $C_{> i}$ is inverse (respectively direct). We say $C$ is \emph{eventually downward} (respectively \emph{upward}) provided $C$ is eventually left downward (respectively upward), eventually right downward (respectively upward), or both. 
\end{definition}
\begin{example}\label{example:eventually-downward-words-for-k[[x,y]]/(xy)}
We continue with Example \ref{example:intro-string-module} and \ref{example:signs-for-k[[x,y]]/(xy)}, where $\Lambda=k[[x,y]]/(xy)$ and $A=y^{-3}x^{2}y^{-1}x^{2}y^{-3}x^{4}$. It is straightforward to check $A$ is a cyclic $\{0,\dots,15\}$-word. Note that $y^{-1}$ is a cyclic $\{0,1\}$-word, and recall the word $D$ satisfied $D=(y^{-1})^{\infty}$. Since the $\mathbb{N}$-word $C$ satisfied $C_{>17}=D$, $C$ is an example of an $\mathbb{N}$-word which is eventually right downward. 
\end{example}
\begin{remark}\label{remark:eventually-downward-words}
Suppose $C$ is a direct $-\mathbb{N}$-word, and so there are arrows $x_{1},x_{2},\dots\in \mathbf{A}$ such that $C_{i}=x_{1-i}$ for all $i$. That is, $C=\dots x_{3}x_{2}x_{1}$. Since $Q$ is finite there exists some $i<0$ for which $x_{1}=x_{1-i}$. Hence we can choose $n>0$ minimal such that $x_{1}=x_{n+1}$. By condition (2) from Definition \ref{definition:complete-gentle-algebra}, and by induction, we have $x_{t}=x_{t+n}$ for all $t>0$. This gives $C={}^{\infty\hspace{-0.5ex}}P$ where $P=x_{n}\dots x_{1}$. Since $C$ is a word $P$ is a cycle with no relations. Furthermore, by the minimality of $n>0$ we have that $P$ cannot be written as a power of a shorter direct word. Thus $P$ is a primitive cycle with no relations. The above argument shows $C$ is eventually left-downward if and only if $(C_{\leq i})^{-1}=(P^{-1})^{\infty}$ for some $i\in I$ and a primitive cycle $P$ with no relations. Dually, $C$ is eventually right downward if and only if there exists $i$ and $P$ with $C_{> i}=(P^{-1})^{\infty}$. 
\end{remark}
\subsection{String modules}
\begin{definition}\label{definition:finitely-generated-words-II}
Let $C$ be an $I$-word. We call $i\in I$ a $C$-\emph{peak} if ($C_{i}$ is direct or $i-1\notin I$) and ($C_{i+1}$ is inverse or $i+1\notin I$). We say $C$ is \emph{peak}-\emph{finite} if there are finitely many $C$-peaks in $I$.

By an \emph{alternating} word we mean a non-trivial word $A$ of the form 
\[
A=\gamma_{1}^{-1}\sigma_{1}\dots\gamma_{i}^{-1}\sigma_{i}\dots\gamma_{n}^{-1}\sigma_{n},
\]
where each $\gamma_{i},\sigma_{i}\in\mathbf{P}$. Thus if each $\gamma_{i}$ (respectively  $\sigma_{i}$) has length $l_{i}$ (respectively  $m_{i}$) then $A$ is a $\{0,\dots,d\}$-word where $d=\sum_{i=1}^{n}(l_{i}+m_{i})$. 
\end{definition}
\begin{lemma}\label{lemma:checking-alternating-word-technical}
Let $n\geq 1$, $C$ be an $I$-word and $\pi(0),\dots,\pi(n)$ be $C$-peaks in $I$ such that, for all $i=0,\dots, n-1$, $\pi(i)<\pi(i+1)$ and there are no $C$-peaks $p$ with $\pi(i)<p<\pi(i+1)$. Then $C_{\pi(0)+1}\dots C_{\pi(n)}$ is alternating.
\end{lemma}
\begin{proof}
Since $n>0$ we have $\pi(0)+1\in I$, and since $\pi(0)$ is a peak we have that $C_{\pi(0)+1}$ is inverse. Likewise $C_{\pi(1)}$ is direct. Choose $d>0$ maximal such that $C_{\pi(0)+1}\dots C_{\pi(0)+d}$ is inverse. Choose $m\geq0$ maximal such that $C_{\pi(1)-m}\dots C_{\pi(1)}$ is inverse. This means $\pi(1)-m>\pi(0)+d$. For a contradiction suppose $\pi(1)-m>\pi(0)+d+1$. Then $C_{\pi(0)+d+1}$ is direct and $C_{\pi(1)-m-1}$ is inverse. Hence there is some $l\geq d$ maximal such that $C_{\pi(0)+d}\dots C_{\pi(0)+l}$ is direct. By the maximality of $l$, and since  $C_{\pi(1)-m-1}$ is inverse this means $\pi(0)+l+1\leq \pi(1)-m-1$ which is a $C$-peak in $I$. This contradicts the hypothesis of the lemma, and so $\pi(1)-m=\pi(0)+d+1$. This shows that $C_{\pi(0)+1}\dots C_{\pi(1)}=\gamma_{1}^{-1}\sigma_{1}$ for some $\gamma_{1},\sigma_{1}\in\mathbf{P}$ (since $C$ is a word). By repeating this argument one can show $C_{\pi(0)+1}\dots C_{\pi(n)}$ is alternating. 
\end{proof}
\begin{lemma}\label{lemma:no-peaks-implies-eventually-upward}
If $C$ is an $I$-word with no $C$-peak in $I$, then $C$ is eventually upward.
\end{lemma}
\begin{proof}
Note that if $I=\{0\}$ then $0$ is a $C$-peak by definition. Hence $I\neq \{0\}$.  Note that if $1\notin I$ then $I=-\mathbb{N}$ in which case $C$ is eventually left upward if and only if the $\mathbb{N}$-word $C^{-1}$ is eventually right upward. Thus from here we assume $1\in I$. Consider the case when $C_{1}$ is direct. We now claim, by induction, that for all $i>0$ we have $i\in I$ and $C_{i}$ is direct. We are assuming that the claim holds for $i=1$. Now suppose  $i\in I$ and $C_{i}$ is direct. Since $i$ is not a $C$-peak, we must have that $i+1\in I$ and that $C_{i+1}$ is direct. Hence our inductive claim holds, and $C$ is eventually right upward in this case. Now instead suppose $C_{1}$ is inverse. As above, for all $i\geq0$ we have that $1-i\in I$ and $C_{1-i}$ is inverse, and so $C$ is eventually left upward. 
\end{proof}
\begin{definition}\label{definition:string-words}
By a \emph{string word} we mean a word of the form $C=B^{-1}(AD)$ where $A$ is either trivial or alternating, and each of $B$ and $D$ are either trivial or inverse. Note that when $C$ is a $\mathbb{Z}$-word the order of composition means that $C=
B^{-1}\mid AD$.  We call $(B,A,D)$ a \emph{decomposition} of $C$.
\end{definition}
Lemma \ref{lemma:decompositions-under-shifts-and-inverses} says that there are no non-trivial shifts of a string word, and only one non-trivial shift of its inverse, which remains a string word. Hence considering string words, rather than all words, is fairly restrictive. In Proposition \ref{proposition:peak-finite-iff-shift-of-string} and Remark \ref{remark:comparing-terminology-used-by-Ricke} we show that this consideration is not too restrictive.
\begin{lemma}\label{lemma:decompositions-under-shifts-and-inverses}
Let $C$ and $C'$ be string words with decompositions $(B,A,D)$ and $(B',A',D')$ respectively. 
\begin{enumerate}
    \item If $C'=C[n]$ then $C'=C$ and $(B',A',D')=(B,A,D)$.
    \item If $C'=C^{-1}[n]$ then $C'=C^{-1}[-d]$ where $A$ is a $\{0,\dots,d\}$-word and $(B',A',D')=(D,A^{-1},B)$.
\end{enumerate}
\end{lemma}
\begin{proof}By assumption we have $C'=B'^{-1}(A'D')$ and $C=B^{-1}(AD)$. Let $C$ be an $I$-word and $C'$ be an $J$-word. Let $A$ be a $\{0,\dots,d\}$-word and $A'$ be a $\{0,\dots,d'\}$-word. Choose $i(-),i(+)\in I$ such that $B=(C_{\leq i(-)})^{-1}$ and $D=C_{>i(+)}$, and choose $j(-),j(+)\in J$ with $B'=(C_{\leq j(-)})^{-1}$ and $D=C_{>j(+)}$. Note that the inverse or shift of $C$ (respectively, $C'$) is a $\mathbb{Z}$-word if and only if $I=\mathbb{Z}$  (respectively, $J=\mathbb{Z}$).

(1) Suppose $C'=C[n]$ for some $n\in\mathbb{Z}$. Then for all $j\in J$ with $j-1\in J$ we have $C'_{j}=C_{j+n}$ which is (direct if $j\leq j(-)$) and  (inverse if $j>j(+)$). Since $C_{i(-)+1}$ is inverse if $i(-)+1\in I$, we have $j(-)\leq i(-)-n$. Since $C_{i(+)}$ is direct if $i(+)-1\in I$, we have $j(+)\geq i(+)-n$. Dually one can show $i(-)\leq j(-)+n$ and $i(+)\geq j(+)+n$, which altogether means $i(\pm )=j(\pm) +n$. Note that if $I\neq \mathbb{Z}$ then $C[n]=C$, which means we can take $n=0$ and hence $i(\pm )=j(\pm)$. Otherwise $I=\mathbb{Z}$, in which case since $C=B^{-1}(AD)$ we must have $i(-)=0$, and similarly $j(-)=0$; and then we have $n=0$ and thus $i(+)=j(+)$. Thus, in case $C'=C[n]$, we must have $C'=C$ and $i(\pm )=j(\pm)$.

(2) Suppose $C'=C^{-1}[n]$ for some $n\in\mathbb{Z}$. Note firstly that $C^{-1}[-d]$ is a string word with decomposition $(D,A^{-1},B)$, since a word $A$ is (trivial or alternating) if and only if the inverse $A^{-1}$ is (trivial or alternating) respectively. Furthermore, if we let $C''=C^{-1}[-d]$ and $(B'',A'',D'')=(D,A^{-1},B)$ then $C'=C''[n+d]$, and by the proof of (1) above this means $C'=C''$ and that $(B',A',D')=(B'',A'',D'')$, as required.
\end{proof}
\begin{remark}
Suppose $C$ is a string word. Taking $n=0$ in Lemma \ref{lemma:decompositions-under-shifts-and-inverses}(1) shows that if $(B,A,D)$ and $(B',A',D')$ are decompositions of $C$ then $B'=B$, $A'=A$ and $D'=D$. Henceforth we refer to $(B,A,D)$ as \emph{the} decomposition of $C$, since it is unique.
\end{remark}
\begin{example}\label{example:alternating-words-for-k[[x,y]]/(xy)}
We continue Example \ref{example:eventually-downward-words-for-k[[x,y]]/(xy)} where $\Lambda=k[[x,y]]/(xy)$. Here $C=B^{-1}(AD)$ where $B=x^{-3}$, $A=y^{-3}x^{2}y^{-1}x^{2}y^{-3}x^{4}$ and $D=(y^{-1})^{\infty}$. The $C$-peaks in $\mathbb{N}$ are $3$, $8$, $11$ and $18$. Note that $A$ is a alternating word, and that $B$ and $D$ are both inverse. Hence $C$ is a string word with decomposition $(B,A,D)$.
\end{example}
\begin{proposition}\label{proposition:peak-finite-iff-shift-of-string}
Let $C$ be an $I$-word which is not eventually upward. Then $C$ is peak-finite if and only if $C[n]$ is a string word for some $n\in\mathbb{Z}$. Hence when $I\neq \mathbb{Z}$, $C$ is peak-finite if and only if $C$ is a string word.
\end{proposition}
\begin{proof}
The proof that any string word is peak-finite is straightforward, and it is clear that $C$ is peak-finite if and only if any $C[n]$ is. Conversely, suppose $C$ is peak-finite. It suffices to find $i,j\in I$ such that $A=(C_{>i})_{\leq j-i}$ is (trivial or alternating) and each of $B=(C_{\leq i})^{-1}$ and and $D=C_{>j}$ are (trivial or inverse).

By Lemma \ref{lemma:no-peaks-implies-eventually-upward} there exists at least one $C$-peak in $I$. Write $\pi(0),\dots,\pi(d)\in I$ for the $d+1$ distinct $C$-peaks in $I$ (with $d\in\mathbb{N}$), where $\pi(i)<\pi(j)$ if and only if $i<j$. Note that if $\pi(d)+1\in I$ we must have that  $C_{\pi(d)+1}$ is inverse. We now claim $C_{>\pi(d)}$ is trivial or inverse. For a contradiction assume that $C_{\pi(d)+n}$ is direct for some $n>0$.  The statement of the Lemma includes the assumption the $C$ is not right upward, and so $C_{\pi(d)+m}$ is inverse for some $m>n$. Taking $m$ to be minimal shows that $\pi(d)+m-1$ is a peak, contradicting the maximality of the $C$-peak $\pi(d)$, noting that $m>1$. Thus $C_{>\pi(d)}$ is trivial or inverse. Similarly $C_{\leq \pi(1)}$ is trivial or direct. Let $B=(C_{\leq \pi(1)})^{-1}$ and $D=C_{>\pi(d)}$. If $d=0$ then $C=BD$ and we can take $A=1_{h(D),s(D)}$. Otherwise $d>0$ and $A=C_{\pi(0)+1}\dots C_{\pi(d)}$ is alternating by Lemma \ref{lemma:checking-alternating-word-technical}.
\end{proof}
\begin{remark}\label{remark:comparing-terminology-used-by-Ricke}
We now compare the terminology from Definitions \ref{definition:finitely-generated-words} and \ref{definition:string-words} with some terminology used in the thesis of Ricke \cite{Ric2017}. Following \cite[\S 3.1.2, p.24]{Ric2017}, a word $C$ is called \emph{eventually inverse} (respectively, \emph{direct}) if there are only finitely many $i>0$ such that $C_{i}$ is direct (respectively, inverse). Hence we have that $C$ is eventually inverse if and only if $I\neq \mathbb{N},\mathbb{Z}$ or ($\mathbb{N}\subseteq I$ and $C$ is eventually right downward). Similarly $C^{-1}$ is eventually inverse if and only if $I\neq -\mathbb{N},\mathbb{Z}$ or ($-\mathbb{N}\subseteq I$ and $C$ is eventually left downward). 

Hence the condition that (both $C$ and $C^{-1}$ are eventually inverse) is equivalent to saying that $C$ is peak-finite and not eventually upward, which by Proposition \ref{proposition:peak-finite-iff-shift-of-string} is equivalent to saying that some shift of $C$ is a string word. In this way, the words defining the string modules in \cite[Theorem 3.3.28]{Ric2017} are equivalent to the words indexing the string modules in (our main) Theorems \ref{theorem:main-fin-gen-modules-are-sums-of-strings-and-bands}, \ref{theorem:main-isomorphisms-between-string-and-band-modules} and \ref{theorem:main-krull-remak-schmidt-property}. 
\end{remark}
In what follows we define a module $M(C)$ for any string word $C$.  String modules have been defined before, for example by Ringel \cite{Rin1975}, Butler and Ringlel \cite{ButRin1987} and Crawley-Boevey \cite{Cra19882,Cra1989,Cra19892,Cra2018}. These authors worked in the context of a $k$-algebra for some field $k$, and defined a string module for each word by means of: a $k$-basis indexed by the vertices in-between the letters; and relations (between the elements of the basis) defined by each letter. Recall that the canonical ground ring $R$ is not necessarily a field. Consequently it is unclear how to adjust the definitions these authors use to our purposes. 

Instead we consider an idea of Huisgen-Zimmermann and Smal\o{} \cite{HuiSma2005}. Here, in the terminology of Definition \ref{definition:string-words}, string modules are defined by introducing a generator for each $C$-peak in $I$, and introducing relations dependent upon the decomposition $(B,A,D)$ of $C$. The reader should be warned that the \emph{generalised words} considered in \cite{HuiSma2005} are not the same as those we consider in this article, in Definition \ref{definition:generalised-words-for-complexes}. 
\begin{definition}\label{definition:string-modules-by-string-words}
Let $C$ be a string word with decomposition $(B,A,D)$. By Lemma \ref{lemma:no-peaks-implies-eventually-upward} and Proposition \ref{proposition:peak-finite-iff-shift-of-string} there are $d+1>0$ distinct $C$-peaks $\pi(0),\dots,\pi(d)\in I$ such that $\pi(i)<\pi(j)$ when $i<j$. Let $g_{C,i}=e_{v_{C}(\pi(i))}$ for each $i$. If $I$ is bounded below then $B$ is finite and there is at most one $w\in\mathbf{A}$ such that $wB^{-1}$ is a word. When such an $w$ exists we let $L_{-}(C)=\Lambda wB^{-1}$, considered as a submodule of $\Lambda g_{C,0}\subseteq N(C)$. If $I\supseteq -\mathbb{N}$ or no such $w$ exists, let $L_{-}(C)=0$. Dually, if $I$ is bounded above there is at most one $z\in\mathbf{A}$ such that $Dz^{-1}$ is a word, in which case we let $L_{+}(C)=\Lambda zD^{-1}$, considered as a submodule of $\Lambda g_{C,n}\subseteq N(C)$. Likewise let $L_{+}(C)=0$ if $I\supseteq\mathbb{N}$ or no such $z$ exists. Let $M(C)=N(C)/L(C)$ where
\[
N(C)=\bigoplus_{i=0}^{d}\Lambda g_{C,i},\,L(C)=L_{-}(C)+L_{0}(C)+L_{+}(C)
\]
and
\[
L_{0}(C)=
\begin{cases}
\sum_{i=0}^{d-1}\Lambda (\gamma_{i+1}g_{C,i}-\sigma_{i+1}g_{C,i+1})) & (\text{if }A=\gamma_{1}^{-1}\sigma_{1}\dots\gamma_{d+1}^{-1}\sigma_{d+1})\\
0 & (\text{otherwise})
\end{cases}
\]
\end{definition}
\begin{example}\label{example:string-module-for-k[[x,y]]/(xy)} We continue with Example (\ref{example:intro-string-module} and) \ref{example:alternating-words-for-k[[x,y]]/(xy)}. Recall that $C=B^{-1}(AD)$ where  $B=x^{-3}$, $D=(y^{-1})^{\infty}$ and $A=y^{-3}x^{2}y^{-1}x^{2}y^{-3}x^{4}$, and so the $C$-peaks in $\mathbb{N}$ are $3$, $8$, $11$ and $18$. Note that $D$ is an $\mathbb{N}$-word and $xB^{-1}=x^{4}$. This means $N(C)=\bigoplus_{i=0}^{3}\Lambda g_{C,i}$, $L_{-}(C)=\Lambda x^{4}g_{C,0}$ and $L_{+}(C)=0$; and so
\[
L(C)=\Lambda x^{4}g_{C,0} + \Lambda (y^{3}g_{C,0}-x^{2}g_{C,1})+\Lambda (yg_{C,1}-x^{2}g_{C,2})+\Lambda(y^{3}g_{C,2}-x^{4}g_{C,3}).
\]
Recall the diagram depicting $M(C)$ from Example \ref{example:intro-string-module}. The left-most bullet has no arrow labelled $x$ leaving it, indicating that the associated element $x^{3}g_{C,1}+L(C)\in M(C)$ is annihilated by $x$. The remaining bullets note the other relations, for example, the right-most bullet corresponds to the relation $y^{3}g_{C,2}+L(C)=x^{4}g_{C,3}+L(C)$. The arrow labelled $y^{\infty}$ indicates that $L_{+}(C)=0$, and hence $y^{n}g_{C,4}\notin L(C)$ for all $n>0$.  
\end{example}
\begin{definition}\label{definition:equivalence-relation-on-words}
We say that two words $C$ and $C'$ are \emph{equivalent}, and write $C'\sim C$, if there is some $n\in\mathbb{Z}$ such that $C'=C[n]$ or $C'=C^{-1}[n]$. 
\end{definition}

\begin{lemma}\label{lemma:equivalence-relation-on-words-strings}
If $C$ and $C'$ are string words with $C\sim C'$ then $M(C)\simeq M(C')$ as $\Lambda$-modules. 
\end{lemma}
\begin{proof}
Let $C$ be an $I$-word and $C'$ be a $J$-word. Let $(B,A,D)$ and $(B',A',D')$ be the decompositions of $C$ and $C'$ respectively. By Lemma \ref{lemma:decompositions-under-shifts-and-inverses} it suffices to assume that $C'=C^{-1}[-d]$ where $A$ is a $\{0,\dots,d\}$-word, and that $B'=D$, $A'=A^{-1}$ and $D'=B$. Note that $J=\mp\mathbb{N}$ if $I=\pm\mathbb{N}$ and $I=J$ otherwise. If $I$ is finite let $I=\{0,\dots,t\}$, if $I=\mathbb{Z}$ let $t=d$, and otherwise if $I=-\mathbb{N},\mathbb{N}$ let $t=0$. In this notation we have $C'_{j}=C_{1+t-j}^{-1}$ and $v_{C'}(j)=v_{C}(t-j)$ for each $j\in J$. Let $\pi(0),\dots,\pi(d)$ be the $C$-peaks in $I$. Hence $\pi'(0),\dots,\pi'(d)$ are the $C'$-peaks in $J$ where $\pi'(l)=t-\pi(l)$ for each $l=0,\dots,d$. Consider the $\Lambda$-module isomorphism $\varphi\colon N(C')\to N(C)$ given by $g_{C',l}\mapsto g_{C,t-l}$ for each $l$. Since $B'=D$ we have that the image of $L_{-}(C')$ under $\varphi$ is $L_{+}(C)$. Similarly: the image of $L_{+}(C')$ under $\varphi$ is $L_{-}(C)$ since $D'=B$; and the image of $L_{0}(C')$ under $\varphi$ is $L_{0}(C)$. Hence $\varphi$ induces an isomorphism $M(C')\simeq M(C)$.
\end{proof}
\subsection{Band modules}
\begin{definition}\cite[pp. 2-3, \S 2]{Cra2018}\label{definition:inverses-shifts-words}
We say a $\mathbb{Z}$-word $C$ is $p$-\emph{periodic} if $C=C[p]$ for some minimal $p>0$. We say $C$ is \emph{primitive} if either $C$ or $C^{-1}$ has the form ${}^{\infty\hspace{-0.5ex}}P^{\infty}$ for some primitive cycle $P$.
\end{definition}
\begin{remark}\label{remark:primitive-periodic-words} By arguments such as those from Remark \ref{remark:eventually-downward-words}, since $Q$ is finite a $\mathbb{Z}$-word $C$ is primitive if and only if $C$ is periodic and (direct or inverse).
\end{remark}
\begin{example}\label{2-by-2-matrices-periodic-words}
We continue with Example \ref{2-by-2-matrices-cyclic-words}, where $\Lambda \simeq \widehat{\mathbb{Z}}_{p}Q/\langle a^{2}, b^{2}, a b+ b a-p\rangle$. Since the primitive cycles are $ a b$ and $ b a$, the primitive words are the $2$-periodic $\mathbb{Z}$-words ${}^{\infty\hspace{-0.1ex}}( a b)^{\infty}=\dots a b\mid a b\dots$, ${}^{\infty\hspace{-0.1ex}}( b a)^{\infty}$, $ {}^{\infty\hspace{-0.1ex}}( a^{-1} b^{-1})^{\infty}$ and ${}^{\infty\hspace{-0.1ex}}( b^{-1} a^{-1})^{\infty}$. The $\mathbb{Z}$-word $C={}^{\infty\hspace{-0.1ex}}E^{\infty}$ from Example \ref{example:intro-band-module} is $12$-periodic and non-primitive.
\end{example}
\begin{definition}\label{definition:band-words}
By a \emph{band word} we mean a periodic non-primitive word of the form $C={}^{\infty\hspace{-0.5ex}}_{}A^{\infty}_{}$ where the cyclic word $A$ is alternating. Note $A$ is uniquely defined (by the period of $C$), and we call $A$ the \emph{cycle of} $C$. 
\end{definition}
\begin{lemma}\label{lemma:shifting-alternating-periods-is-alternating}
Let $t\in\mathbb{Z}$, and let $C$ and $C[t]$ be band words, say with cycles $A$ and $A'$ respectively. Then either $A=A'$ or there is some $m$ with $1<m\leq n$ such that 
\[
A'=\gamma_{m}^{-1}\sigma_{m}\dots\gamma_{n}^{-1}\sigma_{n}\gamma_{1}^{-1}\sigma_{1}\dots\gamma_{m-1}^{-1}\sigma_{m-1}\text{ where }A=\gamma_{1}^{-1}\sigma_{1}\dots\gamma_{i}^{-1}\sigma_{i}\dots\gamma_{n}^{-1}\sigma_{n}.
\]
\end{lemma}
\begin{proof}
Let $A=\gamma_{1}^{-1}\sigma_{1}\dots\gamma_{n}^{-1}\sigma_{n}$, which is a $\{0,\dots,d\}$-word for some $d$. Let $t=j+dl$ where $j\in\{0,\dots,d-1\}$ and $l\in\mathbb{Z}$. Assume $A\neq A'$, so that $C\neq C[t]$, and hence $j>0$ as $C=C[d]$. Since $j>0$ and $C=C[dl]$ we have $A'_{d}=A_{j}$ and $A'_{1}=A_{j+1}$. Since $A'$ is alternating $0$ is a $C'$-peak in $\mathbb{Z}$, and so $A'_{d}A'_{1}=xy^{-1}$ for some distinct $x,y\in\mathbf{A}$ with $t(x)=t(y)$. Altogether $A_{j}A_{j+1}=xy^{-1}$, and so for some $m$ with $1< m\leq n$ we have $x=\mathrm{f}(\sigma_{m-1})$ and $y=\mathrm{f}(\gamma_{m})$. We now have $A'_{1}=\mathrm{f}(\gamma_{m})^{-1}$ and $A'_{d}=\mathrm{f}(\sigma_{m-1})$. Since ($A'_{h}=A_{h+j}$ when $0< h\leq d-j$) and ($A'_{h}=A_{h-d+j}$ when $d-j<h\leq d$), this shows $A'$ has the required form.
\end{proof}
\begin{definition}\label{definition:periodicstring-modules}
Let $C$ be band word with cycle $A=\gamma_{1}^{-1}\sigma_{1}\dots\gamma_{n}^{-1}\sigma_{n}$. Write $\{\pi(i):i\in\mathbb{Z}\}$ for the $C$-peaks in $\mathbb{Z}$ such that $\pi(0)=0$ and $\pi(i)<\pi(j)$ when $i<j$. Let $g_{C,i}=e_{v_{C}(\pi(i))}$ for each $i$. Define $\gamma_{j},\sigma_{j}\in\mathbf{P}$ for all $j\in\mathbb{Z}$ by $\gamma_{i+qn}=\gamma_{i}$ and $\sigma_{i+qn}=\sigma_{i}$ for all $q\in\mathbb{Z}$, $i=1,\dots ,n$. Let $M(C)=N(C)/L(C)$ where
\[
N(C)=\bigoplus _{i\in\mathbb{Z}}\Lambda g_{C,i},\, L(C)=\sum _{i\in\mathbb{Z}}\Lambda (\gamma_{i+1}g_{C,i}-\sigma_{i+1}g_{C,i+1})
\]
\end{definition}
\begin{lemma}\label{lemma:periodic-string-module-has-T-automorphism}
Let $C={}^{\infty\hspace{-0.5ex}}_{}A^{\infty}_{}$ for some alternating word $A$ of the form $A=\gamma_{1}^{-1}\sigma_{1}\dots\gamma_{i}^{-1}\sigma_{i}\dots\gamma_{n}^{-1}\sigma_{n}$. Then $M(C)$ is a right $R[T,T^{-1}]$-module where $T(g_{C,i}+L(C))=g_{C,i-n}+L(C)$. 
\end{lemma}
\begin{proof}
Since $R$ lies in the centre of $\Lambda$, $
M(C)$ is a right $R$-module via restriction. Extending the assignment $g_{C,i}\mapsto g_{C,i-n}$ linearly over $\Lambda$ defines an endomorphism of the projective module $N(C)$ where
\[
\gamma_{i+1}g_{C,i}-\sigma_{i+1}g_{C,i+1}\mapsto \gamma_{i+1}g_{C,i-n}-\sigma_{i+1}g_{C,i-n+1}.
\]
By definition we have $\gamma_{j}=\gamma_{j-n}$ and $\sigma_{j}=\sigma_{j-n}$ for all $j\in\mathbb{Z}$, which shows the right-hand expression above lies in $L(C)$. This shows $T$ defines a $\Lambda$-module endomorphism of $M(C)$ where $g_{C,i}+L(C)\mapsto g_{C,i-n}+L(C)$. Similarly there is a $\Lambda$-module endomorphism $T^{-1}$ of $M(C)$ given by $g_{C,i}\mapsto g_{C,i+n}$. Clearly $T^{-1}$ is the inverse of $T$, and so $M(C)$ is a right $R[T,T^{-1}]$-module. 
\end{proof}
\begin{lemma}\label{lemma:equivalence-relation-on-band-words}
Let $\gamma_{i},\sigma_{i}\in\mathbf{P}$, $1<m\leq n$, and  $A$, $A'$ and $A''$ be cyclic and alternating words where
\[
\begin{array}{ccc}
A=\gamma_{1}^{-1}\sigma_{1}\dots\gamma_{i}^{-1}\sigma_{i}\dots\gamma_{n}^{-1}\sigma_{n}, & A'=\gamma_{m}^{-1}\sigma_{m}\dots\gamma_{n}^{-1}\sigma_{n}\gamma_{1}^{-1}\sigma_{1}\dots\gamma_{m-1}^{-1}\sigma_{m-1}, & A''=A^{-1}.
\end{array}
\]
Then for any $R[T,T^{-1}]$-module $V$ we have $\Lambda$-module isomorphisms
\[
M({}^{\infty\hspace{-0.5ex}}_{}A'^{\infty}_{})\otimes_{R[T,T^{-1}]} V\simeq M({}^{\infty\hspace{-0.5ex}}_{}A^{\infty}_{})\otimes_{R[T,T^{-1}]} V\simeq M({}^{\infty\hspace{-0.5ex}}_{}A''^{\infty}_{})\otimes_{R[T,T^{-1}]} \mathrm{res}\,V.
\]
\end{lemma}
\begin{proof}Let $C={}^{\infty\hspace{-0.5ex}}_{}A^{\infty}_{}$, $C'={}^{\infty\hspace{-0.5ex}}_{}A'^{\infty}_{}$ and $C''={}^{\infty\hspace{-0.5ex}}_{}A''^{\infty}_{}$. By writing $d$ for the sum of the lengths of the paths $\gamma_{i}$ and $\sigma_{i}$ over all $i=1,\dots,n$, we have that $A$ is a $\{0,\dots,d\}$-word. Clearly $C$, $C'$ and $C''$ are $d$-periodic. Clearly $C'=C[j]$ for some $j\in\{0,\dots,d-1\}$: explicitly, where $j$ sum of the lengths of the paths $\gamma_{i}$ and $\sigma_{i}$ over all $i=1,\dots,m-1$. Hence $v_{C'}(t)=v_{C}(t+j)$ for all $t\in\mathbb{Z}$, and so the assignment $g_{C,i}\mapsto g_{C',i+m-1}$ defines an isomorphism $\varphi\colon N(C)\to N(C')$. By writing $A'=\gamma_{1}^{\prime -1}\sigma'_{1}\dots\gamma_{n}^{\prime -1}\sigma'_{n}$ where ($\gamma'_{h}=\gamma_{h+m}$ and $\sigma'_{h}=\sigma_{h+m}$ when $0< h\leq n-m$) and ($\gamma'_{h}=\gamma_{h+m-n}$ and $\sigma'_{h}=\sigma_{h+m-n}$ when $n-m<h\leq n$); it is straightforward to check that the image of $\varphi$ restricted to $L(C)$ lies in $L(C')$.

This gives a $\Lambda$-linear map $\psi\colon M(C)\to M(C')$ sending $g_{C,i}+L(C)$ to $ g_{C',i+m-1}+L(C')$, and so
\[
\psi((g_{C,i}+L(C))T^{\pm 1})=\psi(g_{C,i\mp d}+L(C))=g_{C',i\mp d+m-1}+L(C')=(g_{C',i+m-1}+L(C'))T^{\pm 1}.
\]
Hence $\psi$ is a homomorphism of $\Lambda\text{-}R[T,T^{-1}]$-bimodules. Dually there is a $\Lambda\text{-}R[T,T^{-1}]$-bimodule map $\xi\colon M(C')\to M(C)$ given by $g_{C',i}+L(C')\mapsto g_{C,i-m+1}+L(C)$. Clearly $\psi=\xi^{-1}$ and so $M(C')\simeq M(C)$ as $\Lambda\text{-}R[T,T^{-1}]$-bimodules. The proof that $M(C)\otimes V\simeq M(C'')\otimes \mathrm{res}\,V$ is similar and omitted, recalling $\mathrm{res}$ is the involution of $R[T,T^{-1}]\text{-}\boldsymbol{\mathrm{Mod}}$ swapping the action of $T^{-1}$ and $T$. 
\end{proof}
\begin{remark}\label{remark:comparing-terminology-used-by-Ricke-II}
In Proposition \ref{proposition:peak-finite-iff-shift-of-string} we saw how the string words considered here are equivalent to the words indexing the string modules in the thesis of Ricke; see \cite[Theorem 3.3.28]{Ric2017} and  Remark \ref{remark:comparing-terminology-used-by-Ricke}. In Proposition \ref{proposition:periodic-non-primitive-given-by-alternating} we see how words indexing the band modules from \cite{Ric2017} are equivalent to band words; see Remark \ref{remark:primitive-periodic-words}.
\end{remark}
\begin{proposition}
\label{proposition:periodic-non-primitive-given-by-alternating}
Let $C$ be a $p$-periodic non-primitive $\mathbb{Z}$-word. Then $p>1$ and there exists a unique integer $n$ with $0\leq n <p$ such that $C[n]$ is a band word.
\end{proposition}
\begin{proof}
Note that for any $i\in\mathbb{Z}$, if $C_{1+i}\dots C_{p+i}$ is direct (respectively, inverse) then $C$ is direct (respectively, inverse). Since $C$ is periodic-non-primitive this would be a contradiction by Remark \ref{remark:primitive-periodic-words}. Hence for all $i$ we have that $C_{1+i}\dots C_{p+i}$ is not direct, nor inverse. Hence, in particular, $p>1$. Consider firstly the case where $C_{1}$ is direct. Begin by choosing minimal $m>0$ such that $C_{1+m}$ is inverse. By the above $m<p$. Next choose $l$ minimal such that $1<l$ and $C_{l+m}$ is direct. Since $C_{1}=C_{p+1}$ is direct it is impossible that ($m+1\leq p+1$ and $p+1\leq l-1+m$). Since $m<p$ this means $l+m\leq p+1$.

By the minimality of $l$ we have that $C_{l+m}$ is direct.  Now choose $h$ minimal such that $l< h$ and $C_{h+m}$ is inverse. As $C_{1+m}=C_{1+p+m}$ is inverse it is impossible that ($l-1+m\leq p+1+m$ and $p+1+m\leq h-1+m$), and so we must have that $h-1\leq p$. So far, by the minimality of $l$, $m$ and $h$, we have that $C_{1+m}\dots C_{l-1+m}$ is inverse and that $C_{l+m}\dots C_{h-1+m}$ is direct. 

Hence $C_{1+m}\dots C_{h-1+m}=\gamma_{1}^{-1}\sigma_{1}$ for some $\gamma_{1},\sigma_{1}\in\mathbf{P}$. Recall $h-1\leq p$. Let $J$ be the set of $j>0$ with $l\leq j \leq p$ and $C_{1+m}\dots C_{j-1+m}=\gamma_{1}^{-1}\sigma_{1}\dots \gamma_{d}^{-1}\sigma_{d}$ for some $\gamma_{d},\sigma_{d}\in\mathbf{P}$. By the above we have that $n\in J$, so $J\neq \emptyset$. Since $J$ is bounded above by $p$ (and below by $l$), there is a maximal element $t=\mathrm{max}(J)$. Assume for the moment that $t<p$. By the maximality of $t$ we must have that $C_{t+m}$ is inverse, and altogether this means $C_{i+m}$ is inverse whenever $t<i\leq p$. This contradicts that $C_{m}=C_{p+m}$ is direct, and hence we have $t=p$. Thus the statement of the lemma holds when $C_{1}$ is direct. 

Now consider instead the case where $C_{1}$ is inverse. Let $D=(C[1])^{-1}$, and so $D=\dots C_{3}^{-1}C_{2}^{-1}\mid C_{1}^{-1}C_{0}^{-1}\dots$, which means that $D_{1}=C_{1}^{-1}$ is direct. Hence by the above there is some $q\in\mathbb{Z}$ such that $D_{1+q}\dots D_{p+q}$ is cyclic and alternating. Since $D_{i}=C_{2-i}^{-1}$ for each $i\in\mathbb{Z}$ we must have that $C_{1-q}^{-1}\dots C_{2-p-q}^{-1}$ is cyclic and alternating. This means $C_{1+m}\dots C_{p+m}$ is cyclic and alternating where $m=1-p-q$.

Thus we have some $m\in\mathbb{Z}$ such that $C[m]$ is a band word. Writing $n$ for the remainder of dividing $m$ by $p$ gives $C[m]=C[n]$ since $C$ is $p$-periodic. That $n$ is unique follows from the minimality of $p$.
\end{proof}
\begin{notation}\label{notation:R[T,T^-1]-modules}
Let $R[T,T^{-1}]\text{-\textbf{Mod}}_{R\text{-\textbf{Proj}}}$ be the full subcategory of $R[T,T^{-1}]$-modules which are projective (and hence, as $R$ is local, free) as $R$-modules. Let $R[T,T^{-1}]\text{-\textbf{Mod}}_{R\text{-\textbf{proj}}}$ be the full subcategory of $R[T,T^{-1}]\text{-\textbf{Mod}}_{R\text{-\textbf{Proj}}}$ consisting of $R[T,T^{-1}]$-modules which are finitely generated (and free) as $R$-modules. 

Any object $0\neq V$ of $R[T,T^{-1}]\text{-\textbf{Mod}}_{R\text{-\textbf{proj}}}$ satisfies $V\simeq R^{n}$ (as $R$-modules) for some $n>0$, and $T$ corresponds to an $R$-module automorphism of $V$, and hence is given by an $n\times n$ matrix $(r_{ij})$ with entries $r_{ij}\in R$ such that the determinant $\mathrm{det}(r_{ij})$ is a unit in $R$ (that is, lies outside $\mathfrak{m}$).
\end{notation}
\begin{definition}\label{definition:band-module}
Let $C$ be a band word, and so $C={}^{\infty\hspace{-0.5ex}}A^{\infty}$ for some cyclic alternating word $A$. For any object $V$ of $R[T,T^{-1}]\text{-\textbf{Mod}}_{R\text{-\textbf{Proj}}}$ let $M(C,V)=M(C)\otimes_{R[T,T^{-1}]}V$.
\end{definition}
\begin{lemma}\label{lemma:equivalence-relation-on-words-bands}
Let $V$ be an object of \emph{$R[T,T^{-1}]\text{-\textbf{Mod}}_{R\text{-\textbf{Proj}}}$} and let $C$ and $C'$ be band words. 
\begin{enumerate}
    \item If $C'=C[n]$ for some $n$ then $M(C',V)\simeq M(C,V)$.
    \item If $C'=C^{-1}[n]$ for some $n$ then $M(C',V)\simeq M(C,\mathrm{res}\,V)$.
\end{enumerate}
\end{lemma}
\begin{proof}
By definition $C={}^{\infty\hspace{-0.5ex}}A^{\infty}$ and $C'={}^{\infty\hspace{-0.5ex}}A'^{\infty}$ for some cyclic alternating words $A$ and $A'$. It suffices to separately consider the cases where $C'=C[t]$ for some $t\in\mathbb{Z}$, and where $C'=C^{-1}$. Let $A=\gamma_{1}^{-1}\sigma_{1}\dots\gamma_{i}^{-1}\sigma_{i}\dots\gamma_{n}^{-1}\sigma_{n}$. When $C'=C[t]$ we have, by Lemma \ref{lemma:shifting-alternating-periods-is-alternating}, that either $A'=A$ or there is some $m$ with $1<m\leq n$ such that $A'=\gamma_{m}^{-1}\sigma_{m}\dots\gamma_{n}^{-1}\sigma_{n}\gamma_{1}^{-1}\sigma_{1}\dots\gamma_{m-1}^{-1}\sigma_{m-1}$. Likewise if $C'=C^{-1}$ then $A'=\sigma_{n}^{-1}\gamma_{n}\dots\sigma_{1}^{-1}\gamma_{1}$. The result follows by Lemma \ref{lemma:equivalence-relation-on-band-words}. 
\end{proof}
\begin{definition}\label{definition:string-and-band-modules}By a \emph{finitely generated string module} we mean a module of the form $M(C)$ where $C$ is a string word. By a \emph{finitely generated band module} we mean a module of the form $M(C,V)$ where $C$ is a band word and $V$ is indecomposable as an $R[T,T^{-1}]$-module and finitely generated as an $R$-module. 
\end{definition}
\begin{corollary}\label{corollary:fin-gen-string-band-modules-are-fin-gen} Finitely generated string and band modules are finitely generated $\Lambda$-modules.
\end{corollary}
\begin{proof}
Let $C$ be an $I$-word. If there are $d>0$ distinct $C$-peaks in $I$ then by construction $M(C)$ is a quotient of a direct summand of $\Lambda^{d}$, in which case $M(C)$ is clearly finitely generated. Suppose $C$ is a band word with alternating cycle $A$.  By definition we have $M(C,V)=M(C)\otimes_{R[T,T^{-1}]}V$. Let $(v_{\omega}\colon \omega\in\Omega)$ be an $R$-basis of $V$. Since $R$ lies in the centre of $\Lambda$, and by construction, the $\Lambda$-module $M(C,V)$ is generated by the elements $g_{C,i,\omega}=g_{C,i}\otimes v_{\omega}$ where $i$ runs through $\mathbb{Z}$ and $\omega$ through $\Omega$. Writing any $i\in\mathbb{Z}$ as $i=j+nd$ for $j=0,\dots,n-1$ and $d\in\mathbb{Z}$ gives $g_{C,i,\omega}=g_{C,j}\otimes T^{-d}v_{\omega}$, and hence $M(C,V)$ is generated by the elements $g_{C,0,\omega},\dots, g_{C,n-1,\omega}$ with $\omega\in\Omega$. 
\end{proof}
\section{String and band complexes}\label{section:string-and-band-complexes}
\begin{definition}\label{definition:generalised-words-for-complexes}\cite[\S 4.1]{BekMer2003} By a \emph{generalised letter} we mean one of $\langle\gamma\rangle $ (called \emph{direct}) or $\langle\gamma\rangle ^{-1}$ (called \emph{inverse}) where $\gamma\in\mathbf{P}$. The head and tail are defined by $h(\langle\gamma\rangle )=t(\gamma)=t(\langle\gamma\rangle^{-1} )$ and $t(\langle\gamma\rangle )=h(\gamma)=h(\langle\gamma\rangle^{-1} )$. 

Let $I$ be one of the
sets $\{0,\dots,m\}$ (for some $m\geq0$), $\mathbb{N}$, $-\mathbb{N}=\{-n\colon n\in\mathbb{N}\}$,
or $\mathbb{Z}$. In case $I=\{0\}$ the \emph{trivial} generalised words are 
$ \langle1 _{v,1}\rangle $ and $ \langle1 _{v,-1}\rangle $ for each vertex
$v$. For $I\neq\{0\}$, a \emph{generalised} $I$-\emph{word} refers to a sequence of the form 
\[
 \mathcal{C} =\begin{cases}
 \mathcal{C}_{1} \dots  \mathcal{C}_{m}  & (\mbox{if }I=\{0,\dots,m\})\\
 \mathcal{C}_{1}  \mathcal{C}_{2} \dots & (\mbox{if }I=\mathbb{N})\\
\dots \mathcal{C}_{-2}  \mathcal{C}_{-1}  \mathcal{C}_{0}  & (\mbox{if }I=-\mathbb{N})\\
\dots\dots \mathcal{C}_{-2}  \mathcal{C}_{-1}  \mathcal{C}_{0} \mid  \mathcal{C}_{1}  \mathcal{C}_{2} \dots & (\mbox{if }I=\mathbb{Z})
\end{cases}
\]where each $\mathcal{C}_{i}$ is a generalised letter, and any sequence of the form $\mathcal{C}_{i}\mathcal{C}_{i+1}$ has the form:
\begin{enumerate}
\item $\langle\gamma\rangle \langle\sigma\rangle ^{-1}$ where $h(\gamma)=h(\sigma)$ and $\mathrm{l}(\gamma)\neq \mathrm{l}(\sigma)$;
\item $\langle\gamma\rangle ^{-1}\langle\sigma\rangle ^{-1}$ where $t(\gamma)=h(\sigma)$ and $\mathrm{f}(\gamma)\mathrm{l}(\sigma)=0$;
\item $\langle\gamma\rangle ^{-1}\langle\sigma\rangle $ where $t(\gamma)=t(\sigma)$ and $\mathrm{f}(\gamma)\neq \mathrm{f}(\sigma)$; or
\item $\langle\gamma\rangle \langle\sigma\rangle $ where $h(\gamma)=t(\sigma)$ and $\mathrm{f}(\sigma)\mathrm{l}(\gamma)=0$.
\end{enumerate}
For each $i\in I$ there is an \emph{associated vertex} $v_{\mathcal{C}}(i)$
defined by $v_{\mathcal{C}}(i)=h(\mathcal{C}_{i})$ for $i\leq0$ and $v_{\mathcal{C}}(i)=t(\mathcal{C}_{i})$
for $i>0$ provided $I\neq\{0\}$, and $v_{\langle 1 _{v,\pm1}\rangle}(0)=v$ otherwise.

The \emph{inverse} $\mathcal{C}^{-1} $ of a generalised $I$-word $\mathcal{C}$ is defined
by $\langle1 _{v,\delta}\rangle ^{-1}= \langle1 _{v,-\delta}\rangle $ if $I=\{0\}$,
and otherwise inverting the generalised letters and reversing their order; analogously to Definition \ref{definition:words-for-modules}. Note the generalised $\mathbb{Z}$-words are indexed so
that
\[
\left(\dots \mathcal{C}_{-1}\mathcal{C}_{0}\mid  \mathcal{C}_{1}\mathcal{C}_{2}\dots\right)^{-1}=\dots \mathcal{C}_{2}^{-1}\mathcal{C}_{1}^{-1}\mid  \mathcal{C}^{-1}_{0}\mathcal{C}^{-1}_{-1}\dots
\]
Let $d\in\mathbb{Z}$. If $I=\mathbb{Z}$ we define the generalised $\mathbb{Z}$-word $\mathcal{C}[d]$ by $\mathcal{C}[d]_{i}=\mathcal{C}_{i+d}$ for all $i\in\mathbb{Z}$. If $I\neq\mathbb{Z}$ let $\mathcal{C}=\mathcal{C}[d]$.
\end{definition}
The angled brackets are meant to help the reader distinguish between (sequences of consecutive arrows which make up a path in $\mathbf{P}$) and (generalised letters which are all direct or all inverse). 
\begin{definition}\label{definition:sign-and-composing-generalised-words}
Recall, from Definition \ref{definition:sign-and-composing-words}, that for any letter $l$ we chose $s(l)\in\{1,-1\}$ (such that if $l\neq l'$, $h(l)=h(l')$ and $s(l)=s(l')$ then $\{l,l'\}=\{x^{-1},y\}$ where $xy\in\mathscr{I}$). We define the \emph{sign} of any generalised letter by setting $s(\langle \gamma \rangle)=s(\mathrm{f}(\gamma)^{-1})$ and $s(\langle \gamma \rangle^{-1})=-s(\mathrm{l}(\gamma))$ for each $\gamma\in\mathbf{P}$.

Let $\mathcal{C}$ be a generalised $I$-word with $I\subseteq\mathbb{N}$. We define the \emph{sign} of $\mathcal{C}$ by $s(\langle 1_{v,\delta}\rangle)=\delta$ when $I=\{0\}$ and $s(\mathcal{C})=s(\mathcal{C}_{1})$ otherwise. Now let $\mathcal{C}$ be a generalised (finite or $-\mathbb{N}$)-word and $\mathcal{D}$ be a generalised (finite or $\mathbb{N}$)-word. The \emph{composition of} $\mathcal{C}$ \emph{and} $\mathcal{D}$ is given by concatenating the letters together, written $\mathcal{CD}$. We say that $\mathcal{C}$ and $\mathcal{D}$ are \emph{composable} if $h(\mathcal{C}^{-1})=h(\mathcal{D})$ and $s(\mathcal{C}^{-1})=-s(\mathcal{D})$. 
\end{definition}
\begin{remark}\label{remark:generalised-sign-and-composability}
Let $\mathcal{C}$ be a (finite or $-\mathbb{N}$)-word and $\mathcal{D}$ be a (finite or $\mathbb{N}$)-word, and write $\mathcal{D}=\mathcal{D}_{1}\dots$ and $\mathcal{C}=\dots \mathcal{C}_{0}$. Analogously to Example \ref{remark:sign-and-composability},  $\mathcal{C}$ and $\mathcal{D}$ are composable if and only if $\mathcal{CD}$ is a word; see \cite[Proposition 2.1.13]{Ben2018} for details.
\end{remark}
\begin{example}\label{example:generalised-signs-for-k[[x,y]]/(xy)}
We continue with Example \ref{example:signs-for-k[[x,y]]/(xy)} where $\Lambda=k[[x,y]]/(xy)$ and we chose $s(x)=1=s(y^{-1})$ and $s(x^{-1})=-1=s(y)$.  Let $\Lambda=k[[x,y]]/(xy)$. Here $\mathbf{P}=\{x^{n},y^{m}\colon 0<n,m\in \mathbb{N}\}$. This means $s(\langle x^{n} \rangle)=s(\langle x^{n} \rangle^{-1})=-1$ and $s(\langle y^{n}\rangle)=s(\langle y^{n}\rangle^{-1})= 1$ for any $n>0$. Fix a non-trivial generalised $I$-word $\mathcal{C}=\dots \mathcal{C}_{i}\dots$ and suppose $i,i\pm 1\in I$. Then 
\[
\mathcal{C}_{i}\mathcal{C}_{i+1}\in\{ \langle x^{n}\rangle^{\delta}\langle y^{m} \rangle^{\varepsilon}, \langle y^{m}\rangle^{\varepsilon}\langle x^{n} \rangle^{\delta}\colon 0<n,m\in \mathbb{N},\delta,\varepsilon \in\{1,-1\}\}.
\]
Conversely, any generalised letter  has the form $\langle x^{n}\rangle ^{\delta}$ or $\langle y^{m}\rangle^{\varepsilon}$, and a sequence of generalised letters is a generalised word if and only if it does not contain a pair of consecutive generalised letters of the form $\langle x^{n}\rangle^{\delta}\langle x^{m}\rangle^{\epsilon}$ or $\langle y^{n}\rangle^{\delta}\langle y^{m}\rangle^{\epsilon}$. For example
\[
\begin{array}{c}
     \mathcal{C}=\dots\langle x\rangle ^{-1}\langle y\rangle ^{-1}\langle x\rangle ^{-1}\langle y\rangle ^{-1}\langle x\rangle ^{-1}\langle y\rangle ^{-1}\langle x^{4}\rangle ^{-1} \langle y^{3}\rangle \langle x^{2}\rangle ^{-1}\langle y\rangle \langle x^{2}\rangle ^{-1}\langle y^{3}\rangle \langle x^{4}\rangle ^{-1} 
\end{array}
\]
is a generalised $-\mathbb{N}$-word. If $\mathcal{D}=\dots \mathcal{D}_{0}$ is a non-trivial (finite or $-\mathbb{N}$)-word and $\mathcal{E}=\mathcal{E}_{1}\dots$ is a non-trivial (finite or $\mathbb{N}$)-word, then $\mathcal{D}$ and $\mathcal{E}$ are composable if and only if $(\mathcal{D}_{0},\mathcal{E}_{1})$ is one of the following
\[
\begin{array}{c}
(\langle x^{n}\rangle,\langle y^{m}\rangle),(\langle x^{n}\rangle,\langle y^{m}\rangle^{-1}),(\langle y^{n}\rangle,\langle x^{m}\rangle),(\langle y^{n}\rangle,\langle x^{m}\rangle^{-1})\\
(\langle x^{n}\rangle^{-1},\langle y^{m}\rangle),(\langle x^{n}\rangle^{-1},\langle y^{m}\rangle^{-1}),(\langle y^{n}\rangle^{-1},\langle x^{m}\rangle),(\langle y^{n}\rangle^{-1},\langle x^{m}\rangle^{-1}),
\end{array}
\]
for some $n,m>0$. For example, if we let $\mathcal{B}=\langle x^{4}\rangle\langle y\rangle \langle x\rangle \langle y\rangle \langle x\rangle\dots$ and $\mathcal{A}=\langle y^{3}\rangle \langle x^{2}\rangle ^{-1}\langle y\rangle \langle x^{2}\rangle ^{-1}\langle y^{3}\rangle \langle x^{4}\rangle ^{-1}$ then $\mathcal{B}^{-1}$ and $\mathcal{A}$ are composable with composition $\mathcal{C}=\mathcal{B}^{-1}\mathcal{A}$.
\end{example}
\subsection{String complexes.}
\begin{definition}\label{definition:string-complexes}\cite[Definition 2]{BekMer2003} Let $I$ be one of $\{0,\dots,m\}$, $\mathbb{N}$, $-\mathbb{N}$ or $\mathbb{Z}$ and $\mathcal{C}$ be a generalised $I$-word. To define a complex $P(\mathcal{C})$ we recall a map $\mathscr{H}_{\mathcal{C}}\colon I\to\mathbb{Z}$ which indexes the homogeneous degree of $P(\mathcal{C})$. For each $\gamma\in\mathbf{P}$ let $\mathscr{H}(\langle\gamma\rangle )=-1$
and $\mathscr{H}(\langle\gamma\rangle ^{-1})=1$.  Now let
\[
\mathscr{H}_{\mathcal{C}}(i)=
\begin{cases}
\mathscr{H}(\mathcal{C}_{1})+\dots+\mathscr{H}(\mathcal{C}_{i}) & (\text{if }i>0)\\
0 & (\text{if }i=0)\\
-(\mathscr{H}(\mathcal{C}_{0})+\dots+\mathscr{H}(\mathcal{C}_{i+1})) & (\text{if }i<0)
\end{cases}
\]
For $n\in\mathbb{Z}$ let $P^{n}(\mathcal{C})$ be the sum $\bigoplus\Lambda e_{v_{\mathcal{C}}(i)}$ over $i\in \mathscr{H}_{\mathcal{C}}^{-1}(n)$. For each $i\in I$ let $ g_{\mathcal{C},i}$ denote the coset of $e_{v_{\mathcal{C}}(i)}$
in $P(\mathcal{C})$ in degree $\mathscr{H}_{\mathcal{C}}(i)$. Define $d_{P(\mathcal{C})}$ by extending $g_{\mathcal{C},i}\mapsto g_{\mathcal{C},i}^{-}+ g_{\mathcal{C},i}^{+}$
linearly over $\Lambda$ for each $i\in I$, where:
\begin{enumerate}
\item $g_{\mathcal{C},i}^{+}= \mu bg_{\mathcal{C},i+1}$ if $i+1\in I$ and $\mathcal{C}_{i+1}=\langle \mu\rangle ^{-1}$, and $g_{\mathcal{C},i}^{+}=0$ otherwise; and
\item $g_{\mathcal{C},i}^{-}=\eta g _{\mathcal{C},i-1}$ if $i-1\in I$ and $\mathcal{C}_{i}=\langle\eta\rangle $, and $g_{\mathcal{C},i}^{-}=0$ otherwise.
\end{enumerate}
\end{definition}
\begin{example}\label{example:string-complex-for-k[[x,y]]/(xy)}
We continue with Example \ref{example:generalised-signs-for-k[[x,y]]/(xy)} where $\Lambda=k[[x,y]]/(xy)$ and we considered the generalised words $\mathcal{A}$, $\mathcal{B}$ and $\mathcal{C}=\mathcal{B}^{-1}\mathcal{A}$. By definition we have $ g_{\mathcal{C},i}\mapsto 0$ when $i=0,-2,-4,-6$, $ g_{\mathcal{C},-1}\mapsto y^{3} g_{\mathcal{C},-2}+x^{4} g_{\mathcal{C},0}$, $ g_{\mathcal{C},-3}\mapsto y g_{\mathcal{C},-4}+x^{2} g_{\mathcal{C},-2}$, $ g_{\mathcal{C},-5}\mapsto y^{3} g_{\mathcal{C},-6}+x^{2} g_{\mathcal{C},-4}$, $ g_{\mathcal{C},-7}\mapsto x^{4} g_{\mathcal{C},-6}$, $ g_{\mathcal{C},i}\mapsto y g_{\mathcal{C},i+1}$ when $i<-7$ is even, and $ g_{\mathcal{C},i}\mapsto x g_{\mathcal{C},i+1}$ when $i<-8$ is odd. The function $\mathscr{H}_{\mathcal{C}}\colon -\mathbb{N}\to \mathbb{Z}$ is given by 
\[(\mathscr{H}_{\mathcal{C}}(i)\colon i\in -\mathbb{N})=(\mathscr{H}_{\mathcal{C}}(0),\mathscr{H}_{\mathcal{C}}(-1),\dots)=(0,-1,0,-1,0,-1,0,-1,-2,-3,\dots).
\]
In this case the complex $P(\mathcal{C})$, which lies strictly in all non-positive homogeneous degrees, is given below.
\[
\begin{tikzcd}[ampersand replacement=\&]
    \cdots\arrow{r}{x}\& \Lambda\arrow{r}{y}\& \Lambda\arrow{r}{x}\& \Lambda\arrow{r}{\begin{psmallmatrix}
    y  \\
    0 \\
    0  \\
    0
    \end{psmallmatrix}}\& \Lambda^{4}\arrow{rrr}{\begin{psmallmatrix}
    x^{4} & y^{3} & 0 & 0 \\
    0 & x^{2} & y & 0 \\
    0 & 0 & x^{2} & y^{3} \\
    0 & 0 & 0 & x^{4}
    \end{psmallmatrix}}\&\&\& \Lambda^{4}\arrow{r}{d^{0}_{P(\mathcal{C})}}\& 0\arrow{r}{d^{1}_{P(\mathcal{C})}}\& \cdots       
    \end{tikzcd}
\]
We depict $P(\mathcal{C})$ as follows, where the differential $d_{P(\mathcal{C})}$ is considered to be pointing downwards.  
\[
\begin{tikzcd}[column sep=0.3cm, row sep=0.4cm]
    \vdots\arrow{d} & & \ddots\arrow[swap]{dr}{x} &  & & & & & & & & & & & \\
    P^{-4}(\mathcal{C})\arrow{d}{d^{-4}_{P(\mathcal{C})}}\arrow[dashed,-]{rrr} & & & \Lambda\arrow[swap]{dr}{y}\arrow[dashed,-]{rrrrrrrrrr} & & & & & & & & & & \, & \\
    P^{-3}(\mathcal{C})\arrow{d}{d^{-3}_{P(\mathcal{C})}}\arrow[dashed,-]{rrrr} & &  & & \Lambda\arrow[swap]{dr}{x}\arrow[dashed,-]{rrrrrrrrr} & & & & & & & & & \, & \\
    P^{-2}(\mathcal{C})\arrow{d}{d^{-2}_{P(\mathcal{C})}}\arrow[dashed,-]{rrrrr} & & & & & \Lambda\arrow[swap]{dr}{y}\arrow[dashed,-]{rrrrrrrr} & & & & & & & & \, & \\
    P^{-1}(\mathcal{C})\arrow{d}{d^{-1}_{P(\mathcal{C})}}\arrow[dashed,-]{rrrrrr} & & & & & & \Lambda\arrow[swap]{dr}{x^{4}}\arrow[dashed,-]{rr} & & \Lambda\arrow[swap]{dl}{y^{3}}\arrow[swap]{dr}{x^{2}}\arrow[dashed,-]{rr} & &  \Lambda\arrow[swap]{dl}{y}\arrow[swap]{dr}{x^{2}}\arrow[dashed,-]{rr} & & \Lambda\arrow[swap]{dl}{y^{3}}\arrow[swap]{dr}{x^{4}}\arrow[dashed,-]{r} & \, & \\
    P^{0}(\mathcal{C})\arrow[dashed,-]{rrrrrrr} & & & & & & & \Lambda\arrow[dashed,-]{rr} & & \Lambda\arrow[dashed,-]{rr} & & \Lambda\arrow[dashed,-]{rr} & & \Lambda 
    \end{tikzcd}
\]
In this picture one uses the dashed horizontal lines to determine the direct summands of each homogeneous component $P^{n}(\mathcal{C})$. To determine the differential $d_{P(\mathcal{C})}$ one (likewise) uses the diagonal downwards arrows to form a matrix whose entries are elements $\gamma\in\mathbf{P}$. In general, each of these arrows should be considered as the map $\Lambda e_{h(\gamma)}\to \Lambda e_{t(\gamma)}$ given by right multiplication by $\gamma$. 
\end{example}
Lemma \ref{lemma:isos-between-string-complexes} describes some isomorphisms between complexes of the form $P(\mathcal{C})$.
\begin{lemma}\emph{\cite[Lemma 3.4]{Ben2016}}\label{lemma:isos-between-string-complexes} Let $\mathcal{C}$ be a generalised $I$-word.
\begin{enumerate}
\item If $I=\{0,\dots,m\}$ then there is an isomorphism of complexes
\[
P(\mathcal{C}^{-1})\rightarrow P(\mathcal{C})[\mathscr{H}_{\mathcal{C}}(m)],\, g_{\mathcal{C}^{-1},i}\mapsto  g_{\mathcal{C},m-i}.
\]
\item If $I$ is infinite then there is an isomorphism of complexes
\[
P(\mathcal{C}^{-1})\rightarrow P(\mathcal{C}),\, g_{\mathcal{C}^{-1},i}\mapsto  g_{\mathcal{C},-i}.
\]
\item If $I=\mathbb{Z}$ and $t\in\mathbb{Z}$ then  there is an isomorphism of complexes
\[
P(\mathcal{C}[t])\rightarrow P(\mathcal{C})[\mathscr{H}_{\mathcal{C}}(t)],\, g_{\mathcal{C}[t],i}\mapsto  g_{\mathcal{C},i+t}.
\]
\end{enumerate}
\end{lemma}
Lemma \ref{lemma:string-complex-technical} is a technical result which we apply in \S\ref{section:Resolutions-and-string-and-band-complexes}.
\begin{lemma}\label{lemma:string-complex-technical}
Let $\mathcal{C}$ be a generalised $I$-word, fix $i\in I$ and let $\mathscr{H}_{\mathcal{C}}(i)=t$ and $L_{j}=\Lambda( g_{\mathcal{C},j}^{+}+ g_{\mathcal{C},j}^{-})$ for all $j\in \mathscr{H}_{\mathcal{C}}^{-1}(t)$. Let $M=M'\oplus M''$ where 
\[
M'=\begin{cases}
\Lambda \mathrm{f}(\mu) g_{\mathcal{C},i-1}& (\text{if }\mathcal{C}_{i}=\langle\mu\rangle)\\
0 & (\text{otherwise}) 
\end{cases}
M''=\begin{cases}
\Lambda \mathrm{f}(\eta) g_{\mathcal{C},i+1} & (\text{if }\mathcal{C}_{i+1}=\langle\eta\rangle^{-1})\\
0 & (\text{otherwise}).
\end{cases}
\]
Suppose $M\subseteq \sum_{j}L_{j}$. Then $M\subseteq L_{i}$ and the following statements hold.
\begin{enumerate}
    \item If $\mathcal{C}_{i}=\langle\mu\rangle$ then $\mu\in\mathbf{A}$ and $\mathcal{C}_{i+1}\neq\langle\eta\rangle^{-1}$ for all $\eta\in\mathbf{P}$.
    \item If $\mathcal{C}_{i+1}=\langle\eta\rangle^{-1}$ then $\eta\in\mathbf{A}$ and $\mathcal{C}_{i}\neq\langle\mu\rangle$ for all $\mu\in\mathbf{P}$.
\end{enumerate}
\end{lemma}
\begin{proof}
Let $m\in M$, say where $m=\lambda'\mathrm{f}(\mu) g_{\mathcal{C},i-1}+\lambda''\mathrm{f}(\eta) g_{\mathcal{C},i+1}$ for some $\lambda',\lambda''\in \Lambda$ where ($\mathcal{C}_{i}\neq \langle\mu\rangle$ implies $\lambda'=0$) and ($\mathcal{C}_{i+1}\neq\langle\eta\rangle^{-1}$ implies $\lambda''=0$). Let $L=\sum L_{j}$. Assuming $M\subseteq L$ we also have $m=\sum_{j}l_{j}$ where $l_{j}=\lambda_{j}( g_{\mathcal{C},j}^{+}+ g_{\mathcal{C},j}^{-})$ for some $\lambda_{j}\in\Lambda$. We claim $\sum_{j\neq i}l_{j}=0$. For any $j\neq i$ we have $ g_{\mathcal{C},j}^{+}\in \Lambda  g_{\mathcal{C},i-1}$ if and only if ($j=i-2$ and $\mathcal{C}_{i-1}=\langle\sigma\rangle^{-1}$) in which case $ g_{\mathcal{C},j}^{+}=\sigma  g_{\mathcal{C},i-1}$. Likewise, $ g_{\mathcal{C},j}^{-}\in \Lambda  g_{\mathcal{C},i+1}$ if and only if $j=i+2$ and $\mathcal{C}_{i+2}=\langle\gamma\rangle$), in which case $ g_{\mathcal{C},j}^{-}=\gamma  g_{\mathcal{C},i+1}$. Without loss of generality we assume $i\pm 2\in I$,  $\mathcal{C}_{i-1}=\langle\sigma\rangle^{-1}$ and $\mathcal{C}_{i+2}=\langle\gamma\rangle$, so
\[
\begin{array}{c}
\lambda'\mathrm{f}(\mu) g_{\mathcal{C},i-1}+\lambda''\mathrm{f}(\eta) g_{\mathcal{C},i+1}= m
   =\lambda_{i-2}\sigma  g_{\mathcal{C},i-1} + \lambda_{i}(\mu  g_{\mathcal{C},i-1} + \eta  g_{\mathcal{C},i+1}) + \lambda_{i+2}\gamma  g_{\mathcal{C},i+1} +m'
\end{array}
\]
where
\[
\begin{array}{c}
m'=(\lambda_{i-2}b^{-}_{i-2}+\lambda_{i+2} g_{\mathcal{C},i+2}^{+}) +\sum_{h\in \mathscr{H}_{\mathcal{C}}^{-1}(t+1),\,h<i-2,\,h>i+2}l_{h}.
\end{array}
\]
Considering that $m\in \bigoplus\Lambda  g_{\mathcal{C},h}$ over all $h\in \mathscr{H}_{\mathcal{C}}^{-1}(t+1)$, we have that $m'=0$ and
\[
\lambda'\mathrm{f}(\mu)-\lambda_{i}\mu - \lambda_{i-2}\sigma=0,\,\lambda''\mathrm{f}(\eta)-\lambda_{i}\eta-\lambda_{i+2}\gamma=0.
\]
Since $\Lambda \sigma\cap \Lambda \mathrm{f}(\mu)=0$ and $\Lambda \mathrm{f}(\eta)\cap\Lambda\gamma=0$ by condition (6) of Definition \ref{definition:complete-gentle-algebra}, we have that $\lambda_{i-2}\sigma=0=\lambda_{i+2}\gamma$. Hence $M\subseteq L_{i}$. We now show (1) holds; the proof that (2) holds is similar, and omitted. 

So we assume $\mathcal{C}_{i}=\langle\mu\rangle$. Since $M\subseteq L_{i}$ we have that $\mathrm{f}(\mu) g_{\mathcal{C},i-1}=\lambda( g_{\mathcal{C},i}^{+}+ g_{\mathcal{C},i}^{-})$ and for some $\lambda\in\Lambda e_{v_{\mathcal{C}}(i)}$, which means $\mathrm{f}(\mu) g_{\mathcal{C},i-1}=\lambda\mu  g_{\mathcal{C},i-1}$ since $\Lambda  g_{\mathcal{C},i-1}\cap \Lambda  g_{\mathcal{C},i+1}=0$. This must mean that $\mathrm{f}(\mu)=\lambda\mu $, and hence $\Lambda\mathrm{f}(\mu)\subseteq \Lambda \mu$, and hence $\Lambda\mathrm{f}(\mu)= \Lambda \mu$. By Lemma \ref{lemma:qbsb-uniserial-cyclics-given-by-arrows} this gives $\mu=\mathrm{f}(\mu)$, and so $\mu\in\mathbf{A}$. To finish the proof, we now suppose $\mathcal{C}_{i} \mathcal{C}_{i+1}=\langle \mu\rangle \langle \eta\rangle ^{-1}$ for a contradiction. In this case $\mathrm{f}(\mu) g_{\mathcal{C},i-1}=\lambda( g_{\mathcal{C},i}^{+}+ g_{\mathcal{C},i}^{-})$ also gives $\lambda\eta  g_{\mathcal{C},i-1}=0$, and so $\lambda\in\mathrm{rad}(\Lambda)$ by Lemma \ref{lemma:technical-comp-gen-props}(1). But then the equation $\mathrm{f}(\mu)=\lambda\mu$ gives $\Lambda \mu\subseteq \mathrm{rad}(\Lambda\mu)$ which gives the contradiction $\Lambda\mu=0$ (and so $\mu\notin\mathbf{P}$) by Nakayama's lemma.
\end{proof}
\subsection{Band complexes.}
\begin{definition}\label{definition:cyclic-generalised-words}
 If $n\neq 0$ a generalised $\{0,\dots,n\}$-word $\mathcal{A}=\mathcal{A}_{1}\dots \mathcal{A}_{n}$ is called \emph{weak}-\emph{cyclic} if $\mathcal{A}^{2}$ is a generalised word and $\mathcal{A}$ is not a non-trivial power of another  generalised word. By Remark \ref{remark:generalised-sign-and-composability}, $\mathcal{A}^{2}$ is a generalised word if and only if $t(\mathcal{A}_{n})=h(\mathcal{A}_{1})$ and $s(\mathcal{A}_{n}^{-1})=-s(\mathcal{A}_{1})$. If additionally $\mathscr{H}_{\mathcal{A}}(n)=0$ we say $\mathcal{A}$ is \emph{cyclic}. If $\mathcal{A}$ is weak-cyclic, let: ${}^{\infty\hspace{-0.5ex}}\mathcal{A}=\dots \mathcal{A}\mathcal{A}$, a generalised $-\mathbb{N}$-word; $\mathcal{A}^{\infty}= \mathcal{A}\mathcal{A}\dots$, a generalised $\mathbb{N}$-word; and ${}^{\infty\hspace{-0.5ex}}\mathcal{A}^{\infty}=\dots \mathcal{A}\mid \mathcal{A}\mathcal{A}\dots$, a generalised $\mathbb{Z}$-word. 

We say that a generalised $\mathbb{Z}$-word $\mathcal{C}$ is \emph{periodic}
if $\mathcal{C}=\mathcal{C}[p]$ and $\mathscr{H}_{\mathcal{C}}(p)=0$ for some $p>0$. In this case the minimal such $p$ is the \emph{period} of $\mathcal{C}$, and we say $\mathcal{C}$ is $p$-\emph{periodic}. We say $ \mathcal{C}$ is \emph{aperiodic} if $\mathcal{C}$ is not periodic. In case $\mathcal{C}$ is $p$-periodic the generalised $\{0,\dots,p\}$-word we call $\mathcal{A}=\mathcal{C}_{1}\dots\mathcal{C}_{p}$ the \emph{cycle} of $\mathcal{C}$ (which is cyclic). To denote that a generalised word $\mathcal{A}$ is the cycle of a periodic generalised word $\mathcal{C}$ we write $\mathcal{C}={}^{\infty\hspace{-0.5ex}}\mathcal{A}^{\infty}$.  
\end{definition}
\begin{remark}
By the minimality of the period, the cycle of a periodic generalised word is a cyclic generalised word. Conversely, if $\mathcal{A}$ is a cyclic generalised $\{0,\dots,p\}$-word then ${}^{\infty\hspace{-0.5ex}}\mathcal{A}^{\infty}$ is a $p$-periodic generalised $\mathbb{Z}$-word.
\end{remark}
\begin{definition}
\cite[Definition 3.5]{Ben2016} Let $\mathcal{C}$ be a $p$-periodic generalised word. By Lemma \ref{lemma:isos-between-string-complexes}(3) $P^{n}(\mathcal{C})$ is a $\Lambda\text{-}R[T,T^{-1}]$-bimodule
where $T g_{\mathcal{C},i}= g_{\mathcal{C},i-p}$. By periodicity the map $d_{P(\mathcal{C})}^{n}\colon P^{n}(\mathcal{C})\rightarrow P^{n+1}(\mathcal{C})$ is $\Lambda\otimes_{R}R[T,T^{-1}]$-linear. For any object $V$ of $R[T,T^{-1}]\text{-\textbf{Mod}}_{R\text{-\textbf{Proj}}}$ let $P^{n}(\mathcal{C},V)=P^{n}(\mathcal{C})\otimes_{R[T,T^{-1}]}V$
and $d_{P(\mathcal{C},V)}^{n}=d_{P(\mathcal{C})}^{n}\otimes1_{V}$ for each $n\in\mathbb{Z}$. By \cite[Lemma 3.6]{Ben2016} $P(\mathcal{C},V)$ is a complex of projective modules.
\end{definition}
\begin{example}\label{example:band-complexes-part-1-for-2-by-2-p-adics}
We continue with Example \ref{2-by-2-matrices-cyclic-words}, where $\Lambda \simeq \widehat{\mathbb{Z}}_{p}Q/\langle a^{2}, b^{2}, a b+ b a-p\rangle$. Let
\[
\mathcal{A}=\langle a b a\rangle\langle b a\rangle^{-1}\langle b a b\rangle\langle a\rangle^{-1}\langle b\rangle\langle a b\rangle^{-1}
\]
We now check $\mathcal{A}$ is cyclic, giving an example of a $6$-periodic word $\mathcal{C}={}^{\infty\hspace{-0.5ex}}\mathcal{A}^{\infty}$. By definition we have $(\mu_{\mathcal{A}}(0),\dots,\mu_{\mathcal{A}}(6))=(0,-1,0,-1,0,-1,0)$. That $\mathcal{A}^{2}$ is a generalised word follows from the fact that $\mathrm{f}( a b a)= a\neq b=\mathrm{f}( a b)$. By observation, clearly $\mathcal{A}\neq \mathcal{B}^{n}$ for all generalised words $\mathcal{B}$ and all $n>1$. Similary to Example \ref{example:string-complex-for-k[[x,y]]/(xy)} we depict the complex $P(\mathcal{C})$ as follows, where we indicate the action of $T$ on each of $P^{-1}(\mathcal{C})$ and $P^{0}(\mathcal{C})$. 
\[
\begin{tikzcd}[column sep=0.3cm, row sep=0.4cm]
    P^{-1}(\mathcal{C})\arrow{d}{d^{-1}_{P(\mathcal{C})}}\arrow[dashed,-]{rr} & & \cdots\arrow[swap]{dr}{ a b}\arrow[dashed,-]{rr} & & \Lambda\arrow[swap]{dl}{ a b a}\arrow[swap]{dr}{ b a}\arrow[dashed,-]{rr} & &  \Lambda\arrow[swap]{dl}{ b a b}\arrow[swap]{dr}{ a}\arrow[dashed,-]{rr} & & \Lambda\arrow[swap]{dl}{ b}\arrow[swap]{dr}{ a b}\arrow[dashed,-]{rr} & &  \Lambda\arrow[dotted,bend right=20,swap]{llllll}{T}\arrow[swap]{dl}{ a b a}\arrow[swap]{dr}{ b a}\arrow[dashed,-]{rr} & & \Lambda\arrow[dotted,bend right=20,swap]{llllll}{T}\arrow[swap]{dl}{ b a b}\arrow[swap]{dr}{ a}\arrow[dashed,-]{rr} & & \Lambda\arrow[dotted,bend right=20,swap]{llllll}{T}\arrow[swap]{dl}{ b}\arrow[swap]{dr}{ a b}\arrow[dashed,-]{rr} & &  \Lambda\arrow[dotted,bend right=20,swap]{llllll}{T}\arrow[swap]{dl}{ a b a}\arrow[dashed,-]{r} & \cdots &\\
    P^{0}(\mathcal{C})\arrow[dashed,-]{rr} & & \cdots\arrow[dashed,-]{r} & \Lambda\arrow[dashed,-]{rr} & & \Lambda\arrow[dashed,-]{rr} & & \Lambda\arrow[dashed,-]{rr} & & \Lambda\arrow[dashed,-]{rr}\arrow[dotted,bend left=20]{llllll}{T} & & \Lambda\arrow[dashed,-]{rr}\arrow[dotted,bend left=20]{llllll}{T} & & \Lambda\arrow[dashed,-]{rr}\arrow[dotted,bend left=20]{llllll}{T} & & \Lambda\arrow[dashed,-]{rr}\arrow[dotted,bend left=20]{llllll}{T} & & \cdots
    \end{tikzcd}
\]
\end{example}
\begin{definition}\label{definition:band-pres}\cite[Definition 2]{BekMer2003} Let $\mathcal{E}$ be a cyclic generalised $\{0,\dots,p\}$-word. Let $V$ be an object of $R[T,T^{-1}]\text{-\textbf{Mod}}_{R\text{-\textbf{Proj}}}$ with free $R$-basis $(v_{\omega}\colon\omega\in\Omega)$. If $n\in\mathbb{Z}$ let $\left\langle n,p\right\rangle =\mathscr{H}_{\mathcal{C}}^{-1}(n)\cap[0,p-1]$. Let $T^{\pm1}v_{\omega}=\sum_{\tau}a^{\pm}_{\tau\omega}v_{\tau}$ for some $a^{\pm}_{\tau\omega}\in R$. Fix $n\in\mathbb{Z}$. Let $P^{n}(\mathcal{E},V)$ be the $\Lambda$-module $\bigoplus_{\Omega}\bigoplus_{i}\Lambda  g_{\mathcal{E},i}$ where $i$ runs through $\left\langle n,p\right\rangle$. Let  $g_{\mathcal{E},i,\omega}$ denote the copy of $ g_{\mathcal{E},i}$ in the summand indexed by ($i$ and) $\omega\in\Omega$. 

Define  $d_{P(\mathcal{E},V)}^{n}:P^{n}(\mathcal{E},V)\to P^{n+1}(\mathcal{E},V)$ by sending $g_{\mathcal{E},i,\omega}\mapsto g_{\mathcal{E},i,\omega}^{+}+g_{\mathcal{E},i,\omega}^{-}$ where 
\[
 \begin{array}{c}
 g_{\mathcal{E},i,\omega}^{+}=\begin{Bmatrix}\mu c _{i+1,\omega} & (\mbox{if }i<p-1,\,\mathcal{E}_{i+1}=\langle\mu\rangle^{-1})\\
\mu(\sum_{\tau}a^{-}_{\tau\omega} g_{\mathcal{E},0,\tau}) &  (\mbox{if }i=p-1,\,\mathcal{E}_{p}=\langle\mu\rangle^{-1})\\
0 & (\mbox{otherwise})
\end{Bmatrix}\\
\\
  g_{\mathcal{E},i,\omega}^{-}=\begin{Bmatrix}\eta c _{i-1,\omega} & (\mbox{if }j>0,\,\mathcal{E}_{i}=\langle\eta\rangle)\\
\eta(\sum_{\tau}a^{+}_{\tau\omega} g_{\mathcal{E},p-1,\tau}) &  (\mbox{if }i=0,\,\mathcal{E}_{p}=\langle\eta\rangle\\
0 & (\mbox{otherwise})
\end{Bmatrix}
\end{array}
\]
\end{definition}
\begin{remark}\label{remark:isomorphism-for-band-complex}It is straightforward to check that $P(\mathcal{E},V)$ defines a complex of projective modules, and that there is an isomorphism of complexes $P(\mathcal{C},V)\to P(\mathcal{E},V)$; see \cite[Lemma 3.6]{Ben2016} and \cite[Remark 3.8]{Ben2016} for details. 
\end{remark}
Lemma \ref{lemma:band-complex-technical} is a technical result used in \S\ref{section:Resolutions-and-string-and-band-complexes}
\begin{lemma}\label{lemma:band-complex-technical}
Let $V$ be an object of \emph{$R[T,T^{-1}]\text{-\textbf{Mod}}_{R\text{-\textbf{Proj}}}$} with $R$-basis $(v_{\omega}\colon\omega\in\Omega)$. Let $\mathcal{E}$ be a cyclic generalised $I=\{0,\dots,p\}$-word, fix $i\in I$ with $\mathcal{E}_{i} \mathcal{E}_{i+1}=\langle \mu\rangle \langle \eta\rangle ^{-1}$ for some $\mu,\eta\in\mathbf{P}$ \emph{(}where $\mathcal{E}_{p+1}=\mathcal{E}_{1}$\emph{)}. Let $\mathscr{H}_{\mathcal{E}}(i)=t$ and $L_{j}=\sum_{\omega}\Lambda(g_{\mathcal{E},j,\omega}^{+}+g_{\mathcal{E},j,\omega}^{-})$ for all $j\in \mathscr{H}_{\mathcal{E}}^{-1}(t)$ with $0\leq j\leq p-1$. Then for all $\tau\in\Omega$ we have
\[
\sum_{j}L_{j}\nsupseteq 
\begin{cases}
\Lambda \mathrm{f}(\mu)g_{\mathcal{E},i-1,\tau}\oplus\Lambda \mathrm{f}(\eta)g_{\mathcal{E},i+1,\tau} & (\text{if }0<i<p-1) \\
\Lambda \mathrm{f}(\mu)g_{\mathcal{E},p-1,\tau}\oplus\Lambda \mathrm{f}(\eta)g_{\mathcal{E},1,\tau} & (\text{if }i=0)\\
\Lambda \mathrm{f}(\mu)g_{\mathcal{E},p-2,\tau}\oplus\Lambda \mathrm{f}(\eta)g_{\mathcal{E},0,\tau} & (\text{if }i=p-1)
\end{cases}
\]
\end{lemma}
\begin{proof}
We assume $i=0$, since the case where $i=p-1$ is similar, and the case where $0<i<p-1$ is simpler. Let $M=\Lambda \mathrm{f}(\mu)g_{\mathcal{E},p-1,\tau}\oplus\Lambda \mathrm{f}(\eta)g_{\mathcal{E},1,\tau}$, and for a contradiction let us assume $M$ is contained in $\sum_{j}L_{j}$. As in the proof of Lemma \ref{lemma:string-complex-technical} we may assume, without loss of generality, that $\mathcal{E}_{p-1}=\langle\sigma\rangle^{-1}$ and $\mathcal{E}_{1}=\langle\gamma\rangle$. Now let $m\in M$. Similarly to the proof of Lemma \ref{lemma:string-complex-technical} we now have
\[
\begin{array}{c}
\lambda'\mathrm{f}(\mu)g_{\mathcal{E},p-1,\tau}+\lambda''\mathrm{f}(\eta)g_{\mathcal{E},1,\tau}=m=\sum_{j,\omega}\lambda_{j,\omega}(c^{-}_{j,\omega}+g_{\mathcal{E},j,\omega}^{+})\\
=\sum_{\omega}(\lambda_{p-2,\omega}\sigma g_{\mathcal{E},p-1,\omega}+\lambda_{0,\omega}(\mu(\sum_{\nu}a^{+}_{\nu\omega}g_{\mathcal{E},p-1,\nu})+\eta g_{\mathcal{E},1,\omega})+\lambda_{1}\gamma g_{\mathcal{E},1,\omega} +m'
\end{array}
\]
where $m'\in\bigoplus_{\omega,i\neq p-1,0,1}\Lambda g_{\mathcal{E},i,\omega}$ and $\lambda',\lambda'',\lambda _{j,\omega}\in\Lambda$. Since $m\in \Lambda g_{\mathcal{E},p-1,\tau}\oplus \Lambda g_{\mathcal{E},1,\tau}$ we have $m'=0$. Similarly, and as $R$ lies in the centre of $\Lambda$, this also gives
\[
\begin{array}{cc}
\lambda'\mathrm{f}(\mu)=\lambda_{p-2,\tau}\sigma+\lambda_{0,\tau}(\sum_{\nu}a_{\nu\tau}^{+})\mu, & \lambda''\mathrm{f}(\eta)=\lambda_{0,\tau}\eta+\lambda_{1,\tau}\gamma
\end{array}
\]
and hence $\lambda_{p-2,\tau}\sigma=0=\lambda_{1,\tau}\gamma$ since $\Lambda\mathrm{f}(\mu)\cap\Lambda\sigma=0=\Lambda\mathrm{f}(\eta)\cap\Lambda\gamma$. Now we have shown that $M\subseteq L_{0}$. In particular we have 
\[
\begin{array}{c}
\mathrm{f}(\mu)g_{\mathcal{E},p-1,\tau}=\sum_{\omega}\lambda_{\omega}((\sum_{\nu}a_{\nu\omega}^{+}\mu g_{\mathcal{E},p-1,\tau})+\eta g_{\mathcal{E},1,\omega})
\end{array}
\]
for some $\lambda_{\omega}\in\Lambda$. Since $\Lambda g_{\mathcal{E},p-1,\tau}\cap\Lambda g_{\mathcal{E},0,\omega}=0$ we have $\lambda_{\omega}\eta=0$ and so $\lambda_{\omega}\in\mathrm{rad}(\Lambda)$ for each $\omega$; see for example Lemma \ref{lemma:technical-comp-gen-props}(1). Hence $\mathrm{f}(\mu)g_{\mathcal{E},p-1,\tau}=\sum_{\omega,\nu}\lambda_{\omega}a_{\nu\omega}^{+}\mu g_{\mathcal{E},p-1,\omega}$ gives $\mathrm{f}(\mu)g_{\mathcal{E},p-1,\tau}=(\sum_{\nu}\lambda_{\tau}a_{\nu\tau}^{+})\mu g_{\mathcal{E},p-1,\tau}$, and so $\Lambda \mathrm{f}(\mu)$ lies in $ \mathrm{rad}(\Lambda\mathrm{f}(\mu))$. As in Lemma \ref{lemma:string-complex-technical}, this a contradiction by Nakayama's lemma.
\end{proof}
\section{Resolutions and string and band complexes}\label{section:Resolutions-and-string-and-band-complexes}
We now focus on complexes which arise as projective resolutions of finitely generated modules. In what follows we say that a chain complex  $P$ of projective $\Lambda$-modules is a \emph{resolution in degree }$d\in\mathbb{Z}$ if $P^{l}=0$ for all $l>d$ and $H^{m}(P)=0$ for all $m\neq d$. We say $P$ is \emph{point}-\emph{wise}-\emph{finite} provided $P^{h}$ is finitely generated for all $h\in\mathbb{Z}$. Our aim for \S \ref{section:Resolutions-and-string-and-band-complexes} is to characterise the generalised words which index those string and band complexes which are point-wise-finite resolutions.
\subsection{String complexes}
\begin{lemma}\label{lemma:string-complex-resolution-technical}
Let $\mathcal{C}$ be a generalised $I$-word such that $P(\mathcal{C})$ is a resolution in degree $d\in\mathbb{Z}$. If  $\mathcal{C}_{i} \mathcal{C}_{i+1}=\langle \mu\rangle \langle \eta\rangle ^{-1}$ where $\mu,\eta\in\mathbf{P}$ then $\mathscr{H}_{\mathcal{C}}(i)=d-1$.
\end{lemma}
\begin{proof}Let $\mathscr{H}_{\mathcal{C}}(i)=t$. Suppose $t\geq d$. By construction we have $\mathscr{H}_{\mathcal{C}}(i\pm 1)\geq d+1$ which contradicts the fact that $P^{l}(\mathcal{C})=0$ when $l>d$. Hence we must have $t< d$. For a contradiction assume $t\leq  d-2$. Let
\[
\begin{array}{ccc}
L=\mathrm{im}(d^{t}_{P(\mathcal{C})}), & M=\Lambda \mathrm{f}(\mu) g_{\mathcal{C},i-1}\oplus\Lambda \mathrm{f}(\eta) g_{\mathcal{C},i+1}, & N=\mathrm{ker}(d^{t+1}_{P(\mathcal{C})}).
\end{array}
\]
After considering different cases, by construction $M\subseteq N$. For example: when $\mathcal{C}_{i-1}=\langle\gamma\rangle$ we have $\mu\gamma=0$, and so $\mathrm{f}(\mu) g_{\mathcal{C},i-1}\in N$ by Lemma \ref{lemma:technical-comp-gen-props}(2); and when $\mathcal{C}_{i-1}=\langle\gamma\rangle^{-1}$ we have $ g_{\mathcal{C},i-1}\in N$. Since $t+1\leq d-1$ we have $N/L=H^{t+1}(P(\mathcal{C}))=0$ by assumption. Since $L=\sum L_{j}$ where $L_{j}=\Lambda( g_{\mathcal{C},j}^{+}+ g_{\mathcal{C},j}^{-})$ for all $j\in \mathscr{H}_{\mathcal{C}}^{-1}(t)$, the inclusion $M\subseteq L$ contradicts the second statement(s) of part (1) (and (2)) of Lemma \ref{lemma:string-complex-technical}.
\end{proof}
\begin{lemma}\label{lemma:string-complex-resolution-technical-II}
Let $\mathcal{C}$ be a generalised $I$-word such that $P(\mathcal{C})$ is a resolution in degree $d\in\mathbb{Z}$. If  $\mathcal{C}_{i} \mathcal{C}_{i+1}=\langle \mu\rangle^{-1} \langle \eta\rangle$ where $\mu,\eta\in\mathbf{P}$ then $\mathscr{H}_{\mathcal{C}}(i)=d$.
\end{lemma}
\begin{proof}
Let $\mathscr{H}_{\mathcal{C}}(i)=t$ and $P=P(\mathcal{C})$. By assumption we have that $H^{t}(P)\neq 0$, since $\mathrm{ker}(d_{P})$ contains the submodule $\Lambda  g_{\mathcal{C},i}$, which cannot be contained in the image since $\mathrm{im}(d_{P})$ is contained in $ \mathrm{rad}(P)$. 
\end{proof}
\begin{lemma}\label{lemma:string-complex-resolution-technical-III}
Let $\mathcal{C}$ be a generalised $I$-word such that $P(\mathcal{C})$ is a resolution in degree $d\in\mathbb{Z}$. If  $\mathcal{C}_{i} \mathcal{C}_{i+1}=\langle \mu\rangle ^{-1}\langle \eta\rangle ^{-1}$ where $\mu,\eta\in\mathbf{P}$ then\emph{:}
\begin{enumerate}
    \item we have $\mathscr{H}_{\mathcal{C}}(i)<d$ and $\mu\in\mathbf{A}$\emph{;}
    \item for all $j<i$ in $I$ there exists $ a\in\mathbf{A}$ with $\mathcal{C}_{j}=\langle  a\rangle ^{-1}$\emph{;}
    \item and for any $ b\in\mathbf{A}$ the composition $\langle b\rangle^{-1}\mathcal{C}$ is not a generalised word.
\end{enumerate}
\end{lemma}
\begin{proof}(1) Let $\mathscr{H}_{\mathcal{C}}(i)=t$. If $t\geq d$ then we contradict that $P^{d+1}(\mathcal{C})=0$, and so $t< d$. Now let
\[
\begin{array}{ccc}
L=\mathrm{im}(d^{t-1}_{P(\mathcal{C})}), & M=\Lambda \mathrm{f}(\mu) g_{\mathcal{C},i-1}, & N=\mathrm{ker}(d^{t}_{P(\mathcal{C})}).
\end{array}
\]
Note $M\subseteq N$ by construction, and since $t<d$ we have $N/L=H^{t}(P(\mathcal{C}))=0$ by assumption. This gives $M\subseteq L$, and hence $M\subseteq L_{i}$ and $\mu\in\mathbf{A}$ by Lemma \ref{lemma:string-complex-technical}. 

(2) For a contradiction we assume we can choose $j\in I$ maximal such that $j<i$ and $\mathcal{C}_{j}\neq \langle  a\rangle ^{-1}$ for all $ a\in\mathbf{A}$. Note that if $\mathcal{C}_{j}= \langle \sigma\rangle ^{-1}$ then we must have $\sigma \in\mathbf{A}$ by part (1), and so $\mathcal{C}_{j}= \langle \sigma\rangle$ for some $\sigma\in\mathbf{P}$. By the maximality of $j<i$ we have $\mathscr{H}_{\mathcal{C}}(j)<t\leq d-1$ and $\mathcal{C}_{j+1}=\langle\gamma\rangle^{-1}$ for some $\gamma\in\mathbf{P}$. That we have $\mathcal{C}_{j}\mathcal{C}_{j+1}=\langle \sigma\rangle\langle\gamma\rangle^{-1}$ and $\mathscr{H}_{\mathcal{C}}(j)<d-1$ contradicts Lemma \ref{lemma:string-complex-resolution-technical}.

(3) For a contradiction suppose there exists $ b\in\mathbf{A}$ such that $\langle b\rangle^{-1}\mathcal{C}$ is a generalised word. This means $I\subseteq \mathbb{N}$. By part (2) above we have that $\mathcal{C}_{1}=\langle a\rangle^{-1}$ for some $ a\in\mathbf{A}$. Since $\langle b\rangle^{-1}\langle a\rangle^{-1}$ is a generalised word we have $ b a=0$. By part (1) above we have $0<d$. Now let
\[
\begin{array}{ccc}
L=\mathrm{im}(d^{-1}_{P(\mathcal{C})}), & M=\Lambda  b  g_{\mathcal{C},0}, & N=\mathrm{ker}(d^{0}_{P(\mathcal{C})}).
\end{array}
\] 
Since $\mathcal{C}_{1}=\langle a\rangle^{-1}$ we must have $L\subseteq \bigoplus_{i>0}\Lambda  g_{\mathcal{C},i}$. Since $ b a=0$ we have $M\subseteq N$, and $L=N$ since $H^{0}(P(\mathcal{C}))=0$. Altogether $M\subseteq \bigoplus_{i>0}\Lambda  g_{\mathcal{C},i}$, and hence $M=0$, which contradicts that $ b\in\mathbf{P}$ and $t( b)=v_{\mathcal{C}}(0)$.
\end{proof}
The proof of Lemma \ref{lemma:string-complex-resolution-technical-IV} are similar to that of Lemma \ref{lemma:string-complex-resolution-technical-III}, and omitted. 
\begin{lemma}\label{lemma:string-complex-resolution-technical-IV}
Let $\mathcal{C}$ be a generalised $I$-word such that $P(\mathcal{C})$ is a resolution in degree $d\in\mathbb{Z}$. If  $\mathcal{C}_{i} \mathcal{C}_{i+1}=\langle \mu\rangle\langle \eta\rangle$ where $\mu,\eta\in\mathbf{P}$ then\emph{:}
\begin{enumerate}
    \item we have $\mathscr{H}_{\mathcal{C}}(i)<d$ and $\eta\in\mathbf{A}$\emph{;}
    \item for all $j>i+1$ in $I$ there exists $ a\in\mathbf{A}$ with $\mathcal{C}_{j}=\langle  a\rangle $\emph{;}
    \item and for any $d\in\mathbf{A}$ the composition $\mathcal{C}\langle d\rangle$ is not a generalised word.
\end{enumerate}
\end{lemma}
\begin{lemma}\label{lemma:string-complex-resolution-technical-VI}
Let $\mathcal{C}$ be a generalised $I$-word such that $P(\mathcal{C})$ is a resolution in degree $d\in\mathbb{Z}$. 
\begin{enumerate}
    \item If $I\neq-\mathbb{N},\mathbb{Z}$ and $\mathcal{C}_{1}$ is direct then $d=0$.
    \item If $I\neq\mathbb{N},\mathbb{Z}$, $m=\mathrm{max}(I)$ and $\mathcal{C}_{m}$ is inverse then $d=\mathscr{H}_{\mathcal{C}}(m)$.
\end{enumerate}
\end{lemma}
\begin{proof}
We only prove (1) holds; the proof that (2) holds is similar. Since $I\neq-\mathbb{N},\mathbb{Z}$ we have $-1\notin I$, and so by Definition \ref{definition:string-complexes} we have $ g_{\mathcal{C},0}^{-}=0$. Since $\mathcal{C}_{1}$ is direct we have  $ g_{\mathcal{C},0}^{+}=0$, which together means $\Lambda  g_{\mathcal{C},0}$ is contained in the kernel of the differential. Since the image lies in the radical we have $d=0$.
\end{proof}
Lemma \ref{lemma:string-complex-resolution-given-by-a-certain-form} is a technical result which simplifies the proof of Lemma \ref{lemma:string-complex-resolution-given-by-a-certain-form-II}.
\begin{lemma}\label{lemma:string-complex-resolution-given-by-a-certain-form}
Let $\mathcal{C}$ be a generalised $I$-word such that $P(\mathcal{C})$ is point-wise-finite resolution in degree $d\in\mathbb{Z}$.  \begin{enumerate}
    \item Let $\Psi$ be the set of $i\in I$ such that $\mathcal{C}_{j}$ is inverse whenever $I\ni j-1<i$. If $\Psi\neq \emptyset$ then $\Psi$ has a maximal element; and if $\Psi=\emptyset$ then $-\mathbb{N}\nsubseteq I$.
    \item Let $\Phi$ be the set of $i\in I$ such that $\mathcal{C}_{j}$ is direct whenever $I\ni j+1> i$. If $\Phi\neq \emptyset$ then $\Psi$ has a minimal element; and if $\Phi=\emptyset$ then $\mathbb{N}\nsubseteq I$.
\end{enumerate}
\end{lemma}
\begin{proof}
(1) For a contradiction suppose $\emptyset\neq\Psi$ which does not have a maximal element. Then we must have that $I=\mathbb{N}$ or $I=\mathbb{Z}$, and that $\mathcal{C}_{i}$ is inverse for all $i\in I$. This would contradict that $P(\mathcal{C})$ is bounded above.

Now suppose $\Psi=\emptyset$. By Lemma \ref{lemma:string-complex-resolution-technical-III} this means that the set of $i\in I$ such that ($\mathcal{C}_{i}$ and $\mathcal{C}_{i+1}$ are both inverse) is empty. For a contradiction suppose $-\mathbb{N}\subseteq I$.

As above, since $P^{l}(\mathcal{C})=0$ for all $l>d$, we cannot have that $\mathcal{C}_{j}$ is direct for all $j$. Choose $n\in I$ with $\mathcal{C}_{n}$ inverse. Since no such $i$ as above exists, $\mathcal{C}_{n-1}$ is direct. By Lemma \ref{lemma:string-complex-resolution-technical} this means $\mathscr{H}_{\mathcal{C}}(n)=d-1$. Since $P^{d+1}(\mathcal{C})=0$ we must have that $\mathcal{C}_{n-2}$ is inverse. As above this means $\mathcal{C}_{n-3}$ is direct (since no such $i$ as above exists). Continuing this way, one can show $\mathcal{C}_{n-2m}$ is inverse and $\mathcal{C}_{n-2m-1}$ is direct for all $m\in\mathbb{N}$. This contradicts that $P^{d}(\mathcal{C})$ is finitely generated.

(2) The proof here is similar to the proof of (1), but where one swaps direct (respectively inverse) letters with inverse (respectively direct) letters, and where one applies Lemma \ref{lemma:string-complex-resolution-technical-IV} instead of Lemma \ref{lemma:string-complex-resolution-technical-III}.
\end{proof}
\begin{definition}\label{definition:alternating-words-string-resolutions}
By an \emph{alternating} generalised word we mean one of the form 
\[
\mathcal{A}=\langle \mu_{1}\rangle \langle \eta_{1}\rangle ^{-1}\dots \langle \mu_{n}\rangle \langle \eta_{n}\rangle ^{-1}
\]
for some $\mu_{i},\eta_{i}\in\mathbf{P}$. By a \emph{string}-\emph{resolution} we mean a generalised word of the form $\mathcal{C}=\mathcal{B}^{-1}(\mathcal{A}\mathcal{D})$ such that the following statements hold.
\begin{enumerate}
 \item The generalised word $\mathcal{A}$ is either trivial or alternating. 
 \item If the generalised (finite or $\mathbb{N}$)-word $\mathcal{B}$ is non-trivial, then:
 \begin{enumerate}
     \item $\mathcal{B}$ is direct, and for all $n>1$ we have $\mathcal{B}_{n}=\langle\sigma\rangle$ for some $\sigma\in\mathbf{A}$;
     \item and if $ b\in\mathbf{A}$ then $\mathcal{B}\langle b\rangle$ is not a generalised word.
 \end{enumerate}
 \item If the generalised (finite or $\mathbb{N}$)-word $\mathcal{D}$ is non-trivial, then:
 \begin{enumerate}
     \item $\mathcal{D}$ is direct, and for all $n>1$ we have $\mathcal{D}_{n}=\langle\gamma\rangle$ for some $\gamma\in\mathbf{A}$;
     \item and if $ d\in\mathbf{A}$ then $\mathcal{D}\langle d\rangle$ is not a generalised word.
 \end{enumerate}
\end{enumerate}
In this case we refer to the triple $(\mathcal{B},\mathcal{A},\mathcal{D})$ as \emph{a decomposition of} $\mathcal{C}$.
\end{definition}
\begin{lemma}\label{lemma:generalised-decompositions-under-shifts-and-inverses}
Let $\mathcal{C}$ and $\mathcal{C}'$ be string words with decompositions $(\mathcal{B},\mathcal{A},\mathcal{D})$ and $(\mathcal{B}',\mathcal{A}',\mathcal{D}')$ respectively, and suppose $\mathcal{A}$ is a generalised $\{0,\dots,d\}$-word. The following statements hold.
\begin{enumerate}
    \item If $\mathcal{C}'=\mathcal{C}[n]$ then $\mathcal{C}'=\mathcal{C}$ and $(\mathcal{B}',\mathcal{A}',\mathcal{D}')=(\mathcal{B},\mathcal{A},\mathcal{D})$.
    \item If $\mathcal{C}'=\mathcal{C}^{-1}[n]$ then $\mathcal{C}'=\mathcal{C}^{-1}[-d]$ and $(\mathcal{B}',\mathcal{A}',\mathcal{D}')=(\mathcal{D},\mathcal{A}^{-1},\mathcal{B})$.
\end{enumerate}
\end{lemma}
\begin{proof}
It is straightforward to adapt the proof of Lemma \ref{lemma:decompositions-under-shifts-and-inverses}.
\end{proof}
\begin{definition}\label{definition:interval-of-a-string-resolution} Let $\mathcal{C}$ be a generalised word which is a string resolution. Taking $n=0$ in Lemma \ref{lemma:generalised-decompositions-under-shifts-and-inverses} shows that if $(\mathcal{B},\mathcal{A},\mathcal{D})$ and $\mathcal{B}',\mathcal{A}',\mathcal{D}')$ are decompositions of $\mathcal{C}$ then $\mathcal{B}'=\mathcal{B}$, $\mathcal{A}'=\mathcal{A}$ and $\mathcal{D}'=\mathcal{D}$. Henceforth we refer to $(\mathcal{B},\mathcal{A},\mathcal{D})$ as \emph{the} decomposition of $\mathcal{C}$, since it is unique. In particular, this means that the elements $\iota(-),\iota(+)\in I$ such that $\mathcal{B}=(\mathcal{C}_{\leq \iota(-)})^{-1}$ $\mathcal{A}=(\mathcal{C}_{> \iota(+)})_{\leq\iota(+)-\iota(-)}$ and $\mathcal{D}=\mathcal{C}_{>\iota(+)}$ are also unique. From now on we will refer to the pair $[\iota(-),\iota(+)]$ as \emph{the interval of} $\mathcal{C}$.
\end{definition}
\begin{remark}\label{remark:uniqueness-for-string-resolutions}
Let $\mathcal{C}$ be a generalised word which is a string resolution with decomposition $(\mathcal{B},\mathcal{A},\mathcal{D})$ and interval $[\iota(-),\iota(+)]$. Since $\mathcal{A}$ is trivial or alternating, we have $\mathscr{H}_{\mathcal{C}}(\iota(-))=\mathscr{H}_{\mathcal{C}}(\iota(+))$. If additionally $\mathcal{C}$ is a generalised $\mathbb{Z}$-word then $\iota(-)=0$ and $\mathcal{A}$ is a generalised $\{0,\dots,\iota(+)\}$-word. 
\end{remark}
\begin{lemma}\label{lemma:string-complex-resolution-given-by-a-certain-form-II}
Let $\mathcal{C}$ be a generalised $I$-word such that $P(\mathcal{C})$ is a point-wise-finite resolution in degree $d\in\mathbb{Z}$. Then there exists $n\in\mathbb{Z}$ such that $\mathcal{C}[n]$ is a string resolution and 
\[
d=\begin{cases}
\mathscr{H}_{\mathcal{C}}(\iota(\pm)) & \emph{(}\text{if }I\neq\mathbb{Z}\emph{)}\\
\mathscr{H}_{\mathcal{C}}(n) & \emph{(}\text{if }I=\mathbb{Z}\emph{)}
\end{cases}
\]
where, in case $I\neq \mathbb{Z}$, $[\iota(-),\iota(+)]$ is the interval of $\mathcal{C}$.
\end{lemma}
\begin{proof}Let $\Psi$ be the set of $i\in I$ such that $\mathcal{C}_{j}$ is inverse whenever $I\ni j-1<i$. Let $\Phi$ be the set of $i\in I$ such that $\mathcal{C}_{j}$ is direct whenever $I\ni j+1> i$. If $\Psi\neq \emptyset$ (respectively $\Phi\neq \emptyset$) then let $n=\mathrm{max}(\Psi)$ (respectively $m=\mathrm{min}(\Phi)$), which exists by Lemma \ref{lemma:string-complex-resolution-given-by-a-certain-form}. If $\Psi=\emptyset$ let $n=\mathrm{min}(I)=0$. If $\Phi=\emptyset$ then let $m=\mathrm{max}(I)$; and so either $I=\{0,\dots,m\}$ or $m=0$ by Lemma \ref{lemma:string-complex-resolution-given-by-a-certain-form}. If $-\mathbb{N}\nsubseteq I$ let $u=h(\mathcal{C})$ and $\delta=s(\mathcal{C})$ so that $\mathcal{C}=\langle1_{u,\delta}\rangle\mathcal{C}$. If $\mathbb{N}\nsubseteq I$ let $w=t(\mathcal{C})$ and $\rho=-s(\mathcal{C}^{-1})$ so that $\mathcal{C}=\mathcal{C}\langle1_{w,\rho}\rangle$. 

If $\Psi\neq\emptyset$ let $\mathcal{B}=( \mathcal{C}_{\leq n})^{-1}=\mathcal{C}_{n}^{-1}\mathcal{C}_{n-1}^{-1}\dots$, and if $\Psi=\emptyset$ let $\mathcal{B}=\langle1_{u,-\delta}\rangle$. Dually, if $\Phi\neq\emptyset$ let $\mathcal{D}= \mathcal{C}_{> m-1}=\mathcal{C}_{m}\mathcal{C}_{m+1}\dots$, and if $\Phi=\emptyset$ let $\mathcal{D}=\langle1_{w,\rho}\rangle$. If $n=m-1$ let $\mathcal{A}=\langle1_{v,\varepsilon}\rangle$. Otherwise $n<m-1$, in which case: the maximality of $n$ and minimality of $m$ give us that $\mathcal{C}_{n+1}$ is direct and $\mathcal{C}_{m-1}$ is inverse, which must mean $n+1<m-1$; and hence we let $\mathcal{A}=\mathcal{C}_{n+1}\dots\mathcal{C}_{m-1}$. By construction we have $\mathcal{C}[n]=\mathcal{B}^{-1}\mathcal{A}\mathcal{D}$. 

Suppose $\mathcal{A}$ is non-trivial, and that there is some $i\in I$ such that $\mathcal{C}_{i-1}$ and $\mathcal{C}_{i}$ are both inverse. By  Lemma \ref{lemma:string-complex-resolution-technical-III} we have that $i\in\Psi$, which means $i\leq n$ by the maximality of $n$. Likewise if there is some $j\in I$ such that $\mathcal{C}_{j}$ and $\mathcal{C}_{j+1}$ are both direct then $\Phi\neq\emptyset$ and $j\geq m$ by Lemma \ref{lemma:string-complex-resolution-technical-IV}. As above note $\mathcal{C}_{n+1}$ is direct and $\mathcal{C}_{m-1}$ is inverse. Since any such $i$ as above must satisfy $i\leq n$, we must have that $\mathcal{C}_{m-2}$ is direct. Dually, since any such $j$ satisfies $j\geq m$, we must have that $\mathcal{C}_{n+2}$ is inverse. Continuing this way, one can show (1) holds.

That $\mathcal{B}$ and $\mathcal{D}$ are direct follows from their construction. By Lemma \ref{lemma:string-complex-resolution-technical-III}(2) we have that part (2a) holds. By Lemma \ref{lemma:string-complex-resolution-technical-III}(3) we have that part (2b) holds. The proofs of parts (3a) and (3b) follow similarly from Lemma \ref{lemma:string-complex-resolution-technical-IV}. From here it suffices to verify the claim concerning $d$. 

In case $\mathcal{C}$ is trivial there is nothing to prove, so we may assume $\mathcal{C}$ is non-trivial. Suppose $I\neq\mathbb{Z}$, which means $\mathcal{C}=\mathcal{B}^{-1}(\mathcal{A}\mathcal{D})$ and hence $\iota(-)=n$. Note that if $\mathcal{B}$ is trivial then $I\neq -\mathbb{N}$ and $n=0$ and, whether or not $\mathcal{A}$ is trivial, $\mathcal{C}_{1}$ is direct. So, by Lemma \ref{lemma:string-complex-resolution-technical-VI}(1) we have $d=0$. Hence we can assume $\mathcal{B}$ is non-trivial. Dually, by applying Lemma \ref{lemma:string-complex-resolution-technical-VI}(2) we can assume $\mathcal{D}$ is non-trivial. Again, in any case this means $\mathcal{C}_{n}$ is inverse and $\mathcal{C}_{n+1}$ is direct. By Lemma \ref{lemma:string-complex-resolution-technical-II} and Remark \ref{remark:uniqueness-for-string-resolutions} we have $d=\mathscr{H}_{\mathcal{C}}(\iota(\pm))$. Suppose instead $I=\mathbb{Z}$. Here we again have that $\mathcal{C}_{n}$ is inverse and $\mathcal{C}_{n+1}$ is direct, and by Lemma \ref{lemma:string-complex-resolution-technical-II} we have $\mathscr{H}_{\mathcal{C}}(n)=d$. 
\end{proof}
For Lemma \ref{lemma:string-complex-given-by-string-resolution-is-a-resolution} we require the following result from the author's Ph.D thesis \cite{Ben2018}.
\begin{lemma}\label{lemma:kernel-of-string-complexes}\emph{\cite[Corollary 2.7.8]{Ben2018}} Let $\mathcal{C}$ be a generalised $I$-word. For any
\emph{$n\in\mathbb{Z}$} we have \emph{$\mathrm{ker}(d_{P(\mathcal{C})}^{n})=\bigoplus_{i\in\mathscr{H}_{\mathcal{C}}^{-1}(n)}\Lambda\kappa(i) g_{\mathcal{C},i}$} where for each $i$ we let
\[
\kappa(i)=
\begin{cases}
e_{v_{\mathcal{C}}(i)} & (\text{if }(i-1\notin I\text{ or }\mathcal{C}_{i}=\langle\gamma\rangle ^{-1})\text{ and }(i+1\notin I \text{ or }\mathcal{C}_{i+1}=\langle\sigma\rangle))\\
\mathrm{f}( \gamma) & (\text{if }i\pm1\in I\text{ and }\mathcal{C}_{i}\mathcal{C}_{i+1}=\langle\gamma\rangle ^{-1}\langle\sigma\rangle ^{-1})\\
\mathrm{f}( \sigma) & (\text{if }i\pm1\in I\text{ and }\mathcal{C}_{i}\mathcal{C}_{i+1}=\langle\gamma\rangle \langle\sigma\rangle)\\
 b & (\text{if }i-1\notin I\ni i+1, \mathcal{C}_{i+1}=\langle\sigma\rangle ^{-1}, b\in\mathbf{A}\text{ and } b\mathrm{l}(\sigma)=0)\\
 a & (\text{if }i+1\notin I\ni i-1,\mathcal{C}_{i}=\langle\gamma\rangle, a\in\mathbf{A}\text{ and } a \mathrm{l}(\gamma)=0)\\
0 &  (\text{otherwise})
\end{cases}
\]
\end{lemma}
\begin{example}\label{example:kernel-parts-example-k[[x,y]]/(xy)} We continue Example \ref{example:string-complex-for-k[[x,y]]/(xy)}. So $\mathcal{C}={}^{\infty\hspace{-0.1ex}}(\langle y\rangle ^{-1}\langle x\rangle ^{-1})\langle x^{4}\rangle ^{-1} \langle y^{3}\rangle \langle x^{2}\rangle ^{-1}\langle y\rangle \langle x^{2}\rangle ^{-1}\langle y^{3}\rangle \langle x^{4}\rangle ^{-1}$ and $\Lambda=k[[x,y]]/(xy)$. For $i=0,-2,-4,-6$ we have $\kappa(i)=1_{\Lambda}$. For $i=-1,-3,-5$ we have $\kappa(i)=0$. When $i<-6$ we are in the second case: when $i<-6$ is odd $\kappa(i)=y$, and when $i<-6$ is even $\kappa(i)=x$.
\end{example}

\begin{lemma}\label{lemma:string-complex-given-by-string-resolution-is-a-resolution}
Let $\mathcal{C}$ be a string resolution with interval $[\iota(-),\iota(+)]$. Then $P(\mathcal{C})$ is a point-wise-finite resolution in degree $\mathscr{H}_{\mathcal{C}}(\iota(\pm))$. 
\end{lemma}
\begin{proof}
Suppose $\mathcal{C}$ is a generalised $I$-word. Let $(\mathcal{B},\mathcal{A},\mathcal{D})$ be the decomposition of $\mathcal{C}$. Let $P=P(\mathcal{C})$ and $d=\mathscr{H}(\iota(\pm))$. For each $n\in\mathbb{Z}$ it is straightforward to check that the pre-image $\mathscr{H}_{\mathcal{C}}^{-1}(n)$ is a finite set. Hence $P$ is point-wise-finite. By construction, since each of $\mathcal{B}$ and $\mathcal{D}$ are (trivial or direct) and $\mathcal{A}$ is trivial or alternating, we have that $P^{l}=0$ when $l>d$. Hence it suffices to let $l<d$ and show $H^l(P)=0$. 

Consider firstly the case $l=d-1$. Suppose that $\mathcal{A}$ is a generalised $\{0,\dots,2m\}$-word. Note that $\mathscr{H}_{\mathcal{C}}^{-1}(l)$ is the set of integers of the form $\iota(-)-1+2i$ where $i$ begins with $0$ (respectively $1$) if $\mathcal{B}$ is non-trivial (respectively, trivial), and $i$ runs through to $m+1$ (respectively, $m$) if $\mathcal{D}$ is non-trivial (respectively, trivial). Furthermore, note that for any such integer $j=\iota(-)-1+2i$ with $1\leq i\leq m$ we have that $\mathcal{C}_{j}\mathcal{C}_{j+1}$ has the form $\langle\mu\rangle\langle\eta\rangle^{-1}$ where $\mu,\eta\in\mathbf{P}$. In the sense of Lemma \ref{lemma:kernel-of-string-complexes} we have $\kappa(j)=0$ for each such $j$. If both $\mathcal{B}$ and $\mathcal{D}$ are trivial then from here there is nothing to prove. Instead suppose that $\mathcal{B}$ is non-trivial. Let $t=\iota(-)-2$. Note that if $t\notin I$ and $\kappa(t)\neq 0$ then $\kappa(t)=\langle b\rangle$ but then $\mathcal{B}\langle b\rangle$ is a generalised word, which is impossible since $\mathcal{C}$ is a string resolution. Hence we can assume $t\in I$. Writing $\mathcal{B}_{2}=\langle b'\rangle$ for some $ b'\in\mathbf{A}$ gives $\kappa(t)= b'$. 

Together with a dual argument we have $H^{d-1}(P)=0$. Using similar arguments, again applying Lemma \ref{lemma:kernel-of-string-complexes}, it is straightforward to show $H^{l}(P)=0$ when $l<d-1$. 
\end{proof}
Just as we have done for string complexes, we now characterize band complexes which are resolutions.
\subsection{Band complexes}
\begin{lemma}\label{lemma:band-complex-resolution-technical}
Let $V$ be an object of \emph{$R[T,T^{-1}]\text{-\textbf{Mod}}_{R\text{-\textbf{Proj}}}$}. Let $\mathcal{E}$ be a cyclic generalised $\{0,\dots,p\}$-word such that $P(\mathcal{E},V)$ is a resolution in degree $d\in\mathbb{Z}$. If  $\mathcal{E}_{i} \mathcal{E}_{i+1}=\langle \mu\rangle \langle \eta\rangle ^{-1}$ for some $\mu,\eta\in\mathbf{P}$ \emph{(}where $\mathcal{E}_{p+1}=\mathcal{E}_{1}$\emph{)} then $\mathscr{H}_{\mathcal{C}}(i)=d-1$.
\end{lemma}
\begin{proof}Let $\mathscr{H}_{\mathcal{C}}(i)=t$. Suppose $t\geq d$. Then by construction we have $\mathscr{H}_{\mathcal{C}}(i\pm 1)\geq d+1$ which contradicts the fact that $P^{l}(\mathcal{E},V)=0$ when $l>d$. Hence we must have $t< d$. For a contradiction assume $t\leq  d-2$.

For all $\tau\in\Omega$ let
\[
M_{\tau}
\begin{cases}
\Lambda \mathrm{f}(\mu)g_{\mathcal{E},i-1,\tau}\oplus\Lambda \mathrm{f}(\eta)g_{\mathcal{E},i+1,\tau} & (\text{if }0<i<p-1) \\
\Lambda \mathrm{f}(\mu)g_{\mathcal{E},p-1,\tau}\oplus\Lambda \mathrm{f}(\eta)g_{\mathcal{E},1,\tau} & (\text{if }i=0)\\
\Lambda \mathrm{f}(\mu)g_{\mathcal{E},p-2,\tau}\oplus\Lambda \mathrm{f}(\eta)g_{\mathcal{E},0,\tau} & (\text{if }i=p-1)
\end{cases}
\]
Now let
\[
\begin{array}{ccc}
L=\mathrm{im}(d^{t}_{P(\mathcal{E},V)}), & M=\sum_{\tau}M_{\tau}, & N=\mathrm{ker}(d^{t+1}_{P(\mathcal{E},V)}).
\end{array}
\]
By construction $M\subseteq N$, and since $t+1\leq d-1$ we have $N/L=H^{t+1}(P(\mathcal{E},V))=0$ by assumption. Since $L=\sum L_{j}$ where $L_{j}=\sum_{\omega}\Lambda(g_{\mathcal{E},j,\omega}^{+}+g_{\mathcal{E},j,\omega}^{-})$ for all $j\in \mathscr{H}_{\mathcal{E}}^{-1}(t)$, that $M\subseteq L$ contradicts Lemma \ref{lemma:band-complex-technical}.  
\end{proof}
The proof of Lemma \ref{lemma:band-complex-resolution-technical-II} is similar to the proof of Lemma \ref{lemma:string-complex-resolution-technical-II}, and omitted.
\begin{lemma}\label{lemma:band-complex-resolution-technical-II}
Let $V$ be an object of \emph{$R[T,T^{-1}]\text{-\textbf{Mod}}_{R\text{-\textbf{Proj}}}$}. Let $\mathcal{E}$ be a cyclic generalised $\{0,\dots,p\}$-word, where $P(\mathcal{E},V)$ is a resolution in degree $d\in\mathbb{Z}$. If  $\mathcal{E}_{i} \mathcal{E}_{i+1}=\langle \mu\rangle^{-1} \langle \eta\rangle $ for some $\mu,\eta\in\mathbf{P}$ \emph{(}where $\mathcal{E}_{p+1}=\mathcal{E}_{1}$\emph{)} then $\mathscr{H}_{\mathcal{C}}(i)=d$.
\end{lemma}
In Lemma \ref{lemma:band-complex-resolution-technical-III} we show that if $P(\mathcal{E},V)$ is a resolution then the cyclic generalised word $\mathcal{E}$ must be an alternating sequence of direct and inverse  generalised letters.
\begin{lemma}\label{lemma:band-complex-resolution-technical-III}
Let $V$ be an object of \emph{$R[T,T^{-1}]\text{-\textbf{Mod}}_{R\text{-\textbf{Proj}}}$}. Let $\mathcal{E}$ be a cyclic generalised $\{0,\dots,p\}$-word where $P(\mathcal{E},V)$ is a resolution in degree $d\in\mathbb{Z}$. Then $\mathcal{E}_{i} \mathcal{E}_{i+1}=\langle \mu\rangle^{-1} \langle \eta\rangle $ or  $\mathcal{E}_{i} \mathcal{E}_{i+1}=\langle \mu\rangle\langle \eta\rangle^{-1} $ for some $\mu,\eta\in\mathbf{P}$ \emph{(}where $\mathcal{E}_{p+1}=\mathcal{E}_{1}$\emph{)}.
\end{lemma}
\begin{proof}
For a contradiction, and without loss of generality, suppose $\mathcal{E}_{i} \mathcal{E}_{i+1}=\langle \mu\rangle \langle \mu'\rangle $ where $1\leq i<p$. Since $\mathcal{E}$ is cyclic we have $\mathscr{H}_{\mathcal{E}}(p)=0$, and so we cannot have that $\mathcal{E}_{j}$ is direct for all $i$. Let $\mathcal{E}_{i\pm p}=\mathcal{E}_{i}$ for each $i=1,\dots,p$. In this notation there exists $n$ minimal such that $i+1<n$, $\mathcal{E}_{n}$ is inverse and $\mathcal{E}_{n-1}$ is direct. Hence $\mathcal{E}_{j}$ is direct for all $j$ with $i\leq j<n$. Likewise, we can choose $m$ with $ m <i$ such that $\mathcal{E}_{m}$ is inverse and $\mathcal{E}_{j}$ is direct for all $j$ with $m< j\leq i+1$. Hence
\[
\mathcal{E}_{m}\mathcal{E}_{m+1}\dots \mathcal{E}_{n-1}\mathcal{E}_{n}=\langle\eta_{-}\rangle^{-1}\langle\mu_{m+1}\rangle\dots\langle\mu_{n-1}\rangle\langle\eta_{+}\rangle^{-1}
\]
for some $\eta_{\pm},\mu_{j}\in\mathbf{P}$. By Lemma \ref{lemma:band-complex-resolution-technical} we have that $\mathscr{H}_{\mathcal{E}}(n)=d-1$. By Lemma \ref{lemma:band-complex-resolution-technical-II} we have that $\mathscr{H}_{\mathcal{E}}(m)=d$. This is only possible if $m+1=n-1$, but $m<i$ and $n>i+1$, so we have a contradiction.
\end{proof}
\begin{definition}\label{definition:alternating-words-band-resolutions}By a \emph{band}-\emph{resolution} we mean a periodic generalised $\mathbb{Z}$-word of the form  $\mathcal{C}={}^{\infty\hspace{-0.5ex}}\mathcal{A}^{\infty}$ where $\mathcal{A}$ is a cyclic alternating generalised word.
\end{definition}
\begin{lemma}\label{lemma:band-complex-resolution-given-by-a-certain-form}
Let $V$ be an object of \emph{$R[T,T^{-1}]\text{-\textbf{Mod}}_{R\text{-\textbf{Proj}}}$}. Let $\mathcal{C}$ be a $p$-periodic generalised $\mathbb{Z}$-word such that $P(\mathcal{C},V)$ is a resolution in degree $d\in\mathbb{Z}$. Then the following statements hold.
\begin{enumerate}
    \item Either \emph{(}$d=0$ and $\mathcal{C}$ is a band resolution\emph{)} or \emph{(}$d=1$ and $\mathcal{C}[-1]$ is a band resolution\emph{)}. 
    \item If $P(\mathcal{C},V)$ is point-wise-finite then $V$ is finitely generated as an $R$-module.
\end{enumerate}
\end{lemma}
\begin{proof}By assumption there is a cyclic generalised $\{0,\dots, p\}$-word $\mathcal{A}=\mathcal{C}_{1}\dots \mathcal{C}_{p}$ such that $\mathcal{C}={}^{\infty\hspace{-0.5ex}}\mathcal{A}^{\infty}$. By Remark \ref{remark:isomorphism-for-band-complex} there is an isomorphism of complexes $P(\mathcal{C},V)\simeq P(\mathcal{A},V)$, and so $P(\mathcal{A},V)$ is a resolution in degree $d$. Since $\mathcal{A}$ is cyclic we must have that $p$ is even, and so by Lemma \ref{lemma:band-complex-resolution-technical-III} we have that $\mathcal{E}=\langle\mu_{1}\rangle\langle\eta_{1}\rangle^{-1}\dots \langle\mu_{n}\rangle\langle\eta_{n}\rangle^{-1}$ or $\mathcal{E}=\langle\mu_{1}\rangle^{-1}\langle\eta_{1}\rangle\dots \langle\mu_{n}\rangle^{-1}\langle\eta_{n}\rangle$ for some $\mu_{i},\eta_{i}\in\mathbf{P}$ where $n=p/2$. This shows (1) holds. For (2), it is straightforward to check that if $P^{n}(\mathcal{E},V)$ is finitely generated as a $\Lambda$-module then $V$ has a finite $R$-basis.
\end{proof}
Lemma \ref{lemma:band-complex-given-by-band-resolution-is-a-resolution} is the analogue of Lemma \ref{lemma:string-complex-given-by-string-resolution-is-a-resolution}, but concerns band complexes, instead of string complexes. In the proof of Lemma \ref{lemma:string-complex-given-by-string-resolution-is-a-resolution} we applied \cite[Corollary 2.7.8]{Ben2018} (see Lemma \ref{lemma:kernel-of-string-complexes}), which was more general than needed for our purposes. Instead of using an analogue of \cite[Corollary 2.7.8]{Ben2018} for band complexes (see for example \cite[Theorem 2.8]{CanPauSch2021}), we prove Lemma \ref{lemma:band-complex-given-by-band-resolution-is-a-resolution} directly.
\begin{lemma}\label{lemma:band-complex-given-by-band-resolution-is-a-resolution}
Let $V$ be an object of \emph{$R[T,T^{-1}]\text{-\textbf{Mod}}_{R\text{-\textbf{Proj}}}$} and let $\mathcal{C}$ be a generalised word which is a band-resolution. Then $P(\mathcal{C},V)$ is a resolution in degree $0$. Furthermore if $V$ is finitely generated as an $R$-module then $P(\mathcal{C},V)$ is point-wise-finite.
\end{lemma}
\begin{proof}
By definition $\mathcal{C}={}^{\infty\hspace{-0.5ex}}\mathcal{A}^{\infty}$ where $\mathcal{A}=\langle \mu_{1}\rangle \langle \eta_{1}\rangle ^{-1}\dots \langle \mu_{n}\rangle \langle \eta_{n}\rangle ^{-1}$ for some $\mu_{i},\eta_{i}\in\mathbf{P}$. Let $(v_{\omega}\colon\omega\in\Omega)$ be an $R$-basis of $V$, and let $P=P(\mathcal{A},V)$ as in Definition \ref{definition:band-pres}. By Remark \ref{remark:isomorphism-for-band-complex} it suffices to prove $P$ is a resolution in degree $0$ and that $P$ is point-wise-finite when $V$ is module finite over $R$. Suppose $\mathcal{A}$ is a $\{0,\dots,p\}$-word. Hence $p$ is even and $n=p/2$. Note that 
\[
\begin{array}{cc}
P^{-1}=\bigoplus_{i=0}^{n-1}\bigoplus_{\omega}\Lambda g_{\mathcal{E},2i+1,\omega} & P^{0}=\bigoplus_{i=0}^{n}\bigoplus_{\omega}\Lambda g_{\mathcal{E},2i,\omega}
\end{array}
\]
and $P^{l}=0$ for $l\neq -1,0$. Thus $P$ is point-wise-finite when $V$ is module finite over $R$. To show $P$ is a resolution in degree $0$ it suffices to show $d_{P}^{-1}$ is injective, since $P$ is concentrated in degrees $-1$ and $0$. For each $h=1,\dots,n$ let $v(h)=h(\mu_{h})$ (and so also $v(h)=h(\eta_{h})$). Any element $m\in P^{-1}$ has the form $m=\sum_{i,\omega}\lambda_{2i+1,\omega}g_{\mathcal{E},2i+1,\omega}$ where $i$ runs through $\{0,\dots,n-1\}$, $\omega$ runs through $\Omega$ and each $\lambda_{2i+1,\omega}\in \Lambda e_{v(i+1)}$. 

Assuming $m$ lies in the kernel of $d^{-1}_{P}$, it now suffices to show $m=0$. The assumption gives
\[
\begin{array}{c}
0
=\sum_{\omega}\left(\lambda_{1,\omega}\mu_{1}g_{\mathcal{E},0,\omega}+\lambda_{2n-1,\omega}\eta_{n}(\sum_{\tau} a_{\tau\omega}^{-}g_{\mathcal{E},0,\tau})\right)+\sum_{i>0}\sum_{\omega}\left(
(\lambda_{2i-1,\omega}\eta_{i}+\lambda_{2i+1,\omega}\mu_{i+1})g_{\mathcal{E},i,\omega}
\right)
\end{array}
\]
Since $P^{0}$ is the direct sum of $\bigoplus_{\omega}\Lambda g_{\mathcal{E},2i,\omega}$ as $i$ runs through $\{0,\dots,n\}$ the above gives
\[
\begin{array}{cc}
0=\sum_{\omega}\left(\lambda_{1,\omega}\mu_{1}g_{\mathcal{E},0,\omega}+\lambda_{2n-1,\omega}\eta_{n}(\sum_{\tau} a_{\tau\omega}^{-}g_{\mathcal{E},0,\tau})\right),& 0=\sum_{\omega}
(\lambda_{2i-1,\omega}\eta_{i}+\lambda_{2i+1,\omega}\mu_{i+1})g_{\mathcal{E},i,\omega}
\end{array}
\]
for each $i=1,\dots,n-1$. Since the sum $\bigoplus_{\omega}\Lambda g_{\mathcal{E},2i,\omega}$ itself is direct, the second equation gives $\lambda_{2i-1,\omega}\eta_{i}+\lambda_{2i+1,\omega}\mu_{i+1}=0$ for each ($i$, and each) $\omega$. Since $\mathcal{A}$ is a generalised word we must have $\mathrm{f}(\eta_{i})\neq\mathrm{f}(\mu_{i+1})$ and hence $\lambda_{2i-1,\omega}\eta_{i}=0$ and $\lambda_{2i+1,\omega}\mu_{i+1}=0$ by condition (6) of Definition \ref{definition:complete-gentle-algebra}. After rearranging the sums over $\tau$ and $\omega$, the first equation shows that the sum over $\omega$ of the terms $\left(\lambda_{1,\omega}\mu_{1}+\sum_{\tau}\lambda_{2n-1,\tau}a_{\omega\tau}^{-}\eta_{n}
\right)g_{\mathcal{E},0,\omega}$ is $0$. As above, since $\bigoplus_{\omega}\Lambda g_{\mathcal{E},0,\omega}$ is a direct sum each of these terms is $0$; and by condition (6) of Definition \ref{definition:complete-gentle-algebra}, this gives $\lambda_{1,\omega}\mu_{1}=0$ and $\sum_{\tau}\lambda_{2n-1,\tau}a_{\omega\tau}^{-}\eta_{n}=0$ for each $\omega$.

Choose $i$ with $0\leq i<n-1$. We have already shown that $\lambda_{2i+1,\omega}\mu_{i+1}=0$, which means $\lambda_{2i+1,\omega}\in\mathrm{rad}(\Lambda e_{v(i)})$ by Lemma \ref{lemma:technical-comp-gen-props}(1). Without loss of generality (by condition (2) of Definition \ref{definition:complete-gentle-algebra}) suppose $\mathbf{A}(v(i)\rightarrow)=\{x,y\}$ where $x\neq y$, and so $\lambda_{2i+1,\omega}=\lambda_{x}x+\lambda_{y}y$ for some $\lambda_{x},\lambda_{y}$. By symmetry we can assume $x\eta_{i+1}=0=y\mu_{i+1}$ by condition (2) of Definition \ref{definition:complete-gentle-algebra}, and so $x\mu_{i+1},y\eta_{i+1}\in\mathbf{P}$ by condition (3) of Definition \ref{definition:complete-gentle-algebra}. Since $i<n-1$ we also have that $\lambda_{2i+1,\omega}\eta_{i+1}=0$, which gives $\lambda_{x}x\mu_{i+1}=0$. By Lemma \ref{lemma:technical-comp-gen-props}(2) this gives $\lambda_{x}x=0$. Likewise, since $\lambda_{2i+1,\omega}\mu_{i+1}=0$ we have $\lambda_{y}yx\eta_{i+1}=0$ and hence $\lambda_{y}y=0$. 

Altogether, so far we have that $\lambda_{2i+1,\omega}=0$ for all $\omega$ and all $i$ with  $0\leq i<n-1$. Recall also $\lambda_{2n-1,\omega}\mu_{n}=0$ and  $\sum_{\tau}\lambda_{2n-1,\tau}a_{\omega\tau}^{-}\eta_{n}=0$ for each $\omega$. Written another way, the latter of these equations gives $\sum_{\omega}\lambda_{2n-1,\omega}a_{\tau\omega}^{-}\eta_{n}=0$ for each $\tau$, and so for each $\nu\in\Omega$ we have
\[
\begin{array}{c}
0=\left(\sum_{\tau}a_{\nu\tau}^{+}\sum_{\omega}\lambda_{2n-1,\omega}a_{\tau\omega}^{-}\right)\eta_{n}=\left(\sum_{\omega}\lambda_{2n-1,\omega}(\sum_{\tau}a_{\nu\tau}^{+}a_{\tau\omega}^{-})\right)\eta_{n}=\left(\sum_{\omega}\lambda_{2n-1,\omega}\delta_{\nu\omega}\right)\eta_{n}
\end{array}
\]
where $\delta_{\nu\nu}=1_{\Lambda}$ and $\delta_{\nu\omega}=0$ when $\omega\neq\nu$. Hence $\lambda_{2n-1,\omega}\eta_{n}=0$. Using Lemma \ref{lemma:technical-comp-gen-props}(1), Lemma \ref{lemma:technical-comp-gen-props}(2) and a similar argument to that above; since $\lambda_{2n-1,\omega}\eta_{n}=0$ and $\lambda_{2n-1,\omega}\mu_{n}=0$ we have that $\lambda_{2n-1,\omega}=0$ for all $\omega$. We finally now have $\lambda_{2i+1,\omega}=0$ for all $i$ and $\omega$, which gives $m=0$ as required. 
\end{proof}
\section{Words and corresponding generalised words}\label{section:words-and-corresponding-generalised-words}
If $M$ is a $\Lambda$-module, recall a chain complex $P$ of projective $\Lambda$-modules is a  \emph{resolution of} $M$ if $P$ is a resolution in degree $0$ where $H^{0}(P)\simeq M$.
\subsection{The resolution and homology of a string}
\begin{definition}\label{definition:left-right-upwards-extensions-of-generalised-words}
Let $\mathcal{E}$ be a finite generalised word. We define three new generalised words $\mathcal{E}(\nwarrow)$, $\mathcal{E}(\nearrow)$ and $\mathcal{E}(\nwarrow\nearrow)$ as follows. Let $\mathcal{GW}_{\mathcal{E}}(\nwarrow)$ be the set of generalised $I$-words of the form
\[
\mathcal{C}=\begin{cases}
\langle 1_{v,\delta}\rangle  & (\mbox{if }I=\{0\})\\
\langle  a_{1}\rangle \,\dots \,\langle  a_{m}\rangle  & (\mbox{if }I=\{0,\dots,m\}\mbox{ for some }m>0)\\
\langle  a_{1}\rangle \,\langle  a_{2}\rangle \,\langle  a_{3}\rangle \,\dots & (\mbox{if }I=\mathbb{N})
\end{cases}
\]
where each $ a_{i}\in\mathbf{A}$ and $\mathcal{C}^{-1}\mathcal{E}$ is a generalised word. Choose $\mathcal{L}\in \mathcal{GW}_{\mathcal{E}}(\nwarrow)$ unique such that there does not exist any $ b\in\mathbf{A}$ such that $\mathcal{L}\langle  b\rangle$ is a generalised word. 

Dually let $\mathcal{GW}_{\mathcal{E}}(\nearrow)$ be the set of generalised words $\mathcal{C}$ of the above form where $\mathcal{E}\mathcal{C}$ is a generalised word, and choose $\mathcal{N}\in \mathcal{GW}_{\mathcal{E}}(\nearrow)$ unique such that there cannot exist $ d\in\mathbf{A}$ where $\mathcal{N}\langle  d\rangle $ is a generalised word. Now let
\[\begin{array}{ccc}
\mathcal{E}(\nwarrow)=\mathcal{L}^{-1}\mathcal{E}, & 
\mathcal{E}(\nearrow)=\mathcal{E}\mathcal{N}, &
\mathcal{E}(\nwarrow\nearrow)=\mathcal{L}^{-1} (\mathcal{E}\mathcal{N}).
\end{array}
\]
Note the placement of $\mid$ when $\mathcal{L}$ and $\mathcal{N}$ are $\mathbb{N}$-words: $\mathcal{E}(\nwarrow\nearrow)=\mathcal{L}^{-1}\mid \mathcal{E}\mathcal{N}$.
\end{definition}
For Definition \ref{definition:string-res-of-a-word} it is useful to recall definitions from \S\ref{section:string-and-band-modules}.
\begin{definition}\label{definition:string-res-of-a-word}
Let $C$ be an $I$-word which is a string word with decomposition $(B,A,D)$. The \emph{resolution of} $C$ is a generalised word denoted $\mathcal{R}_{C}$ and defined as follows. Hence each of $B$ and $D$ are (trivial or inverse), and $A$ is either trivial or
\[
A=\gamma_{1}^{-1}\sigma_{1}\dots\gamma_{n}^{-1}\sigma_{n},\, (\gamma_{i},\sigma_{i}\in\mathbf{P}).
\]
If $A=1_{v,\epsilon}$ let $\mathcal{A}=\langle1_{v,\epsilon}\rangle$, and otherwise (in the above notation) let
\[
\mathcal{A}=\langle\gamma_{1}\rangle\langle\sigma_{1}\rangle^{-1}\dots\langle\gamma_{n}\rangle\langle\sigma_{n}\rangle^{-1}.
\]
If it exists, write $ b$ for the (unique) element of $\mathbf{A}$ such that $B b^{-1}$ is a word. Likewise, if it exists, write $ d$ for the (unique) element of $\mathbf{A}$ such that $D d^{-1}$ is a word. In this notation now define $\mathcal{R}_{C}$ as follows:
\[
\mathcal{R}_{C}=
\begin{cases}
\left(\left(\langle b B^{-1}\rangle^{-1}\mathcal{A}\langle  d D^{-1}\rangle \right)(\nwarrow\nearrow)\right)[1] & (\text{if } b\text{ and } d\text{ both exist})\\
\left(\mathcal{A}\langle  d D^{-1}\rangle \right)(\nearrow) & (\text{if only } d\text{ exists})\\
\left(\langle b B^{-1}\rangle^{-1}\mathcal{A} \right)(\nwarrow) & (\text{if only } b\text{ exists}))\\
\mathcal{A}  & (\text{if neither } b\text{ nor } d\text{ exist})
\end{cases}
\]
\end{definition}
\begin{example}\label{example:left-right-upwards-extensions-for-k[[x,y]]/(xy)} We continue Example \ref{example:kernel-parts-example-k[[x,y]]/(xy)}. So $\Lambda=k[[x,y]]/(xy)$, $\mathcal{A}=\langle y^{3}\rangle \langle x^{2}\rangle ^{-1}\langle y\rangle \langle x^{2}\rangle ^{-1}\langle y^{3}\rangle \langle x^{4}\rangle ^{-1}$ and  $\mathcal{C}={}^{\infty\hspace{-0.1ex}}(\langle y\rangle ^{-1}\langle x\rangle ^{-1}) \langle x^{4}\rangle ^{-1}\mathcal{A}$.  
By construction we have that $\mathcal{C}=\mathcal{E}(\nwarrow)$ where $\mathcal{E}=\langle x^{4}\rangle ^{-1}\mathcal{A}$. Now recall Example \ref{example:string-module-for-k[[x,y]]/(xy)}, in which we have $C=B^{-1}(AD)$ where  $B=x^{-3}$, $D=(y^{-1})^{\infty}$ and $A=y^{-3}x^{2}y^{-1}x^{2}y^{-3}x^{4}$, and so $D$ is an $\mathbb{N}$-word and $xB^{-1}=x^{4}$ is a word. In the third case of Definition \ref{definition:string-res-of-a-word}, this means altogether that $\mathcal{R}_{C}=\mathcal{A}$. According to Lemma \ref{lemma:homology-of-resolution-gives-iso-from-string} below, since the interval of $\mathcal{C}$ is $[-6,0]$, we have $H^{0}(P(\mathcal{C}))\simeq M(C)$.
\end{example}
\begin{lemma}\label{lemma:homology-of-resolution-gives-iso-from-string}
Let $C$ be a string word. Then $\mathcal{C}=\mathcal{R}_{C}$ is a string resolution, and if $t=\mathscr{H}_{\mathcal{C}}(l)=\mathscr{H}_{\mathcal{C}}(m)$ where $[l,m]$ is the interval of $\mathcal{C}$ then there is a $\Lambda$-module isomorphism $H^{t}(P(\mathcal{C}))\simeq M(C)$.
\end{lemma}
\begin{proof}
In the notation of Definition \ref{definition:string-res-of-a-word} let $\delta=-s(\mathcal{A})$, $u=h(\mathcal{A})$, $\rho=-s(\mathcal{A}^{-1})$ and $w=h(\mathcal{A}^{-1})$. Hence $\langle 1_{u,\delta}\rangle^{-1}\mathcal{A}$ and $\mathcal{A}\langle 1_{w,\rho}\rangle$ are generalised words. Likewise let $\mathcal{B}=\langle b B^{-1}\rangle(\nearrow)$ if $ b$ exists, and otherwise let $\mathcal{B}=\langle 1_{u,\delta}\rangle$. Similarly let $\mathcal{D}=\langle d D^{-1}\rangle(\nearrow)$ if $ d$ exists, and otherwise let $\mathcal{D}=\langle 1_{w,\rho}\rangle$. Note that when $ b$ and $ d$ both exist, and when $\mathcal{C}$ is a generalised $\mathbb{Z}$-word, then
\[
\mathcal{C}=\left(\left(\langle b B^{-1}\rangle^{-1}\mathcal{A}\langle  d D^{-1}\rangle \right)(\nwarrow\nearrow)\right)[1]=\left(\mathcal{L}^{-1}(\langle b B^{-1}\rangle^{-1}\mathcal{A}\langle  d D^{-1}\rangle \mathcal{R})\right)[1]
\]
where $\mathcal{L}=\langle  b_{1}\rangle\langle  b_{2}\rangle\dots$ and $\mathcal{N}=\langle d_{1}\rangle\langle d_{2}\rangle\dots$ for some $  b_{i}, d_{i}\in\mathbf{A}$. Hence in this case we have 
\[
\mathcal{C}=\dots\langle  b_{2}\rangle^{-1}\langle  b_{1}\rangle^{-1}\langle b B^{-1}\rangle^{-1}\mid\mathcal{A}\langle  d D^{-1}\rangle\langle d_{1}\rangle\langle d_{2}\rangle\dots 
\]
In any case we have $\mathcal{C}=\mathcal{B}^{-1}(\mathcal{A}\mathcal{D})$, and so $\mathcal{C}$ is a string resolution with decomposition $(\mathcal{B},\mathcal{A},\mathcal{D})$. By definition, if $\pi(0),\dots,\pi(n)$ denotes the $C$-peaks in $I$, then $v_{C}(\pi(i))=t(\gamma_{i})$ for each $i$. Let $P=P(\mathcal{C})$. Hence for each such $i$ we have $v_{\mathcal{C}}(l+2i)=v_{C}(\pi(i))$, and so by construction $P^{t}=\bigoplus_{i=0}^{n}\Lambda  g_{\mathcal{C},l+2i}$ and the assignment $ g_{\mathcal{C},i}\mapsto g_{C,(i-l)/2}$ defines an isomorphism of projective $\Lambda$-modules $P^{t}\simeq N(C)$ as in Definition \ref{definition:string-modules-by-string-words}. The image of $d^{t-1}_{P}$ under this isomorphism is generated by the elements $\gamma_{j+1}g_{C,j}-\sigma_{j+1}g_{C,j+1}$ where $0\leq j<n$, together with the elements ($ b B^{-1} g_{C,0}$ if and only if $ b$ exists) and ($ d D^{-1} g_{C,n}$ if and only if $ d$ exists). 

Note $L_{-}(C)$ is non-trivial if and only if $ b$ exists, in which case it is equal to $\Lambda b B^{-1} g_{C,0}$. Together with a dual argument, we have that the submodule of $N(C)$ generated by the above list of elements is $L(C)$. 
\end{proof}
\begin{definition}\label{definition:left-right-downwards-extensions-of-generalised-words}
Let $A$ be a finite word. Below we define three new words $A(\swarrow)$, $A(\searrow)$ and $A(\swarrow\searrow)$. Let $W_{A}(\swarrow)$ be the set of (trivial or inverse) $I$-words such that $C^{-1}A$ is a word: that is, words of the form
\[
C=\begin{cases}
 1_{v,\delta} & (\mbox{if }I=\{0\})\\
 x_{1}^{-1}\dots x_{m}^{-1} & (\mbox{if }I=\{0,\dots,m\}\mbox{ for some }m>0)\\
 x_{1}^{-1}  x_{2} ^{-1} x_{3}^{-1}\dots  & (\mbox{if }I=-\mathbb{N})
\end{cases}
\]
where $s(C)=s(A^{-1})$ and each $x_{i}\in\mathbf{A}$. Choose $B\in W_{A}(\swarrow)$ unique such that $Bx^{-1}$ is not a word for all $x\in\mathbf{A}$. Hence $B$ is maximal in the sense that $B^{-1}A$ is a word but $xB^{-1}A$ is not. Dually let $W_{A}(\searrow)$ be the set of (trivial or inverse) $I$-words $C$ such that $AC$ is a word, and choose $D\in W_{A}(\searrow)$ unique such that ($AD$ is a word, and) $Dy^{-1}$ is a not a word for all $y\in\mathbf{A}$. Now let
\[\begin{array}{ccc}
A(\swarrow)=B^{-1}A, & 
A(\searrow)=AD, &
A(\swarrow\searrow)=B^{-1}(AD).
\end{array}
\]
\end{definition}
\begin{lemma}\label{lemma:technical-inverting-sarrow-extension-words}
If $A$ is a \emph{(}trivial or alternating\emph{)} $\{0,\dots,m\}$-word then $(A(\swarrow\searrow))^{-1}=A^{-1}(\swarrow\searrow)[-m]$. If additionally $C$ and $E$ are each either trivial or \emph{(}finite and inverse\emph{)} words such that $(C^{-1}A)E$ is a word, then
\[
\begin{array}{cc}
    (AE(\swarrow))^{-1}=E^{-1}A^{-1}(\searrow), & (C^{-1}A(\searrow))^{-1}=A^{-1}C(\swarrow).
\end{array}
\]
\end{lemma}
\begin{proof}
Without loss of generality we assume $A$ is non-trivial, and hence $A=\mu_{1}^{-1}\eta_{1}\dots \mu_{n}^{-1} \eta_{n}$ for some $\mu_{i},\eta_{i}\in\mathbf{P}$. Choose words $B$ and $D$ as in Definition \ref{definition:left-right-downwards-extensions-of-generalised-words}, such that $A(\swarrow\searrow)=B^{-1}(AD)$. This means 
\[
\begin{array}{c}
(A(\swarrow\searrow))^{-1}=(B^{-1}(AD))^{-1}=(AD)^{-1}B=D^{-1}A^{-1}\mid B=(D^{-1}\mid A^{-1}B)[-m]
\end{array}
\]
Hence we have $D\in W_{A^{-1}}(\swarrow)$ and $B\in W_{A^{-1}}(\searrow)$. Since $Dx^{-1}$ and $Bz^{-1}$ are not words for all $z\in\mathbf{A}$ we have that $D^{-1}\mid A^{-1}B=A^{-1}(\swarrow\searrow)$ and thus $(A(\swarrow\searrow))^{-1}=A^{-1}(\swarrow\searrow)[-m]$. We similarly have $(A(\swarrow))^{-1}=D^{-1}A^{-1}=A^{-1}(\searrow)$, noting they are both $I$-words where $I\neq\mathbb{Z}$. This gives $(AE(\swarrow))^{-1}=E^{-1}A^{-1}(\searrow)$. The proof that $(C^{-1}A(\searrow))^{-1}=A^{-1}C(\swarrow)$ is similar, and omitted.
\end{proof}
For Definition \ref{definition:homology-word-for-strings} it is convenient to recall Definition \ref{definition:alternating-words-string-resolutions}.
\begin{definition}\label{definition:homology-word-for-strings}
Let $\mathcal{C}$ be a generalised word which is a string resolution with decomposition $(\mathcal{B},\mathcal{A},\mathcal{D})$. If $\mathcal{A}$ is non-trivial let
\[
A=\mu_{1}^{-1}\eta_{1}\dots \mu_{n}^{-1} \eta_{n}\text{ where }\mathcal{A}=\langle \mu_{1}\rangle \langle \eta_{1}\rangle ^{-1}\dots \langle \mu_{n}\rangle \langle \eta_{n}\rangle ^{-1}.
\] 
Recall that $\mathcal{C}=\mathcal{B}^{-1}(\mathcal{A}\mathcal{D})$ and that each $\mathcal{B}$ and $\mathcal{D}$ are direct when they are non-trivial. In this setting, the \emph{homology word associated to} $\mathcal{C}$ is the word
\[
H(\mathcal{C})=
\begin{cases}
A(\swarrow\searrow) &(\text{if both }\mathcal{B}\text{ and }\mathcal{D}\text{ are trivial})\\
(A\mu^{-1})(\swarrow) & (\text{if }\mathcal{B}\text{ is trivial and } \mathcal{D}_{1}=\langle d\mu\rangle\text{ where } d\in\mathbf{A})\\
(\sigma A)(\searrow) & (\text{if }\mathcal{B}_{1}=\langle b\sigma\rangle \text{ where } b\in\mathbf{A}\text{ and }\mathcal{D}\text{ is trivial})\\
\sigma A \mu^{-1}& (\text{if }\mathcal{B}_{1}=\langle b\sigma\rangle \text{ and }\mathcal{D}_{1}=\langle d\mu\rangle\text{ where } b, d\in\mathbf{A})
\end{cases}
\]
Note, for example, that if $\mathcal{D}_{1}=\langle d\mu\rangle$ then the path $\mu$ is possibly trivial (but must satisfy $ d\mu\in\mathbf{P}$). It is straightforward to check, in each base of $\mathcal{B}$ and $\mathcal{D}$, that there are (trivial or inverse) words $B$ and $D$ such that $H(\mathcal{C})=B^{-1}(AD)$ and hence $H(\mathcal{C})$ is a string word with decomposition $(B,A,D)$.  
\end{definition}
In light of Lemma \ref{lemma:generalised-decompositions-under-shifts-and-inverses}, Lemma \ref{lemma:homology-word-of-a-shift-or-inverse-strings} describes how taking the homology word associated to a string resolution behaves under shifts and inverses.
\begin{lemma}\label{lemma:homology-word-of-a-shift-or-inverse-strings}
Let $\mathcal{C}$ be a generalised $I$-word which is a string resolution with interval $[l,m]$ and decomposition $(\mathcal{B},\mathcal{A},\mathcal{D})$ where $\mathcal{A}$ is a generalised $\{0,\dots,d\}$-word. Then $\mathcal{C}^{-1}[-d]$ is a string resolution with decomposition  $(\mathcal{D},\mathcal{A}^{-1},\mathcal{B})$ and there is some $m\in\mathbb{Z}$ such that $H(\mathcal{C}^{-1}[-d])=(H(\mathcal{C})^{-1})[m]$.
\end{lemma}
\begin{proof}
In what follows we use the notation from Definition \ref{definition:homology-word-for-strings}. By assumption we have $\mathcal{C}^{-1}=(\mathcal{A}\mathcal{D})^{-1}\mathcal{B}$, which has the form $\mathcal{D}^{-1}\mathcal{A}^{-1
}\mid \mathcal{B}$ when $I=\mathbb{Z}$. 

In any case this gives $\mathcal{C}^{-1}[-d]=\mathcal{D}^{-1}(\mathcal{A}^{-1}\mathcal{B})$
, which is a string resolution with decomposition $(\mathcal{D},\mathcal{A}^{-1},\mathcal{B})$. Assuming $\mathcal{A}=\langle \mu_{1}\rangle \langle \eta_{1}\rangle ^{-1}\dots \langle \mu_{n}\rangle \langle \eta_{n}\rangle ^{-1}$ gives $A=\mu_{1}^{-1}\eta_{1}\dots \mu_{n}^{-1} \eta_{n}$ and hence $A^{-1}=\eta_{n}^{-1}\mu_{n}\dots\eta_{1}^{-1} \mu_{1}$. By definition we now have that
\[
H(\mathcal{C}^{-1}[d])=
\begin{cases}
A^{-1}(\swarrow\searrow) &(\text{if both }\mathcal{D}\text{ and }\mathcal{B}\text{ are trivial})\\
(A^{-1}\sigma^{-1})(\swarrow) & (\text{if }\mathcal{D}\text{ is trivial and } \mathcal{B}_{1}=\langle b\sigma\rangle\text{ where } b\in\mathbf{A})\\
(\mu A^{-1})(\searrow) & (\text{if }\mathcal{D}_{1}=\langle d\mu\rangle \text{ where } d\in\mathbf{A}\text{ and }\mathcal{B}\text{ is trivial})\\
\mu A^{-1}
\sigma^{-1}& (\text{if }\mathcal{D}_{1}=\langle d\mu\rangle \text{ and }\mathcal{B}_{1}=\langle b\sigma\rangle\text{ where } b, d\in\mathbf{A})
\end{cases}
\]
Choose $m\geq 1$ where $A$ is a $\{0,\dots,m\}$-word; see for example Definition \ref{definition:finitely-generated-words-II}. Note that unless $\mathcal{B}$ and $\mathcal{D}$ are both trivial, we have $H(\mathcal{C}^{-1}[d])=H(\mathcal{C})^{-1}$, which is not a $\mathbb{Z}$-word. Otherwise $H(\mathcal{C})=A(\swarrow\searrow)$  which could be a $\mathbb{Z}$-word, and hence $H(\mathcal{C}^{-1}[-d])=(H(\mathcal{C})^{-1})[m]$ by Lemma \ref{lemma:technical-inverting-sarrow-extension-words}.
\end{proof}
\subsection{A one-to-one correspondence for strings.}
\begin{lemma}\label{lemma:resolution-of-homology-is-the-identity}
Let $\mathcal{C}$ be a string resolution with interval $[\iota(-),\iota(+)]$. Then $\mathcal{R}_{H(\mathcal{C})}=\mathcal{C}$ and consequently $H^{t}(P(\mathcal{C}))\simeq M(H(\mathcal{C}))$ where $t=\mathscr{H}_{\mathcal{C}}(\iota(\pm))$.
\end{lemma}
\begin{proof}
Let $(\mathcal{B},\mathcal{A},\mathcal{D})$ be a decomposition of $\mathcal{C}$. Without loss of generality assume $\mathcal{A}$ and $\mathcal{B}$ are non-trivial, but that $\mathcal{D}$ is trivial. Let $\mathcal{A}=\langle \mu_{1}\rangle \langle \eta_{1}\rangle ^{-1}\dots \langle \mu_{n}\rangle \langle \eta_{n}\rangle ^{-1}$ and $\mathcal{B}_{1}=\langle b\sigma\rangle$ where $ b\in\mathbf{A}$ and $\sigma$ is a (possibly trivial) path. In this scenario we have by definition that if $C$ is the homology word associated to $\mathcal{C}$ then $C=H(\mathcal{C})=(\sigma A)(\searrow)=B^{-1} A D$ where $A=\mu_{1}^{-1}\eta_{1}\dots \mu_{n}^{-1}\eta_{n}$, $B=\sigma^{-1}$ is trivial or inverse and $D$ is a (trivial or inverse) word where $\eta_{n}D$ is a word, and such that $D d^{-1}$ is not a word for all $ d\in\mathbf{A}$. Suppose that $B$ is a $\{0,\dots,d\}$-word and $A$ is a $\{0,\dots,m\}$-word. By construction each of  $B=(C_{\leq d})^{-1}$ and $D=C_{>m+d}$ are (trivial or inverse), and $A=(C_{>d})_{\leq m}$ is alternating.

Furthermore, since $ b\sigma\in\mathbf{P}$ we have that $B b^{-1}$ is a word. Let $\mathcal{L}=\mathcal{B}_{>1}$. By Definition \ref{definition:left-right-upwards-extensions-of-generalised-words}, since $\mathcal{C}=\mathcal{L}^{-1}\mathcal{E}$ is a generalised word where $\mathcal{E}=\langle b\sigma\rangle^{-1}\mathcal{A}$, we have that $\mathcal{L}^{-1}\in\mathcal{GW}_{\mathcal{E}}(\nwarrow)$. By Definition \ref{definition:alternating-words-string-resolutions} (2a) we have that, for all appropriate $i$, $(\mathcal{L}^{-1})_{i}=\langle b_{i}\rangle$ for some $ b_{i}\in\mathbf{A}$. By Definition \ref{definition:alternating-words-string-resolutions} (2b) we have that $\mathcal{L}^{-1}\langle a\rangle$ is not a generalised word for all $ a\in\mathbf{A}$. By Definition \ref{definition:left-right-upwards-extensions-of-generalised-words} this means $\mathcal{C}=\mathcal{E}(\nwarrow)$. Recalling that $D d^{-1}$ is never a word, by Definition \ref{definition:string-res-of-a-word} this together gives $\mathcal{C}=\mathcal{R}_{C}$. The required isomorphism follows by Lemma \ref{lemma:homology-of-resolution-gives-iso-from-string}.
\end{proof}
\begin{lemma}\label{lemma:homology-of-resolution-is-the-identity}
If $C$ is a string word then $H(\mathcal{R}_{C})=C$. 
\end{lemma}
\begin{proof}
Let $(B,A,D)$ be the decomposition of $C$. Let $\mathcal{C}=\mathcal{R}_{C}$. By Lemma \ref{lemma:homology-of-resolution-gives-iso-from-string} $\mathcal{C}$ is a string resolution, and we write $[\iota(-),\iota(+)]$ for the interval of $\mathcal{C}$. If we have both that ($B b^{-1}$ is not a word for all $ b\in\mathbf{A}$) and that ($D d^{-1}$ is not a word for all $ d\in\mathbf{A}$), then by definition $C=A(\swarrow\searrow)=H(\mathcal{C})$. Hence, by symmetry, we can assume $D d^{-1}$ is a word for some $ d\in\mathbf{A}$. Without loss of generality assume that $A=\gamma_{1}^{-1}\sigma_{1}\dots\gamma_{n}^{-1}\sigma_{n}$ for some $\gamma_{i},\sigma_{i}\in\mathbf{P}$.  Let $\mathcal{A}=\langle\gamma_{1}\rangle\langle\sigma_{1}\rangle^{-1}\dots\langle\gamma_{n}\rangle\langle\sigma_{n}\rangle^{-1}$. 

Consider the case where $B b^{-1}$ is not a word for all $ b\in\mathbf{A}$. By construction, in this setting we have that $\mathcal{C}=(\mathcal{A}\langle  d D^{-1}\rangle )(\nearrow)$. Clearly $[0,2n]$ is the interval of $\mathcal{C}$ since $\mathcal{A}=\mathcal{C}_{1}\dots\mathcal{C}_{2n}$. Let $\mathcal{B}=(\mathcal{C}_{\leq0})^{-1}$, which is trivial, and let $\mathcal{D}=\mathcal{C}_{>2n}$. In particular $\mathcal{D}_{1}= d\mu$ for some (possibly trivial) path $\mu$ such that $ d\mu\in\mathbf{P}$ and $D=\mu^{-1}$. By Definition \ref{definition:homology-word-for-strings} this means $H(\mathcal{C})=(A\mu^{-1})(\swarrow)$. Since $B^{-1}A$ is a word we have $B\in W_{A\mu^{-1}}(\swarrow)$, and since $B b^{-1}$ is not word for all $ b\in\mathbf{A}$ we have $(A\mu^{-1})(\swarrow)=B^{-1}A\mu^{-1}=C$.

Consider instead the case where $B b^{-1}$ is a word for $ b\in\mathbf{A}$. Then $\mathcal{C}=\left(\mathcal{L}^{-1}\langle b B^{-1}\rangle^{-1}\mathcal{A}\langle  d D^{-1}\rangle \mathcal{N}\right)[-1]$ where $\mathcal{L}_{i}=\langle a_{i}\rangle$ (respectively, $\mathcal{N}_{i}=\langle a_{i}\rangle$) with $ a_{i}\in\mathbf{A}$ for all appropriate $i>1$. Letting $\mathcal{B}=\langle b B^{-1}\rangle\mathcal{L}$ and $\mathcal{D}=\langle d D^{-1}\rangle\mathcal{N}$ we have that $(\mathcal{B},\mathcal{A},\mathcal{D})$ is the decomposition of $\mathcal{C}$, that $\mathcal{B}\langle b\rangle$ is not a generalised word for all $ b\in\mathbf{A}$ and that $\mathcal{D}\langle d\rangle$ is not a generalised word for all $ d\in\mathbf{A}$. Thus in this setting we have $H(\mathcal{C})=B^{-1}(AD)$ where $B=\mu$ and $D=\sigma^{-1}$. Hence in any case $H(\mathcal{C})=C$.
\end{proof}
\begin{corollary}\label{corollary:bijection-for-strings}There is a bijection between the set of string words $C$ and the set of string resolutions $\mathcal{C}$\emph{:} given explicitly by $C\mapsto \mathcal{R}_{C}$ and  $\mathcal{C}\to H(\mathcal{C})$. Furthermore if $C\leftrightarrow \mathcal{C}$ then $P(\mathcal{C})[\mathscr{H}_{\mathcal{C}}(\iota(\pm))]$ is a resolution of $M(C)$ where $[\iota(-),\iota(+)]$ is the interval of $\mathcal{C}$. 
\end{corollary}
\begin{proof}
By Lemma \ref{lemma:homology-of-resolution-is-the-identity} we have $H(\mathcal{R}_{C})=C$, and by Lemma \ref{lemma:resolution-of-homology-is-the-identity} we have $\mathcal{R}_{H(\mathcal{C})}=\mathcal{C}$. Hence these assignments are mutually inverse. Now suppose $C\leftrightarrow\mathcal{C}$. By Lemma \ref{lemma:resolution-of-homology-is-the-identity} we have $H^{t}(P(\mathcal{C}))\simeq M(C)$ where $t=\mathscr{H}_{\mathcal{C}}(\iota(\pm))$. By Lemma \ref{lemma:string-complex-given-by-string-resolution-is-a-resolution} we have that $P(\mathcal{C})$ is a point-wise-finite resolution in degree $t$. 
\end{proof}
\subsection{The resolution, homology and a one-to-one correspondence for bands.}
\begin{definition}\label{definition:band-res-of-a-word}
Let $C={}^{\infty\hspace{-0.5ex}}A^{\infty}$ for some cyclic alternating word $A$. Now let
\[\mathcal{A}=\langle\gamma_{1}\rangle\langle\sigma_{1}\rangle^{-1}\dots\langle\gamma_{n}\rangle\langle\sigma_{n}\rangle^{-1}\text{ and }\mathcal{R}_{C}={}^{\infty\hspace{-0.5ex}}\mathcal{A}^{\infty}\text{ where }A=\gamma_{1}^{-1}\sigma_{1}\dots\gamma_{n}^{-1}\sigma_{n},\, (\gamma_{i},\sigma_{i}\in\mathbf{P}).
\]
Clearly $\mathcal{R}_{C}$ is a band-resolution in the sense of Definition \ref{definition:alternating-words-band-resolutions}
\end{definition}
\begin{lemma}\label{lemma:homology-of-resolution-gives-iso-from-band} Let $V$ be an object of \emph{$R[T,T^{-1}]\text{-\textbf{Mod}}_{R\text{-\textbf{Proj}}}$} and $C$ be a $p$-periodic non-primitive $\mathbb{Z}$-word. Then there is a $\Lambda$-module isomorphism $H^{0}(P(\mathcal{R}_{C},V))\simeq M(C,V)$.
\end{lemma}
\begin{proof}
We use the notation of Definition \ref{definition:band-res-of-a-word}. Let $P=P(\mathcal{R}_{C},V)$ and $S=P(\mathcal{R}_{C})$. In this setting we have $p=2n$, and that $P^{0}=S^{0}\otimes_{R[T,T^{-1}]}V$ where $S^{0}=\bigoplus_{i\in\mathbb{Z}}\Lambda  g_{\mathcal{C},2i}$. Furthermore, since $\mathcal{A}$ is alternating, we have $\mathrm{ker}(d^{0}_{P})=P^{0}$. Recall $M(C,V)=M(C)\otimes_{R[T,T^{-1}]}v$ where $M(C)=N(C)/L(C)$. Here $N(C)=\bigoplus_{i\in\mathbb{Z}}\Lambda  g_{C,i}$ and $L(C)$ is generated by elements of the form $\gamma_{i+1} g_{C,i}-\sigma_{i+1} g_{C,i+1}$ where $i$ runs through $\mathbb{Z}$. 

Since $v_{C}(i)=v_{\mathcal{C}}(2i)$ for all $i$, the assignment $g_{C,i}\mapsto (-1)^{i} g_{\mathcal{C},2i}$ defines an isomorphism $\varphi\colon N(C)\to S^{0}$ of $\Lambda$-modules. Let $H=H^{0}(P)$ and $G=\mathrm{im}(d^{-1}_{P})$, and so $H=P^{0}/G$. By construction we have $\varphi(\gamma_{i+1} g_{C,i}-\sigma_{i+1} g_{C,i+1})= g_{\mathcal{C},2i}^{-}+ g_{\mathcal{C},2i+1}^{+}$, and so $\varphi$ induces a $\Lambda$-module homomorphism $\psi\colon M(C)\times V\to H$ given by $\psi( g_{C,i}+L(C),v)=(-1)^{i} g_{\mathcal{C},2i}\otimes v+G$. Since $n$ is even we have $\psi( g_{C,i-n}+L(C),v)=(-1)^{i} g_{\mathcal{C},2i}\otimes Tv+G$ and so $\psi$ is $R[T,T^{-1}]$-balanced. The universal property of the tensor product (of modules over $R[T,T^{-1}]$) now induces a $\Lambda$-module homomorphism $\xi\colon M(C,V)\to H$ defined by $\psi$ on pure tensors. Dually, one can define a homomorphism $H\to M(C,V)$ by $ g_{\mathcal{C},2i}\otimes v+G\mapsto((-1)^{i} g_{C,i}+G)\otimes v$, which is the inverse of $\xi$.
\end{proof}
\begin{lemma}\label{lemma:R-C-under-shits-and-inverses} Let $C$ and $E$ be band words. If  $E=C^{\pm1}[n]$ then $\mathcal{R}_{E}=\mathcal{C}^{\pm1}[2m]$ for some $m\in\mathbb{Z}$ where $\mathcal{C}=\mathcal{R}_{C}$ and $\mathscr{H}_{\mathcal{C}}(2m)=0$.
\end{lemma}
\begin{proof}
The proof follows from a straightforward application of Lemma \ref{lemma:shifting-alternating-periods-is-alternating}. 
\end{proof}
\begin{definition}\label{definition:homology-word-for-bands}
Let $\mathcal{C}$ be a generalised word which is a band-resolution, hence $\mathcal{C}={}^{\infty\hspace{-0.5ex}}\mathcal{A}^{\infty}$ for some cyclic alternating generalised word $\mathcal{A}$. Now let 
\[A=\gamma_{1}^{-1}\sigma_{1}\dots\gamma_{n}^{-1}\sigma_{n}\text{ and }H(\mathcal{C})={}^{\infty\hspace{-0.5ex}}A^{\infty}\text{ where }\mathcal{A}=\langle\gamma_{1}\rangle\langle\sigma_{1}\rangle^{-1}\dots\langle\gamma_{n}\rangle\langle\sigma_{n}\rangle^{-1},\, (\gamma_{i},\sigma_{i}\in\mathbf{P}).
\]
\end{definition}
\begin{lemma}\label{lemma:homology-word-of-a-shift-or-inverse-bands}
Let $t\in\mathbb{Z}$ and $\mathcal{C}$ be a band-resolution. Then the following statements hold.
\begin{enumerate}
    \item The generalised word $\mathcal{C}[2t]$ is a band-resolution and $H(\mathcal{C}[2t])=H(\mathcal{C})[d]$ for some $d\in\mathbb{Z}$.
    \item The generalised word $\mathcal{C}^{-1}$ is a band-resolution and $H(\mathcal{C}^{-1})=H(\mathcal{C})^{-1}$.
\end{enumerate}
\end{lemma}
\begin{proof}
We use the notation from Definition \ref{definition:homology-word-for-bands}, and so $\mathcal{C}={}^{\infty\hspace{-0.5ex}}\mathcal{A}^{\infty}$ where $\mathcal{A}=\langle\gamma_{1}\rangle\langle\sigma_{1}\rangle^{-1}\dots\langle\gamma_{n}\rangle\langle\sigma_{n}\rangle^{-1}$, which means $H(\mathcal{C})={}^{\infty\hspace{-0.5ex}}A^{\infty}$ where $A=\gamma_{1}^{-1}\sigma_{1}\dots\gamma_{n}^{-1}\sigma_{n}$. 

(1) Let $t=l+mn$ for some $m\in\mathbb{Z}$ and $l\in\{0,\dots,n-1\}$. Since $\mathcal{A}$ is a generalised $\{0,\dots,2n\}$-word, $\mathcal{C}$ is $2n$-periodic, and so $\mathcal{C}[2t]=\mathcal{C}[2l]$. It suffices to assume $l\neq 0$. By definition $\mathcal{C}[2l]={}^{\infty\hspace{-0.5ex}}\mathcal{A}'^{\infty}$ where
\[
\mathcal{A}'=\langle\gamma_{l+1}\rangle\langle\sigma_{l+1}\rangle^{-1}\dots \langle\gamma_{n}\rangle\langle\sigma_{n}\rangle^{-1}\langle\gamma_{1}\rangle\langle\sigma_{1}\rangle^{-1}\dots\langle\gamma_{l}\rangle\langle\sigma_{l}\rangle^{-1}.
\]
By definition this means $H(\mathcal{C}[2l])={}^{\infty\hspace{-0.5ex}}E^{\infty}$ where $E=\gamma_{l+1}^{-1}\sigma_{l+1}\dots \gamma_{n}^{-1}\sigma_{n}\gamma_{1}^{-1}\sigma_{1}\dots\gamma_{l}^{-1}\sigma_{l}$. It follows that $H(\mathcal{C}[2t])=H(\mathcal{C})[d]$ where $d$ is the sum of the lengths of $\gamma_{i}$ and $\sigma_{i}$ as $i$ runs from $1$ to $l$.

(2) The proof here is similar to the proof of (1), but where $E$ is replaced with $A'=\sigma_{n}^{-1}\gamma_{n}\dots\sigma_{1}^{-1}\gamma_{1}$. 
\end{proof}
\begin{corollary}\label{corollary:bijection-for-bands}There is a bijection between the set of \emph{(}words $C={}^{\infty\hspace{-0.5ex}}A^{\infty}$ where $A$ is cyclic and alternating\emph{)} and the set of \emph{(}generalised words $\mathcal{C}$ which are band-resolutions\emph{):} given explicitly by $C\mapsto \mathcal{R}_{C}$ and  $\mathcal{C}\to H(\mathcal{C})$. Furthermore if $C\leftrightarrow \mathcal{C}$ and $V$ lies in \emph{$R[T,T^{-1}]\text{-\textbf{Mod}}_{R\text{-\textbf{Proj}}}$} then $P(\mathcal{C},V)$ is a resolution of $M(C,V)$. 
\end{corollary}
\begin{proof}
It is straightforward to check that, by Definitions \ref{definition:band-res-of-a-word} and \ref{definition:homology-word-for-bands}, we have $H(\mathcal{R}_{C})=C$ for any such $C$ (as in the statement of the lemma); and similarly we have $\mathcal{R}_{H(\mathcal{C})}=\mathcal{C}$ for any such $\mathcal{C}$. Hence these assignments are mutually inverse. Now suppose $C\leftrightarrow\mathcal{C}$. By Lemma \ref{lemma:homology-of-resolution-gives-iso-from-band} we have $H^{0}(P(\mathcal{C},V))\simeq M(C,V)$. By Lemma \ref{lemma:band-complex-given-by-band-resolution-is-a-resolution} we have that $P(\mathcal{C},V)$ is a point-wise-finite resolution in degree $0$. 
\end{proof}
\begin{example}
We continue with Example \ref{example:band-complexes-part-1-for-2-by-2-p-adics}, where $\Lambda \simeq \widehat{\mathbb{Z}}_{p}Q/\langle a^{2}, b^{2}, a b+ b a-p\rangle$ and $\mathcal{C}={}^{\infty\hspace{-0.5ex}}\mathcal{A}^{\infty}$ where $\mathcal{A}=\langle a b a\rangle\langle b a\rangle^{-1}\langle b a b\rangle\langle a\rangle^{-1}\langle b\rangle\langle a b\rangle^{-1}$. Here $H(\mathcal{C})$ is the band word $C$ from Example \ref{example:intro-band-module}, and so $P(\mathcal{C},V)$ is a resolution of $M(C,V)$ by Corollary \ref{corollary:bijection-for-bands}.
\end{example}
\section{Proofs of the main theorems}\label{section:proofs-of-main}
Theorems \ref{theorem:main-fin-gen-modules-are-sums-of-strings-and-bands}, \ref{theorem:main-isomorphisms-between-string-and-band-modules} and \ref{theorem:main-krull-remak-schmidt-property} together describe the objects in $\Lambda\text{-}\boldsymbol{\mathrm{mod}}$. For the proof we use a description of the objects in $\mathcal{K}(\Lambda\text{-}\boldsymbol{\mathrm{proj}})$ from \cite{Ben2016}; see Theorems \ref{theorem:from-thesis-I}, \ref{theorem:from-thesis-II} and \ref{theorem:from-thesis-III} below. 
\begin{definition}
\cite[Definition 3.12]{Ben2016} A \emph{string complex} has the form $P(\mathcal{C})$ where $\mathcal{C}$ is an aperiodic generalised word. If $V$ is an $R[T,T^{-1}]$-module we call $P(\mathcal{C},V)$ a \emph{band complex} provided $\mathcal{C}$ is a periodic generalised $\mathbb{Z}$-word and $V$ is an indecomposable
$R[T,T^{-1}]$-module.
\end{definition}
\begin{theorem}\emph{\cite[Theorem 1.1]{Ben2016}}\label{theorem:from-thesis-I} The following statements hold.
\begin{enumerate}
\item Every object in $\mathcal{K}(\Lambda\text{-}\boldsymbol{\mathrm{proj}})$ is isomorphic to a (possibly infinite) coproduct of shifts of string complexes $P(\mathcal{C})$ and shifts of band complexes $P(\mathcal{C},V)$.
\item Each shift of a string or band complex is indecomposable in $\mathcal{K}(\Lambda\text{-}\boldsymbol{\mathrm{Proj}})$.
\end{enumerate}
\end{theorem}
\begin{theorem}
\emph{\cite[Theorem 1.2]{Ben2016}}\label{theorem:from-thesis-II} Let $\mathcal{C}$ be a generalised $I$-word, let $\mathcal{E}$ be a generalised $J$-word, let $V,W$ be objects of  $R[T,T^{-1}]\text{-}\boldsymbol{\mathrm{Mod}}_{R\text{-}\boldsymbol{\mathrm{Proj}}}$ and let $n\in\mathbb{Z}$. The following statements hold.
\begin{enumerate}
\item If $\mathcal{C},\mathcal{E}$ are aperiodic then $P(\mathcal{C})[n]\simeq P(\mathcal{E})$ \emph{(}in $\mathcal{K}(\Lambda\text{-}\boldsymbol{\mathrm{Proj}})$\emph{)} if and only if:
\begin{enumerate}
\item $I=\{0,\dots,m\}=J$ and \emph{((}$\mathcal{E}=\mathcal{C}$ and $n=0$\emph{)} or \emph{(}$\mathcal{E}=\mathcal{C}^{-1}$ and $n=\mathscr{H}_{\mathcal{C}}(m)$\emph{));} or
\item $I=\pm\mathbb{N}$, $J=\mp\mathbb{N}$, $n=0$ and $\mathcal{E}=\mathcal{C}^{\pm1}$\emph{;} or
\item $I=\mathbb{Z}=J$, $\mathcal{E}=\mathcal{C}^{\pm1}[t]$ and $n=\mathscr{H}_{\mathcal{C}}(\pm t)$ for some $t\in\mathbb{Z}$.
\end{enumerate}
\item If $\mathcal{C},\mathcal{E}$ are periodic then $P(\mathcal{C},V)[n]\simeq P(\mathcal{E},W)$ if and only if:
\begin{enumerate}
\item $\mathcal{E}=\mathcal{C}[t]$, $k\otimes_{R}V\simeq k\otimes_{R}W$ and $n=\mathscr{H}_{\mathcal{C}}(t)$ for some $t\in\mathbb{Z}$; or
\item $\mathcal{E}=\mathcal{C}^{-1}[t]$, $k\otimes_{R}V\simeq k\otimes_{R}\mathrm{res} \,W$ and $n=\mathscr{H}_{\mathcal{C}}(-t)$ for some $t\in\mathbb{Z}$.
\end{enumerate}
\item If $\mathcal{C}$ is aperiodic and $\mathcal{E}$ is periodic, then $P(\mathcal{C})[n]\not\simeq P(\mathcal{E},V)$.
\end{enumerate}
\end{theorem}
\begin{theorem}
\label{theorem:from-thesis-III}\emph{\cite[Theorem 1.3]{Ben2016}} If two direct sums of shifts of string
and band complexes are isomorphic in $\mathcal{K}(\Lambda\text{-}\boldsymbol{\mathrm{proj}})$, then there is an isoclass
preserving bijection between the summands.
\end{theorem}
By Remark \ref{remark:complete-gentle-semiperfect} any two projective resolutions of an object in $\Lambda\text{-}\boldsymbol{\mathrm{mod}}$ are homotopy equivalent. 
\begin{proof}[Proof of Theorem \ref{theorem:main-fin-gen-modules-are-sums-of-strings-and-bands}.]
(1) Let $M$ be a finitely generated $\Lambda$-module. Let $P_{M}$ be a minimal resolution of $M$. Consider $P_{M}$ as an object in the homotopy category $\mathcal{K}(\Lambda\text{-}\boldsymbol{\mathrm{proj}})$. By Theorem \ref{theorem:from-thesis-I} we have 
\[
P_{M}\simeq\Bigl( \bigoplus_{\mathtt{s}\in\mathtt{S}}P(\mathcal{C}(\mathtt{s}))[n(\mathtt{s})]\Bigr)\oplus\Bigl( \bigoplus_{\mathtt{b}\in\mathtt{B}}P(\mathcal{C}(\mathtt{b}),V^{\mathtt{b}})[m(\mathtt{b})]\Bigr),
\]
where $\mathtt{S}$ and $\mathtt{B}$ are index sets, and for all $\mathtt{s}\in\mathtt{S}$ and $\mathtt{b}\in \mathtt{B}$: $n(\mathtt{s})$ and $m(\mathtt{b})$ are integers; $V^{\mathtt{b}}$ is an object in $R[T,T^{-1}]\text{-\textbf{Mod}}_{R\text{-\textbf{Proj}}}$; $\mathcal{C}(\mathtt{s})$ is an aperiodic generalised word; and $\mathcal{C}(\mathtt{s})$ is a $p_{\mathtt{b}}$-periodic generalised word.

Note that $P_{M}^{l}=0$ for all $l>0$, $H^{m}(P_{M})=0$ for all $m\neq 0$. Since $M$ is finitely generated each $P^{h}_{M}$ is finitely generated, since $P_{M}$ is minimal. Hence $P_{M}$ is a point-wise-finite resolution in degree $0$. This means that, for all $\mathtt{s}\in\mathtt{S}$ and $\mathtt{b}\in \mathtt{B}$: the complex $P(\mathcal{C}(\mathtt{s}))$ is a point-wise-finite resolution in degree $n(\mathtt{s})$; and the complex $P(\mathcal{C}(\mathtt{b}),V^{\mathtt{b}})$ is a point-wise-finite resolution in degree $m(\mathtt{b})$. 

Let $\mathtt{s}\in\mathtt{S}$ and suppose $\mathcal{C}(\mathtt{s})$ is a generalised $I$-word. By Lemma \ref{lemma:string-complex-resolution-given-by-a-certain-form-II} this means either that ($I\neq\mathbb{Z}$ and) $\mathcal{C}(\mathtt{s})$ is a string resolution with $n(\mathtt{s})=\mathscr{H}_{\mathcal{C}(\mathtt{s})}(\iota_{\mathtt{s}}(\pm))$ where $[\iota_{\mathtt{s}}(-),\iota_{\mathtt{s}}(+)]$ is the interval of $\mathcal{C}(\mathtt{s})$; or that ($I=\mathbb{Z}$ and) there exists some $l(\mathtt{s})\in\mathbb{Z}$ such that $\mathcal{C}(\mathtt{s})[l(\mathtt{s})]$ is a string resolution and $n(\mathtt{s})=\mathscr{H}_{\mathcal{C}(\mathtt{s})}(l(\mathtt{s}))$. If $I=\mathbb{Z}$ then $P(C(\mathtt{s}))[n(\mathtt{s})]\simeq P(\mathcal{C}(\mathtt{s})[l(\mathtt{s})])$ by Lemma \ref{lemma:isos-between-string-complexes}(3), and $[0,2h]$ is the interval of $\mathcal{C}(\mathtt{s})[l(\mathtt{s})]$ for some integer $h\geq 0$. In any case this gives $H^{0}(P(\mathcal{C}(\mathtt{s}))[n(\mathtt{s})])\simeq M(H(\mathcal{C}(\mathtt{s})))$ by Lemma \ref{lemma:resolution-of-homology-is-the-identity}. 

By Lemma \ref{lemma:band-complex-resolution-given-by-a-certain-form}, for each $\mathtt{b}$ we have that $V^{\mathtt{b}}$ is $R$-module finite, and that either ($\mathcal{C}(\mathtt{b})$ is a band-resolution and $m(\mathtt{b})=0$) or that ($\mathcal{C}(\mathtt{b})[-1]$ is a band-resolution and $m(\mathtt{b})=-1$). By Lemmas \ref{lemma:isos-between-string-complexes}(3) and \ref{lemma:homology-of-resolution-gives-iso-from-band} this means $H^{0}(P(\mathcal{C}(\mathtt{b})[m(\mathtt{b})],V^{\mathtt{b}}))\simeq M(H(\mathcal{C}(\mathtt{b})),V^{\mathtt{b}})$. Putting this all together gives
\[
M\simeq\Bigl( \bigoplus_{\mathtt{s}\in\mathtt{S}}M(H(\mathcal{C}(\mathtt{s})))\Bigr)\oplus\Bigl( \bigoplus_{\mathtt{b}\in\mathtt{B}}M(H(\mathcal{C}(\mathtt{b})),V^{\mathtt{b}})\Bigr)
\]
which is a direct sum of string modules and band modules.

(2) Suppose $M$ is either a string module or a band module, and assume that $M\simeq L\oplus N$. Let $P_{L}$, $P_{M}$ and $P_{N}$ be resolutions of $L$, $M$ and $N$ respectively. The isomorphism $M\simeq L\oplus N$ of $\Lambda$-modules gives an isomorphism $P_{M}\simeq P_{L}\oplus P_{N}$ in the homotopy category $\mathcal{K}(\Lambda\text{-}\boldsymbol{\mathrm{proj}})$. 

By Corollaries \ref{corollary:bijection-for-strings} and \ref{corollary:bijection-for-bands} the object $P_{M}$ of $\mathcal{K}(\Lambda\text{-}\boldsymbol{\mathrm{Proj}})$ is isomorphic to a shift of a string or band complex, which means $P_{M}$ is indecomposable by Theorem \ref{theorem:from-thesis-I}(2). Hence, without loss of generality, we have $P_{L}\simeq 0$ in $\mathcal{K}(\Lambda\text{-}\boldsymbol{\mathrm{Proj}})$. Since any homotopy equivalence is a quasi-isomorphism, $L\simeq H^{0}(P_{M})=0$.
\end{proof}
\begin{proof}[Proof of Theorem \ref{theorem:main-isomorphisms-between-string-and-band-modules}.]
(1) Let $C$ and $E$ be string words. Suppose that $M(C)\simeq M(E)$. By Corollary \ref{corollary:bijection-for-strings} we have that $P(\mathcal{C})[\mathscr{H}_{\mathcal{C}}(\iota_{C}(\pm))]$ is the resolution of $M(C)$ where $\mathcal{C}=\mathcal{R}_{C}$ and $[\iota_{C}(-),\iota_{C}(+)]$ is the interval of $\mathcal{C}$. Likewise $P(\mathcal{E})[\mathscr{H}_{\mathcal{E}}(\iota_{E}(\pm))]$ is the resolution of $M(E)$ where $\mathcal{E}=\mathcal{R}_{E}$ and $[\iota_{E}(-),\iota_{E}(+)]$ is the interval of $\mathcal{C}$. The isomorphism $M(C)\simeq M(E)$ in $\Lambda\text{-}\boldsymbol{\mathrm{mod}}$ defines an isomorphism $P(\mathcal{C})[n]\simeq P(\mathcal{E})$ in $\mathcal{K}(\Lambda\text{-}\boldsymbol{\mathrm{proj}})$ where $n=\mathscr{H}_{\mathcal{C}}(\iota_{C}(\pm))-\mathscr{H}_{\mathcal{E}}(\iota_{E}(\pm))$, by Remark \ref{remark:complete-gentle-semiperfect}. By Theorem \ref{theorem:from-thesis-II}(1) this means there exists $t\in\mathbb{Z}$ such that $\mathcal{E}=\mathcal{C}^{\pm1}[t]$. By Lemmas \ref{lemma:generalised-decompositions-under-shifts-and-inverses} and \ref{lemma:homology-word-of-a-shift-or-inverse-strings} this means $H(\mathcal{C})$ and $H(\mathcal{E})$ are equivalent in the sense of Definition \ref{definition:equivalence-relation-on-words}. By Lemma \ref{lemma:homology-of-resolution-is-the-identity} this shows that $C$ and $E$ are equivalent. Conversely, supposing $C$ and $E$ are equivalent we have $M(C)\simeq M(E)$ by Lemma \ref{lemma:equivalence-relation-on-words-strings}.

(2) Let $C$ and $E$ be band words. By Corollary \ref{corollary:bijection-for-bands} we have, for any $U$ in $R[T,T^{-1}]\text{-\textbf{Mod}}_{R\text{-\textbf{Proj}}}$, that  $P(\mathcal{C},U)$ is the resolution of $M(C,U)$ where $\mathcal{C}=\mathcal{R}_{C}$; and $P(\mathcal{E},U)$ is the resolution of $M(E,U)$ where $\mathcal{E}=\mathcal{R}_{E}$. 

Suppose $M(C,V)\simeq M(E,W)$. The isomorphism $M(C,V)\simeq M(E,W)$ in $\Lambda\text{-}\boldsymbol{\mathrm{mod}}$ defines an isomorphism $P(\mathcal{C},V)\simeq P(\mathcal{E},W)$ in $\mathcal{K}(\Lambda\text{-}\boldsymbol{\mathrm{proj}})$, by Remark \ref{remark:complete-gentle-semiperfect}. By Theorem \ref{theorem:from-thesis-II}(2) this means there exists $s\in\mathbb{Z}$ such that $\mathcal{E}=\mathcal{C}[s]$, $\mathscr{H}_{\mathcal{C}}(s)=0$ and $k\otimes_{R}V\simeq k\otimes_{R}W$ (respectively, $\mathcal{E}=\mathcal{C}^{-1}[s]$, $\mathscr{H}_{\mathcal{C}}(-s)=0$ and $k\otimes_{R}V\simeq k\otimes_{R}\mathrm{res} \,W$). Note that since $\mathscr{H}_{\mathcal{C}}(s)=0$ (respectively, $\mathscr{H}_{\mathcal{C}}(-s)=0$) we must have that $s$ is even, and so $s=2t$ for some $t\in\mathbb{Z}$. By Lemma \ref{lemma:homology-word-of-a-shift-or-inverse-bands} this means $H(\mathcal{E})=H(\mathcal{C})[-d]$ (respectively, $H(\mathcal{E})=H(\mathcal{C})^{-1}[-d]$) for some $d\in\mathbb{Z}$, and hence $E=C[-d]$ (respectively, $E=C^{-1}[-d]$) by Corollary \ref{corollary:bijection-for-bands}. 

Conversely, suppose now $E=C[n]$ and  $k\otimes_{R}V\simeq k\otimes_{R}W$. By Lemma \ref{lemma:R-C-under-shits-and-inverses} we have that $\mathcal{E}=\mathcal{C}[2m]$ for some $m\in\mathbb{Z}$ where $\mathscr{H}_{\mathcal{C}}(2m)=0$. By Theorem \ref{theorem:from-thesis-II}(2a) this means $P(\mathcal{C},V)\simeq P(\mathcal{E},W)$ and hence $M(C,V)\simeq M(E,W)$. The proof that ($E=C^{-1
}[n]$ and  $k\otimes_{R}V\simeq k\otimes_{R}\mathrm{res}\,W$) implies $M(C,V)\simeq M(E,\mathrm{res}\,W)$ is similar, but uses Theorem \ref{theorem:from-thesis-II}(2b).

(3) As above, if $M(C)\simeq M(E,V)$ where $C$ is a string word and $E$ is a band word then  $P(\mathcal{C})[n]\simeq P(\mathcal{E},V)$ in $\mathcal{K}(\Lambda\text{-}\boldsymbol{\mathrm{proj}})$ for some $n\in\mathbb{Z}$ where $\mathcal{C}=\mathcal{R}_{C}$ and $\mathcal{E}=\mathcal{R}_{E}$, which contradicts Theorem \ref{theorem:from-thesis-II}(3). 
\end{proof}
\begin{proof}[Proof of Theorem \ref{theorem:main-krull-remak-schmidt-property}.]
Let $M$ be a finitely generated $\Lambda$-module such that
\[
\Bigl( \bigoplus_{\mathtt{s}\in\mathtt{S}}M(C(\mathtt{s}))\Bigr)\oplus\Bigl( \bigoplus_{\mathtt{b}\in\mathtt{B}}M(C(\mathtt{b}),V^{\mathtt{b}})\Bigr)\simeq M\simeq\Bigl( \bigoplus_{\mathtt{t}\in\mathtt{T}}M(E(\mathtt{t}))\Bigr)\oplus\Bigl( \bigoplus_{\mathtt{c}\in\mathtt{C}}M(E(\mathtt{c}),W^{\mathtt{c}})\Bigr)
\]
where: $\mathtt{S}$, $\mathtt{B}$, $\mathtt{T}$ and $\mathtt{C}$ are index sets; for each $\mathtt{s}$ and $\mathtt{t}$ both $C(\mathtt{s})$ and $E(\mathtt{t})$ are peak-finite words which are not eventually upward; and for each $\mathtt{b}$ and $\mathtt{c}$, both  $C(\mathtt{b})$ and  $E(\mathtt{c})$ are periodic-non-primitive words, and both $V^{\mathtt{b}}$ and $W^{\mathtt{c}}$ are objects in  $R[T,T^{-1}]\text{-}\boldsymbol{\mathrm{Mod}}_{R\text{-}\boldsymbol{\mathrm{proj}}}$. Let $P_{M}$ be a projective resolution of $M$. Let $\mathcal{C}(\mathtt{s})=\mathcal{R}_{C(\mathtt{s})}$ for each $\mathtt{s}$, and likewise let $\mathcal{C}(\mathtt{b})=\mathcal{R}_{C(\mathtt{b})}$, $\mathcal{E}(\mathtt{t})=\mathcal{R}_{E(\mathtt{t})}$ and $\mathcal{E}(\mathtt{c})=\mathcal{R}_{E(\mathtt{c})}$. By Corollaries \ref{corollary:bijection-for-strings} and \ref{corollary:bijection-for-bands} there exist $n(\mathtt{s}),m(\mathtt{b}),p(\mathtt{t}),q(\mathtt{c})\in\mathbb{Z}$ such that: each $P(\mathcal{C}(\mathtt{s}))[n(\mathtt{s})]$ is a resolution of $M(C(\mathtt{s}))$; $P(\mathcal{E}(\mathtt{t}))[p(\mathtt{t})]$ is a resolution of $M(E(\mathtt{t}))$; $P(\mathcal{C}(\mathtt{b}),V^{\mathtt{b}})[m(\mathtt{b})]$ is a resolution of $M(C(\mathtt{s}),V^{\mathtt{b}})$; and $P(\mathcal{E}(\mathtt{c}),W^{\mathtt{c}})[q(\mathtt{c})]$ is a resolution of $M(E(\mathtt{c}),W^{\mathtt{c}})$. Since any isomorphism of finitely generated $\Lambda$-modules induces a homotopy equivalence between their resolutions: collecting what we have so far gives 
\[
\Bigl( \bigoplus P(\mathcal{C}(\mathtt{s}))[n(\mathtt{s})]\Bigr)\oplus\Bigl( \bigoplus P(\mathcal{C}(\mathtt{b}),V^{\mathtt{b}})[m(\mathtt{b})]\Bigr)\simeq P_{M}\simeq\Bigl( \bigoplus P(\mathcal{E}(\mathtt{t}))[p(\mathtt{t})]\Bigr)\oplus\Bigl( \bigoplus P(\mathcal{E}(\mathtt{c}),W^{\mathtt{c}})[q(\mathtt{c})]\Bigr)
\]
in  $R[T,T^{-1}]\text{-}\boldsymbol{\mathrm{Mod}}_{R\text{-}\boldsymbol{\mathrm{proj}}}$. By Theorem \ref{theorem:from-thesis-III} and Theorem \ref{theorem:from-thesis-II}(3) there are bijections $\mathfrak{u}\colon \mathtt{S}\to\mathtt{T}$ and $\mathfrak{d}\colon \mathtt{B}\to\mathtt{C}$ where $P(\mathcal{C}(\mathtt{s})[n(\mathtt{s})]\simeq P(\mathcal{E}(\mathfrak{u}(\mathtt{s})))[p(\mathfrak{u}(\mathtt{s}))]$ for each $\mathtt{s}$ and $P(\mathcal{C}(\mathtt{b}),V^{\mathtt{b}})[m(\mathtt{b})]\simeq P(\mathcal{E}(\mathfrak{d}(\mathtt{b})),W^{\mathfrak{d}(\mathtt{b})})[q(\mathfrak{d}(\mathtt{b}))]$ for each $\mathtt{b}$. Since any homotopy equivalence is a quasi-isomorphism, this means  $M(C(\mathtt{s}))\simeq M(E(\mathfrak{u}(\mathtt{s})))$ for each $\mathtt{s}$ and $M(C(\mathtt{b}),V^{\mathtt{b}})\simeq M(E(\mathfrak{d}(\mathtt{b})),W^{\mathfrak{d}(\mathtt{b})})$ for each $\mathtt{b}$. 
\end{proof}
 \begin{proof}[Proof of Theorem \ref{theorem:resolutions}.] Definitions \ref{definition:string-words}, \ref{definition:string-res-of-a-word}, \ref{definition:string-words} and \ref{definition:string-res-of-a-word}, together with Corollaries \ref{corollary:bijection-for-strings} and \ref{corollary:bijection-for-bands}, give the algorithm. The final statement of the theorem follows from observing that any band complex defined by a band resolution is concentrated in degree $-1$ and degree $0$.
 \end{proof}
 \begin{acknowledgements}
The author is grateful to be supported by the Alexander von Humboldt Foundation 
in the framework of an Alexander von Humboldt Professorship 
endowed by the German Federal Ministry of Education and Research.
\end{acknowledgements}
\bibliography{biblio}

\providecommand{\bysame}{\leavevmode\hbox to3em{\hrulefill}\thinspace}
\providecommand{\MR}{\relax\ifhmode\unskip\space\fi MR }
\providecommand{\MRhref}[2]{%
  \href{http://www.ams.org/mathscinet-getitem?mr=#1}{#2}
}
\providecommand{\href}[2]{#2}
\begin{thebibliography}{10}

\bibitem{ArnLakPau2016}
Kristin~Krogh Arnesen, Rosanna Laking, and David Pauksztello, \emph{Morphisms
  between indecomposable complexes in the bounded derived category of a gentle
  algebra}, J. Algebra \textbf{467} (2016), 1--46. \MR{3545953}

\bibitem{AssSko1986}
Ibrahim Assem and Andrzej Skowro\'{n}ski, \emph{Iterated tilted algebras of
  type {$\tilde{\bf A}_n$}}, Math. Z. \textbf{195} (1987), no.~2, 269--290.
  \MR{892057}

\bibitem{BekDro2003}
Viktor {Bekkert} and Yuriy {Drozd}, \emph{{Tame-wild dichotomy for derived
  categories}}, arXiv Mathematics e-prints (2003), math/0310352.

\bibitem{BekMer2003}
Viktor Bekkert and H\'{e}ctor Merklen, \emph{Indecomposables in derived
  categories of gentle algebras}, Algebr. Represent. Theory \textbf{6} (2003),
  no.~3, 285--302. \MR{2000963}

\bibitem{Ben2016}
Raphael {Bennett-Tennenhaus}, \emph{{Functorial filtrations for homotopy
  categories of some generalisations of gentle algebras}}, arXiv e-prints
  (2016), arXiv:1608.08514.

\bibitem{Ben2018}
\bysame, \emph{Functorial filtrations for semiperfect generalisations of gentle
  algebras}, Ph.D. thesis, University of Leeds,
  http://etheses.whiterose.ac.uk/19607/, 2017.

\bibitem{BesHol2008}
Christine Bessenrodt and Thorsten Holm, \emph{Weighted locally gentle quivers
  and {C}artan matrices}, J. Pure Appl. Algebra \textbf{212} (2008), no.~1,
  204--221. \MR{2355046}

\bibitem{Bob2011}
Grzegorz Bobi\'{n}ski, \emph{The almost split triangles for perfect complexes
  over gentle algebras}, J. Pure Appl. Algebra \textbf{215} (2011), no.~4,
  642--654. \MR{2738378}

\bibitem{ButRin1987}
Michael Butler and Claus Ringel, \emph{Auslander-{R}eiten sequences with few
  middle terms and applications to string algebras}, Comm. Algebra \textbf{15}
  (1987), no.~1-2, 145--179. \MR{876976}

\bibitem{CanPauSch2021}
Ilke \c{C}anak\c{c}{\i}, David Pauksztello, and Sibylle Schroll, \emph{On
  extensions for gentle algebras}, Canadian Journal of Mathematics \textbf{73}
  (2021), no.~1, 249–292.

\bibitem{Cra19882}
William {Crawley-Boevey}, \emph{Functorial filtrations and the problem of an
  idempotent and a square-zero matrix}, Journal of the London Mathematical
  Society \textbf{2} (1988), no.~3, 385--402.

\bibitem{Cra1989}
\bysame, \emph{Functorial filtrations. {II}. {C}lans and the {G}el`fand
  problem}, J. London Math. Soc. (2) \textbf{40} (1989), no.~1, 9--30.
  \MR{1028911}

\bibitem{Cra19892}
\bysame, \emph{Functorial filtrations. {III}. {S}emidihedral algebras}, J.
  London Math. Soc. (2) \textbf{40} (1989), no.~1, 31--39. \MR{1028912}

\bibitem{Cra2018}
\bysame, \emph{Classification of modules for infinite-dimensional string
  algebras}, Trans. Amer. Math. Soc. \textbf{370} (2018), no.~5, 3289--3313.
  \MR{3766850}

\bibitem{Ful1978}
Kent Fuller, \emph{Biserial rings}, Ring theory ({P}roc. {C}onf., {U}niv.
  {W}aterloo, {W}aterloo, 1978), Lecture Notes in Math., vol. 734, Springer,
  Berlin, 1979, pp.~64--90. \MR{548124}

\bibitem{Happ1988}
Dieter Happel, \emph{Triangulated categories in the representation theory of
  finite-dimensional algebras}, London Mathematical Society Lecture Note
  Series, vol. 119, Cambridge University Press, Cambridge, 1988. \MR{935124}

\bibitem{HuiSma2005}
Birge Huisgen-Zimmermann and Sverre Smal\o{}, \emph{The homology of string
  algebras i}, J. Reine Angew. Math. \textbf{2005} (2005), no.~580, 1--37.

\bibitem{Lam1991}
Tsit Lam, \emph{A first course in noncommutative rings}, Graduate Texts in
  Mathematics, vol. 131, Springer-Verlag, New York, 1991. \MR{1125071}

\bibitem{Ric2017}
Charlotte Ricke, \emph{On tau-tilting theory and perpendicular categories},
  Ph.D. thesis, Rheinische Friedrich-Wilhelms-Universit{\"a}t Bonn,
  http://hss.ulb.uni-bonn.de/2016/4480/4480.htm, 2017.

\bibitem{Rin1975}
Claus Ringel, \emph{The indecomposable representations of the dihedral
  {$2$}-groups}, Math. Ann. \textbf{214} (1975), 19--34. \MR{0364426}

\bibitem{Rin1995}
\bysame, \emph{Some algebraically compact modules. {I}}, Abelian groups and
  modules ({P}adova, 1994), Math. Appl., vol. 343, Kluwer Acad. Publ.,
  Dordrecht, 1995, pp.~419--439. \MR{1378216}

\bibitem{SkoWas1983}
Andrzej Skowro\'{n}ski and Josef Waschb\"{u}sch, \emph{Representation-finite
  biserial algebras}, J. Reine Angew. Math. \textbf{345} (1983), 172--181.
  \MR{717892}

\bibitem{WalWas1985}
Burkhard Wald and Josef Waschb\"{u}sch, \emph{Tame biserial algebras}, J.
  Algebra \textbf{95} (1985), no.~2, 480--500. \MR{801283}

\end{thebibliography}
\bibliographystyle{amsplain}
\end{document}